\documentclass[10pt,twoside]{amsart}
\usepackage{latexsym,amssymb,graphicx}
\usepackage{mathrsfs}
\usepackage[all]{xy}

\let\oldtocsection=\tocsection

\let\oldtocsubsection=\tocsubsection

\let\oldtocsubsubsection=\tocsubsubsection

\renewcommand{\tocsection}[2]{\hspace{0em}\oldtocsection{#1}{#2}}
\renewcommand{\tocsubsection}[2]{\hspace{1em}\oldtocsubsection{#1}{#2}}
\renewcommand{\tocsubsubsection}[2]{\hspace{2em}\oldtocsubsubsection{#1}{#2}}

\newtheorem{theorem}{Theorem}[section]
\newtheorem{lemma}[theorem]{Lemma}
\newtheorem{corollary}[theorem]{Corollary}
\newtheorem{proposition}[theorem]{Proposition}

\theoremstyle{definition}

\newtheorem{definition}[theorem]{Definition}
\newtheorem{example}[theorem]{Example}
\newtheorem{remark}[theorem]{Remark}
\newtheorem{question}[theorem]{Question}

\numberwithin{equation}{section}

\newcommand{\spec}{\mathrm{Spec\hskip .5mm }}

\newcommand{\Z}{\mathbb{Z}}

\newcommand{\B}{\mathcal{B}}
\newcommand{\C}{\mathcal{C}}

\newcommand{\D}{\mathcal{D}}

\newcommand{\F}{\mathscr{F}}

\renewcommand{\H}{\mathscr{H}}

\newcommand{\N}{\mathbb{N}}

\newcommand{\ffi}{\varphi}
\newcommand{\eps}{\varepsilon}

\newcommand{\colim}{\operatorname{colim}}
\newcommand{\Spec}{\operatorname{Spec}}

\newcommand{\Sch}{\operatorname{Sch}}

\newcommand{\im}{\operatorname{Im}}
\newcommand{\Cat}{\operatorname{Cat}}

\newcommand{\holim}{\operatorname{holim}}

\newcommand{\Hom}{\operatorname{Hom}}

\newcommand{\Vect}{\operatorname{Vect}}

\newcommand{\can}{\operatorname{can}}

\newcommand{\RR}{\mathcal{R}}

\newcommand{\pt}{\operatorname{pt}}

\renewcommand{\P}{\mathcal{P}}

\newcommand{\K}{\mathcal{K}}

\newcommand{\Zar}{{Zar}}
\newcommand{\Nis}{{Nis}}

\newcommand{\Sing}{\operatorname{Sing}^{\aaa^1}}

\newcommand{\coker}{\operatorname{coker}}

\newcommand{\aaa}{\mathbb{A}}
\newcommand{\ppp}{\mathbb{P}}

\newcommand{\Tens}{\operatorname{Tens}}

\newcommand{\V}{\mathcal{V}}

\newcommand{\sr}{\operatorname{sr}}

\newcommand{\Aff}{\operatorname{Aff}}
\newcommand{\Tor}{\operatorname{Tor}}

\newcommand{\rk}{\operatorname{rk}}

\newcommand{\Open}{\operatorname{Open}}

\newcommand{\SSS}{\mathscr{S}}

\newcommand{\U}{\mathcal{U}}
\newcommand{\sSets}{\operatorname{sSets}}
\newcommand{\Ex}{\operatorname{Ex}}

\title[Euler class groups]{Euler class groups, and the homology of elementary and special linear groups}
\author{Marco Schlichting}
\address{Marco Schlichting, Mathematics Institute,
Zeeman Building,
University of Warwick,
Coventry CV4 7AL, UK} 

\thanks{}

\email{m.schlichting@warwick.ac.uk}



\begin{document}
\bibliographystyle{alpha}

\begin{abstract}
We prove homology stability for elementary and special linear groups over rings with many units improving known stability ranges.
Our result implies stability for unstable Quillen $K$-groups and proves a conjecture of Bass.
For commutative local rings with infinite residue fields, we show that the obstruction to further stability is given by Milnor-Witt $K$-theory.
As an application we construct Euler classes of projective modules with values in the cohomology of the Milnor-Witt $K$-theory sheaf.
For $d$-dimensional commutative noetherian rings with infinite residue fields we show that the vanishing of the Euler class is necessary and sufficient for a projective module $P$ of rank $d$ to split off a rank $1$ free direct summand.
Along the way we obtain a new presentation of Milnor-Witt $K$-theory.
\end{abstract}

\maketitle

\tableofcontents

\section{Introduction}

The purpose of this paper is to improve stability ranges in homology and algebraic $K$-theory of elementary and special linear groups, and to apply these results to construct obstruction classes for projective modules to split off a free direct summand.
\vspace{1ex}

Our first result concerns a conjecture of Bass \cite[Conjecture XVI on p.\ 43]{BassConj}. 
In {\em loc.\,cit.} he conjectured that for a commutative noetherian ring $A$ whose maximal ideal spectrum has dimension $d$ the canonical maps
$$\pi_iBGL_{n-1}^+(A) \to \pi_iBGL_{n}^+(A)$$
are surjective for $n\geq d+i+1$ and bijective for $n \geq d+i+2$.
Here, for a connected space $X$, we denote by $X^+$ Quillen's plus-construction with respect to the maximal perfect subgroup of $\pi_1X$, and we write $BGL_n^+(A)$ for $BGL_n(A)^+$.
In this generality, there are counterexamples to Bass' conjecture; see \cite[\S 8]{VanDerKallenCounterEx}.
The best general positive results to date concerning the conjecture are due to van der Kallen \cite{VanDerKallen:Invent} and Suslin \cite{Suslin:KStability}.
They prove that the maps are surjective for $n-1 \geq \max(2i,\sr(A)+i-1)$ and bijective for $n-1\geq \max(2i,\sr(A)+i)$
where $\sr(A)$ denotes the stable rank of $A$ \cite{dfn:stableRank}.
Here $A$ need not be commutative nor noetherian.

In this paper we prove Bass' conjecture for rings with many units.
Recall \cite{SuslinNesterenko} that a ring $A$ (always associative with unit) has many units if for every integer $n\geq 1$ there is a family of $n$ {\em central elements} of $A$ such that the sum of each non-empty subfamily is a unit.
Examples of rings with many units are infinite fields, commutative local rings with infinite residue field and algebras over a ring with many units.
Here is our first main result.

\begin{theorem}[Theorem \ref{thm:BassConjecture}]
\label{thm:Main1}
Let $A$ be a ring with many units.
Then the natural homomorphism
$$\pi_iBGL_{n-1}^+(A) \to \pi_iBGL_n^+(A)$$
is an isomorphism for $n\geq i + \sr(A) +1$ and surjective for $n\geq i+\sr(A)$.
\end{theorem}

The ring $A$ in Theorem \ref{thm:Main1} is not assumed to be commutative.
If $A$ is commutative noetherian with maximal ideal spectrum of dimension $d$ then  $\sr(A) \leq d+1$ \cite[Theorem 11.1]{Bass:StableK}.
So, our theorem  proves Bass' conjecture in case $A$ has many units.
If $A$ is commutative local with infinite residue field then $\sr(A)=1$ and the theorem admits the following refinement which shows that the stability range in Theorem \ref{thm:Main1} is sharp in many cases.
Denote by $K^{MW}_n(A)$ the $n$-th Milnor-Witt $K$-theory of $A$ \cite[Definition 3.1]{morel:book} which makes sense for any commutative ring $A$; see Definition \ref{dfn:HopkinsMorelMilnorWittK}.

\begin{theorem}[Theorem \ref{thm:BassConjLocalRingsInText}]
\label{thm:BassConjLocalRings}
Let $A$ be a commutative local ring with infinite residue field.
Then the natural homomorphism
$$\pi_iBGL_{n-1}^+(A) \to \pi_iBGL_n^+(A)$$
is an isomorphism for $n\geq i + 2$ and surjective for $n\geq i+1$.
Moreover, there is an exact sequence for $n\geq 2$
$$\pi_nBGL_{n-1}^+(A) \to \pi_nBGL_{n}^+(A) \to K^{MW}_n(A) \to \pi_{n-1}BGL_{n-1}^+(A)
\to \pi_{n-1}BGL_{n}^+(A).$$
\end{theorem}

Theorem \ref{thm:Main1}
follows from the following homology stability result for elementary linear groups.
Recall \cite[\S 1]{Bass:StableK} that the group of elementary $r\times r$-matrices of a ring $A$ is the subgroup $E_r(A)$ of $GL_r(A)$ generated by the elementary matrices $e_{i,j}(a) = 1+ a\cdot e_ie_j^T$, $a\in A$ where $e_i\in A^r$ is the $i$-th standard column basis vector.

\begin{theorem}[Theorem \ref{thm:Estability}]
\label{thm:Main2}
Let $A$ be a ring with many units.
Then the natural homomorphism
$$H_i(E_{n-1}(A),\Z) \to H_i(E_n(A),\Z)$$
is an isomorphism for $n\geq i + \sr(A) +1$ and surjective for $n\geq i+\sr(A)$.
\end{theorem}

For a division ring $A$ with infinite center, Theorem \ref{thm:Main2} proves a conjecture of Sah \cite[2.6 Conjecture]{Sah}.
From Theorem \ref{thm:Main2} one easily deduces the following homology stability result for the special linear groups of commutative rings.

\begin{theorem}[Theorem \ref{thm:SLstability}]
\label{thm:Main3}
Let $A$ be a commutative ring with many units.
Then the natural homomorphism
$$H_i(SL_{n-1}(A),\Z) \to H_i(SL_n(A),\Z)$$
is an isomorphism for $n\geq i + \sr(A) +1$ and surjective for $n\geq i+\sr(A)$.
\end{theorem}

When $A$ is a commutative local ring with infinite residue field Theorem \ref{thm:Main3} says that $H_i(SL_n(A),SL_{n-1}(A))=0$ for $i<n$.
The following theorem gives an explicit presentation of these groups for $i=n$.

\begin{theorem}[Theorem \ref{thm:HnIsKMW}]
\label{thm:main:SLnLocalA}
Let $A$ be a commutative local ring with infinite residue field.
Then for all $n\geq 2$ we  have
$$
H_{n}(SL_n(A),SL_{n-1}(A)) \cong K^{MW}_n(A).
$$
Moreover, for $n$ even, the map $H_n(SL_n(A)) \to H_n(SL_n(A),SL_{n-1}(A))$ is surjective.
In particular, the map $H_i(SL_{n-1}(A)) \to H_i(SL_n(A))$ is an isomorphism for 
$i \leq n-2$ and surjective (bijective) for $i = n-1$ and $n$ odd ($n$ even).
\end{theorem}

Theorems \ref{thm:Main3} and \ref{thm:main:SLnLocalA} generalize a result of Hutchinson and Tao \cite{HutchinsonTao:Stability} who proved them for fields of characteristic zero, though for $n$ odd, the identification of the relative homology with Milnor-Witt $K$-theory is only implicit in their work.
Contrary to \cite{HutchinsonTao:Stability}, our proof is independent of the characteristic of the residue field, works for local rings other than fields and does not use the solution of the Milnor conjecture on quadratic forms.
In Theorem \ref{thm:KerCokerOdd} we also give explicit computations of the kernel and cokernel of the stabilization map in homology at the edge of stabilization recovering and generalizing the remaining results of \cite{HutchinsonTao:Stability}.
This, however, requires the solution of the Milnor conjecture.
\vspace{1ex}

Our proof of Theorem \ref{thm:main:SLnLocalA} uses a new presentation of the Milnor-Witt $K$-groups $K^{MW}_n(A)$ for $n\geq 2$.
Denote by $\Z[A^*]$ the group ring of the group of units $A^*$ in $A$, and $I[A^*]$
the augmentation ideal.
For $a\in A^*$ denote by $\langle a\rangle \in \Z[A^*]$ the corresponding element in the group ring, and by $[a]\in I[A^*]$ the element $\langle a\rangle -1$.
We define the graded ring $\hat{K}^{MW}(A)$ as the graded $\Z[A^*]$-algebra generated in degree $1$ by $I[A^*]$ modulo the two sided ideal generated by 
the Steinberg relations $[a][1-a]$ for all $a,1-a\in A^*$; see Definition \ref{dfn:KMWhat}.

\begin{theorem}[Theorem \ref{thm:KhatIsK}]
\label{thm:Intro:KMW}
Let $A$ be a commutative local ring.
If $A$ is not a field assume that the cardinality of its residue field is at least $4$.
Then the natural map of graded rings $\hat{K}^{MW}(A) \to K^{MW}(A)$
induces an isomorphism $\hat{K}^{MW}_n(A) \cong K^{MW}_n(A)$ for $n\geq 2$.
\end{theorem}

In particular, for a local ring $A$ with infinite residue field, the Schur multiplier $H_2(SL_2(A))$ has the pleasant presentation as the quotient of $I[A^*]\otimes_{A^*}I[A^*]$ by the Steinberg relations (Theorem \ref{thm:H2andKMW2}); compare \cite[Theorem 9.2]{moore}, \cite[Corollaire 5.11]{matsumoto}, \cite[Theorem 3.4]{vdKH2SL2}.
\vspace{1ex}

Theorems \ref{thm:Main3} and \ref{thm:main:SLnLocalA} are the $SL_n$-analogs of
a result of Nesterenko and Suslin \cite{SuslinNesterenko}.
They proved that
Theorems \ref{thm:Main3} and \ref{thm:main:SLnLocalA} hold when $SL_n(A)$ and $K^{MW}(A)$ are replaced with $GL_n(A)$ and Milnor K-theory $K^M(A)$.
Suslin and Nesterenko's proof rests on the computation of the homology of affine groups \cite[Theorem 1.11]{SuslinNesterenko} which is false if one simply replaces $GL_n(A)$ with $SL_n(A)$.
Our innovation is the correct replacement of \cite[Theorem 1.11]{SuslinNesterenko} in the context of $SL_n(A)$ and of groups related to $E_n(A)$.
This is done in Section \ref{Sec:HofAff} whose main result is Theorem \ref{thm:HomologyOfAffineGps} and its Corollary \ref{cor:HomologyOfAffineGps}.
With our new presentation of Milnor-Witt $K$-theory in Section \ref{sec:KMW}, 
Sections \ref{sec:Stability} and \ref{sec:Obstruction} more or less follow the treatment in \cite{SuslinNesterenko}.  
\vspace{1ex}

The importance of homology stability and the computation of the obstruction to further  stability in Theorem \ref{thm:main:SLnLocalA} lies in the following application.
Let $R$ be a commutative noetherian ring of dimension $n$
all of whose residue fields are infinite.
Let $P$ be an oriented  rank $n$ projective $R$-module.
In \S \ref{sec:EulerClGps}, we define a class
$$e(P)\in H^n_{\Zar}(R,\K^{MW}_n)$$
such that $e(P)=0$ if $P$ splits off a free direct summand of rank $1$.
Here, $\K^{MW}_n$ denotes the Zariski sheaf associated with the presheaf $A \mapsto K^{MW}_n(A)$, and $H^n_{\Zar}$ denotes Zariski cohomology.
We prove the following.

\begin{theorem}[Theorem \ref{thm:GeneralEulerclass}]
\label{thm:intro:Eulerclass}
Let $R$ be a commutative noetherian ring of dimension $n\geq 2$.
Assume that all residue fields of $R$ are infinite.
Let $P$ be an oriented rank $n$ projective $R$-module.
Then 
$$P\cong Q\oplus R \Leftrightarrow e(P)=0 \in H^n_{\Zar}(R,\K^{MW}_n).$$
\end{theorem}

If $R$ has dimension $n$ and is of finite type over an algebraically closed field $k$, then the canonical map $\K^{MW}_n\to \K^M$ of sheaves on $X$ is an isomorphism. 
In particular, if $R$ has dimension $n$ and is smooth over an algebraically closed field, then $H^n_{\Zar}(R,\K^{MW}_n) = H^n_{\Zar}(R,\K^{M}_n)$ is isomorphic to the Chow group of codimension $n$ cycles on $X=\Spec R$, by \cite[Theorem 7.5]{Kerz}, and we recover a result of Murthy \cite{MurthyChernClass}.
If $R$ is smooth over a field of characteristic not $2$ (which is not assumed algebraically closed) then $H^n_{\Zar}(R,\K^{MW}_n)$ is isomorphic to the Chow-Witt groups introduced by Barge and Morel \cite{BargeMorel} and studied by Fasel \cite{Fasel:ChowWitt}.

Theorem \ref{thm:intro:Eulerclass} is a generalization of a theorem of Morel \cite[Theorem 8.14]{morel:book} who proved the result for $R$ smooth of finite type over a perfect field\footnote{Currently, the proof of \cite[Theorem 8.14]{morel:book} is only documented in the literature for infinite perfect fields; see 
\cite[Footnote on p. 5]{morel:book}}.
Our arguments don't use $\aaa^1$-homotopy theory but they can  be used to simplify some proofs in \cite{morel:book}; see proof of Theorem \ref{thm:MorelReps}.
There is also a definition of Euler class groups in terms of generators and relations for which one can prove a result similar to Theorem \ref{thm:intro:Eulerclass} in case $R$ is smooth over an infinite perfect field \cite{Bhatwadekar:Smooth} or in case $R$ contains the rational numbers \cite{Bhatwadekar:ContainsQ}.
The definitions in {\it loc.cit.} are not cohomological in nature, and the relationship with Theorem \ref{thm:intro:Eulerclass} is unclear.

The proof of Theorem \ref{thm:intro:Eulerclass} relies on Theorem \ref{thm:main:SLnLocalA} and a representability result of vector bundles on noetherian affine schemes (Theorem \ref{thm:SingBGRepsVect}) which is of independent interest. 
There is also a version (Theorem \ref{thm:GeneralEulerclassOrientL}) of Theorem \ref{thm:intro:Eulerclass} for projective modules with orientation in a line bundle other than $R$.
\vspace{1ex}

{\bf Conventions}.
All rings are associative with unit.
The group of units of a ring $A$ is denoted $A^*$.
The stable rank \cite{dfn:stableRank} of a ring $A$ is denoted $\sr (A)$.
By ``space'' we mean ``simplicial set''.
Unless otherwise stated, tensor products are over $\Z$ and homology has coefficients in $\Z$.
For a commutative ring $A$ and integer $n\geq 1$, the group $SL_n(A)$ is the group of $n\times n$ matrices with entries in $A$ and determinant $1$.
The symbol $SL_0(A)$ will stand for the discrete set (or discrete groupoid) $A^*$.
This has the effect that for all $n\geq 0$ and any $GL_n(A)$-module $M$, we have
$H_i(SL_n(A),M)=\Tor_i^{GL_n(A)}(\Z[A^*],M)$ where $\Z[A^*]$ is a $GL_n(A)$-module via the determinant map $GL_n(A)\to A^*$.
Moreover, for all $n\geq 0$, we have a homotopy fibration of classifying spaces $BSL_n(A) \to BGL_n(A) \to BA^*$.

We denote by $\sSets$ the category of simplicial sets endowed with its standard Kan model structure.
\vspace{1ex}

{\bf Acknowledgements}.
Most of the results of this article were found in Spring 2014 while the author was visiting Max-Planck-Institute for Mathematics in Bonn.
I would like to thank MPIM Bonn for its hospitality.
I also would like to thank Marc Hoyois whose comments lead to the appendix, 
and Wataru Kai for pointing out a mistake in a previous version in the proof of the centrality of the element $[-1,1]$.

\section{The homology of affine groups}
\label{Sec:HofAff}

For a group $G$, we denote by $\Z[G]$ its integral group ring
and we write $\langle g\rangle$ for the element in $\Z[G]$ corresponding to $g\in G$.
Furthermore, we denote by $\eps:\Z[G] \to \Z:\langle g\rangle \mapsto 1$ the augmentation ring homomorphism, and by $I[G]=\ker(\eps)$ its kernel, the augmentation ideal.
Let $G$ be an abelian group and $s\in \Z[G]$ an element in its group ring.
A $G$-module $M$ is called $s$-torsion if for every $x\in M$ there is $n \in \N$ such that $s^nx=0$, or equivalently if $[s^{-1}]M=0$.
The category of $s$-torsion $G$-modules is closed under taking subobjects, quotient objects and extensions in the category of all $G$-modules.

Many of our computations concern the homology $H_i(G,M)$ of a group $G$ with coefficients in a $G$-module $M$.
We recall the basic functoriality of this construction \cite[\S III.8]{brown:book}.
Let $G$, $G'$ be groups, and $M$, $M'$ be $G$, $G'$-modules, respectively.
A pair of maps $(\ffi,f):(G,M) \to (G',M')$ where $\ffi:G \to G'$ is a group homomorphism and $f:M \to M'$ is a homomorphism of abelian groups with $f(gx)=\ffi(g)f(x)$ for all $g\in G$ and $x\in M$ induces a map of homology groups
$(\ffi,f)_*:H_*(G,M) \to H_*(G',M')$.
Given two such pairs of maps $(\ffi_0,f_0), (\ffi_1,f_1): (G,M) \to (G',M')$.
If there is an element $h\in G'$ such that $\ffi_1(g)=h\ffi_0(g)h^{-1}$ and $f_1(x)=hf_0(x)$ for all $g\in G$ and $x\in M$, then the induced maps on homology agree:
$(\ffi_0,f_0)_* = (\ffi_1,f_1)_*:H_*(G,M) \to H_*(G',M')$.
\vspace{1ex}

For an integer $m\geq 1$, write $[m]$ for the set $\{1,...,m\}$ of integers between $1$ and $m$.
Let $\RR_m$ be the commutative ring
$$\RR_m = \Z[X_1,...,X_m][\Sigma^{-1}]$$ 
obtained by localizing the polynomial ring $\Z[X_1,...,X_m]$ in the $m$ variables $X_1,...,X_m$ 
at the set of all non-empty partial sums of the variables
$$\Sigma = \{X_J|\ \emptyset \neq J \subset [m]\},\hspace{2ex} \text{where}\ X_J=\sum_{j\in J}X_j.$$ 
A ring which admits an $\RR_m$-algebra structure for every $m$ is called a {\em ring with many units}.
For instance, a commutative local ring with infinite residue field has many units, and any algebra over a ring with many units has many units \cite[Corollary 1.3]{SuslinNesterenko}.

We will denote by $s_m\in \Z[\RR_m^*]$ the following element in the group ring of $\RR_m^*$:
$$s_m= -\sum_{\emptyset \neq J \subset [m]}(-1)^{|J|}\langle X_J\rangle.$$
Note that the augmentation homomorphism sends $s_m$ to $1$:
$$\eps(s_m) = -\sum_{\emptyset \neq J \subset [m]}(-1)^{|J|}=(-1)^{|\emptyset|} - \sum_{J \subset [m]}(-1)^{|J|}=  1 -(1-1)^m = 1.$$
More generally, for an integer $t\in \Z$ we will write $s_{m,t}$ for the image of $s_m$ under the ring homomorphism $t:\Z[\RR_m^*] \to \Z[\RR_m^*]:\langle a \rangle \mapsto \langle a^t\rangle$, that is,
$$s_{m,t}= -\sum_{\emptyset \neq J \subset [m]}(-1)^{|J|}\langle (X_J)^t\rangle.$$
For an $\RR_m$-algebra $A$ we will also write $s_m$ and $s_{m,t}$  for the images in $\Z[A^*]$ of $s_m$ and $s_{m,t}$ under the ring homomorphism $\Z[\RR_m^*] \to \Z[A^*]$.

For an integer $k\geq 1$, we denote by $V_k(A)$ the ring $(A^{\otimes k})^{\Sigma_k}$ of invariants of the natural action of the symmetric group $\Sigma_k$ on $A^{\otimes k}$ permuting the $k$ tensor factors.

\begin{lemma}
\label{lem:FollowingsaIdentityHolds}
Let $A$ be a commutative $\RR_m$-algebra and let $k,t\geq 1$ be integers with $k\cdot t < m$.
Then the ring homomorphism $\Z[A^*] \to V_k(A): \langle a\rangle \mapsto a\otimes \cdots \otimes a$ sends $s_{m,t}\in \Z[A^*]$ to $0 \in V_k(A)$.
\end{lemma}

\begin{proof}
Let $a_J$ be the image of $X_J$ under the algebra structure map $\RR_m \to A$.
For a function $\sigma: [k] \to [m]$ we write $a^{\sigma} = a_{\sigma(1)}\otimes \cdots \otimes a_{\sigma(k)}$.
Note that $a_{\emptyset}=0$ and $(a_{\emptyset})^{\otimes k}=0$.
In $V_k(A)$ we have
$$
-s_m = \sum_{J\subset [m]}(-1)^{|J|}(a_J)^{\otimes k} 
= \sum_{\stackrel{J\subset [m]}{\sigma:[k]\to J}}(-1)^{|J|}a^{\sigma}
=\sum_{\sigma:[k]\to [m]}a^{\sigma}\sum_{\im \sigma \subset J \subset [m]}(-1)^{|J|}
=0
$$
since for $I \subset [m]$ with $I\neq [m]$ we have
$$\sum_{I\subset J \subset [m]}(-1)^{|J|} = (-1)^{|I|}\sum_{J \subset [m]-I}(-1)^{|J|}
=(-1)^{|I|}(1-1)^{|[m]-I|}=0.$$
This shows that $s_m=s_{m,1}=0\in V_k(A)$, that is, the case $t=1$.
For general $t\geq 1$, the lemma follows from the case $t=1$ and the commutative diagram of rings
$$\xymatrix{\Z[A^*] \ar[r]^t \ar[d] & \Z[A^*] \ar[d] \\
V_{kt}(A) \ar[r]^{\mu^{\otimes k}} & V_k(A)
}$$
where the top horizontal map is induced by $\langle a\rangle \mapsto \langle a^t \rangle$ and 
the lower horizontal arrow is induced by the $\Sigma_k$-equivariant map $\mu^{\otimes k}: (A^{\otimes t})^{\otimes k} \to A^{\otimes k} $ where $\mu$ is the multiplication map
$\mu:A^{\otimes t} \to A:a_1\otimes \cdots \otimes a_t \mapsto a_1\cdots a_t$.
\end{proof}

For an integer $k\geq 1$ and an $A$-module $M$, consider the $k$-th exterior power $\Lambda^k_{\Z}M$ of $M$ over $\Z$; see \cite[\S V.6]{brown:book}.  
This is an $A^*$-module under the diagonal action $a\cdot (x_1\wedge \dots \wedge x_k) = ax_1\wedge \dots \wedge ax_k$ where $a\in A^*$ and $x_i\in M$.

\begin{corollary}
\label{cor:LambdaIsTorsion}
Let $M$ be an $\RR_m$-module.
Then for all integers $k,t\geq 1$  with $k\cdot t<m$ the $\RR_m^*$-module
$\Lambda^k_{\Z}M$ is $s_{m,t}$-torsion.
\end{corollary}

\begin{proof}
The abelian group $\Lambda^k_{\Z}M$ has a natural $V_k(\RR_m)$-module structure \cite[Lemma 1.7]{SuslinNesterenko}
$$V_k(\RR_m)\times \Lambda^kM \to \Lambda^kM: (a_1\otimes \cdots \otimes a_k ,\ x_1\wedge \dots \wedge x_k) \mapsto a_1x_1\wedge \dots \wedge a_kx_k$$
which induces the diagonal $\RR_m^*$-action on $\Lambda^kM$ via the ring map $\Z[\RR_m^*] \to V_k(\RR_m)$.
The result now follows from Lemma \ref{lem:FollowingsaIdentityHolds} with $A=\RR_m$.
\end{proof}

\begin{proposition}
\label{prop:HqTorsion}
Let $M$ be an $\RR_m$-module.
Then for all integers $t,q\geq 1$ with $tq <m$ the integral homology groups $H_q(M,\Z)$ of $M$ are $s_{m,t}$-torsion.
\end{proposition}

\begin{proof}
Choose a simplicial homotopy equivalence $P_* \to M$ in the category of simplicial $\RR_m$-modules such that $P_*$ is projective in each degree.
Applying the classifying space-functor degree-wise, we obtain an $\RR_m^*$-equivariant weak equivalence of simplicial sets $BP_* \to BM$ and hence $\RR_m^*$-equivariant isomorphisms $H_q(BP_*) \cong H_q(BM)$ of integral homology groups.
To the $\RR_m^*$-equivariant simplicial space $s\mapsto BP_s$ is associated a strongly convergent first quadrant spectral sequence of $\RR_m^*$-modules
$$E^1_{r,s} = H_r(BP_s, \Z) \Rightarrow H_{r+s}(BP_*, \Z)=H_{r+s}(BM,\Z)$$
where $d^1:H_r(P_s) \to H_r(P_{s-1})$ is the alternating sum of the face maps of the simplicial abelian group $s\mapsto H_r(BP_s)$.
Since the ring $\RR_m$ is flat over $\Z$, each $P_s$ is a torsion-free abelian group, and thus, the Pontryagin map $\Lambda^r_{\Z}P_s\to H_r(BP_s)$ is an isomorphism of $\RR_m^*$-modules \cite[Theorem V.6.4.(ii)]{brown:book}. 
By Corollary \ref{cor:LambdaIsTorsion} the $\RR_m^*$-module $\Lambda^r_{\Z}P_s$ is $s_{m,t}$-torsion for all $tr<m$.
Since $E^2_{0,s}=0$ for $s\geq 1$ it follows from the spectral sequence that $H_q(M)$ is $s_{m,t}$-torsion whenever $1\leq tq <m$.
\end{proof}

Let $A$ be a ring and $Z(A)$ its center.
Let $q\geq 1$ be an integer.
The inclusions 
$$GL_q(A) \subset GL_{q+1}(A) \subset GL(A):M \mapsto \left(\begin{smallmatrix}M & 0 \\ 0 & 1\end{smallmatrix}\right)$$
define group homomorphisms $\det:GL_q(A) \to GL(A)^{ab}=K_1(A)$ whose kernel we denote by 
$SG_q(A)$.
If $A$ is an $\RR_m$-algebra, we will need an action of $\RR_m^*$ on the integral homology groups of $SG_q(A)$.
For that end, let $\bar{A}^*$ be the image in $K_1(A)$ of
the map $Z(A)^* \subset GL_1(A) \to K_1(A)$.
We denote by $G_q(A)$ the subgroup of $GL_q(A)$
consisting of those matrices $T\in GL_q(A)$ whose class $\det(T)\in K_1(A)$ lies in the subgroup $\bar{A}^* \subset K_1(A)$. 
Note that $G_q(A)$ contains all invertible diagonal matrices with entries in $Z(A)$.
In particular, the map $\det:GL_q(A) \to K_1(A)$ restricts to a surjective group homomorphism $\det:G_q(A) \to \bar{A}^*$, and we have an exact sequence of groups
$$1 \to SG_q(A) \longrightarrow G_q(A) \stackrel{\det}{\longrightarrow} \bar{A}^* \to 1.$$

We will write $\Aff_{p,q}^G(A)$ and $\Aff_{p,q}^{SG}(A)$ for the following subgroups of $GL_{p+q}(A)$ 
$$\Aff_{p,q}^G = \left(\begin{smallmatrix} G_q(A) &  0\\  M_{p,q}(A) & 1_p\end{smallmatrix}\right)\hspace{3ex}\text{and}\hspace{3ex}
\Aff_{p,q}^{SG} = \left(\begin{smallmatrix} SG_q(A) &  0\\  M_{p,q}(A) & 1_p\end{smallmatrix}\right).$$
For any $M\in M_{p,q}(A)$ and $T\in GL_q(A)$, the matrices $T$ and
$$\left(\begin{smallmatrix} T & 0 \\ M& 1_p\end{smallmatrix}\right) = 
\left(\begin{smallmatrix} 1_q & 0 \\ MT^{-1}& 1_p\end{smallmatrix}\right)
\left(\begin{smallmatrix} T & 0 \\ 0& 1_p\end{smallmatrix}\right)
$$
 have the same class in $K_1(A)$. 
It follows that the map $\det:GL_{p+q}(A) \to K_1(A)$
restricts to a surjective group homomorphism
$\Aff_{p,q}^{G}(A) \to \bar{A}^*$
with kernel the group
$\Aff_{p,q}^{SG}(A)$.
Hence,
for integers $q\geq 1$, $p\geq 0$ the exact sequence of groups
$$1\to  \Aff_{p,q}^{SG}(A) \longrightarrow \Aff_{p,q}^{G}(A) \stackrel{\det}{\longrightarrow} 
\bar{A}^* \to 1$$
makes the homology groups $H_r(\Aff_{p,q}^{SG}(A))$ into $\bar{A}^*$-modules \cite[Corollary III.8.2]{brown:book}.
\vspace{1ex}

For an $\RR_m$-algebra $A$, we denote by $s_{m,t}\in \Z[\bar{A}^*]$ the image of $s_{m,t}\in \Z[\RR_m^*]$ under the ring homomorphism $\Z[R^*_m] \to \Z[\bar{A}^*]$
induced by the group homomorphism $\RR^*_m \to Z(A)^* \to \bar{A}^*$.
The following is our analog of \cite[Theorem 1.11]{SuslinNesterenko}.

\begin{theorem}
\label{thm:HomologyOfAffineGps}
Let $A$ be an $\RR_m$-algebra.
Let $t,q \geq 1$ be integers such that $q$ divides $t$.
Then for all integers $p,r\geq 0$ such that $rt<mq$
the inclusion 
$$SG_q(A) \to \Aff_{p,q}^{SG}(A): M \mapsto \left(\begin{smallmatrix} M &  0\\  0 & 1_p\end{smallmatrix}\right)$$
induces an isomorphism of $\bar{A}^*$-modules
$$H_r(SG_q(A)) \cong s_{m,-t}^{-1}\ H_r(\Aff_{p,q}^{SG}(A)).
$$
\end{theorem}

\begin{proof}
A matrix $T \in G_q(A)$ defines an automorphism of the exact sequence of groups
\begin{equation}
\label{eqn0:thm:HomologyOfAffineGps}
0 \to M_{p,q}(A) \to \left(\begin{smallmatrix} SG_q(A) &  0\\  M_{p,q}(A) & 1_p\end{smallmatrix}\right) \to SG_q(A) \to 1
\end{equation}
through right multiplication by $T^{-1}$ on $M_{p,q}(A)$, through conjugation by 
$\left(\begin{smallmatrix} T &  0 \\  0 & 1_p\end{smallmatrix}\right)$
on the middle term and through conjugation by $T$ on $SG_q(A)$.
This defines an action of the group $G_q(A)$ on the exact sequence and hence an action on the associated Hochschild-Serre spectral sequence
\begin{equation}
\label{eqn1:thm:HomologyOfAffineGps}
E^2_{i,j} = H_i(SG_qA,H_j(M_{p,q}A)) \Rightarrow H_{i+j}\left(\begin{smallmatrix} SG_q(A) &  0 \\  M_{p,q}(A) & 1_p\end{smallmatrix}\right)
\end{equation}
which descents to an $\bar{A}^*$-action via the determinant map $G_q(A) \to \bar{A}^*$, in view of the basic functoriality of group homology recalled at the beginning of this section.
Since the surjection in the exact sequence (\ref{eqn0:thm:HomologyOfAffineGps}) splits, we have $H_i(SG_q,\Z) = E^2_{i,0}=E^{\infty}_{i,0}$.

On the homology groups $H_i(SG_qA,H_j(M_{p,q}A))$, the element $T\in G_q(A)$ acts through conjugation on $SG_q(A)$ and right multiplication by $T^{-1}$ on $M_{p,q}(A)$.  
Since $q$ divides $t$, we can write $t=q\cdot k$ for some integer $k\geq 1$.
For $\bar{a}\in \bar{A}^*$, the element $\langle \bar{a}^{-t}\rangle$ acts on the spectral sequence as the diagonal matrix $T = a^{-k}\cdot 1_q \in G_q$ where $a\in Z(A)^*$ is a lift of $\bar{a}\in \bar{A}^*$.
This element acts on the pair $(SG_qA, H_j(M_{p,q}A))$ through conjugation by $a^{-k}\cdot 1_q$ on $SG_qA$ which is the identity map, and through right translation by $T^{-1}=(a^{-k}\cdot 1_q)^{-1} = a^{k}\cdot 1_q$ on $M_{p,q}A$ which is the action by $\langle a^k\rangle \in \Z[\bar{A}^*]$ induced by the usual left $Z(A)$-module structure on $M_{p,q}A$.
In view of Proposition \ref{prop:HqTorsion} it follows that for $j\geq 1$ and $kj<m$ we have
$$s_{m,-t}^{-1}\ H_i(SG_qA,H_j(M_{p,q}A)) = H_i(SG_qA,s_{m,k}^{-1}H_j(M_{p,q}A)) = 0.$$ 
Moreover, 
$$s_{m,-t}^{-1}\ H_i(SG_qA,H_0(M_{p,q}A)) = H_i(SG_qA,s_{m,k}^{-1}H_0(M_{p,q}A)) = H_i(SG_qA,\Z)$$ since $s_{m,k}$ acts through $\eps(s_{m,k})=1$ on $H_0(M_{p,q}A)=\Z$.
Localizing the spectral sequence (\ref{eqn1:thm:HomologyOfAffineGps})
at $s_{m,-t}\in \Z[\bar{A}^*]$ yields a spectral sequence
which satisfies $s_{m,-t}^{-1}\ E_{i,j}^2 = s_{m,-t}^{-1}\ E_{i,j}^{\infty}=0$ for $tj<mq$
and $s_{m,-t}^{-1}\ E_{i,0}^2 = s_{m,-t}^{-1}\ E_{i,0}^{\infty} = H_i(SG_q,\Z)$.
The claim follows.
\end{proof}

It will be convenient to reinterpret this result in somewhat different notation.
To that end, we introduce the rings $\Lambda$ and $\Lambda_{m,t}$ as
$$\Lambda = \Z[\bar{A}^*],\hspace{3ex}\Lambda_{m,t}=(s_{m,-t})^{-1}\Lambda.$$
Note that the natural maps of groups $GL_q(A) \to K_1(A)$ induce ring homomorphisms $\Z[G_q(A)] \to \Lambda \to \Lambda_{m,t}$ compatible with the inclusions $G_q(A) \subset G_{q+1}(A)$.

Recall that for bounded below complexes of right, respectively left, $G$-modules $M$, respectively $N$, the derived tensor product 
$M\stackrel{L}{\otimes}_GN$ is the complex of abelian groups $P\otimes_GQ$ where $P\to M$ and $Q\to N$ are quasi-isomorphisms of bounded below complexes of  right and left $G$-modules with $P_i$ and $Q_j$ projective right and left $G$-modules, respectively.
The derived tensor product is well-defined up to quasi-isomorphism of complexes.
The natural maps $M\otimes_GQ \leftarrow P\otimes_GQ \rightarrow P\otimes_G N$ are quasi-isomorphisms, and one has
$$\Tor_i^G(M,N)=H_i(M\stackrel{L}{\otimes}_GN),\hspace{3ex}H_i(G,N)=H_i(\Z\stackrel{L}{\otimes}_GN).$$

\begin{corollary}
\label{cor:HomologyOfAffineGps}
Let $A$ be an $\RR_m$-algebra.
Let $t,q \geq 1$ be integers such that $q$ divides $t$.
Then for all integers $p\geq 0$ 
the canonical inclusions of groups and rings
$SG_q(A)\subset G_q(A) \subset \Aff_{p,q}^G(A)$ and $\Z \subset \Lambda_{m,t}$ induce maps of complexes
$$\Z\stackrel{L}{\otimes}_{SG_q(A)}\Z {\longrightarrow}
\Lambda_{m,t}\stackrel{L}{\otimes}_{G_q(A)}\Z {\longrightarrow}
\Lambda_{m,t}\stackrel{L}{\otimes}_{\Aff_{p,q}^G(A)}\Z.
$$
which are isomorphisms on homology groups in degrees $r<mq/t$.
\end{corollary}

\begin{proof}
Recall that for a subgroup $N\subset G$ of a group, we have Shapiro's Lemma
$$\Z[N\backslash G]\stackrel{L}{\otimes}_G\Z = \Z \stackrel{L}{\otimes}_N\Z[G]\stackrel{L}{\otimes}_G\Z = \Z \stackrel{L}{\otimes}_N\Z $$
since $\Z[N\backslash G]=\Z\otimes_N\Z[G]=\Z\stackrel{L}{\otimes}_N\Z[G]$ as $\Z[G]$ is a free $N$-module.
If $N$ is normal in $G$ then $G/N=N\backslash G$ is a group and 
$\Z[G/N]\stackrel{L}{\otimes}_G\Z$ is a complex of left $G/N$-modules.
On homology, the isomorphism 
$$H_i(\Z[G/N]\stackrel{L}{\otimes}_G\Z)\cong 
H_i(\Z \stackrel{L}{\otimes}_N\Z)=H_i(N)$$
 is an isomorphism of $G/N$-modules where the action on $H_i(N)$ is the usual conjugation action.
Applied to $N=\Aff_{p,q}^{SG}$ and $G=\Aff_{p,q}^{G}$, we have an isomorphism
$$H_i(\Lambda\stackrel{L}{\otimes}_{\Aff_{p,q}^G}\Z)\cong H_i(\Aff_{p,q}^{SG})$$
of $\Lambda$-modules.
Localizing at $s_{m,-t}$ using Theorem \ref{thm:HomologyOfAffineGps}, the result follows.
\end{proof}

\section{Stability in homology and $K$-theory}
\label{sec:Stability}

Let $A$ be a ring and recall from \S \ref{Sec:HofAff} the definition of the groups $G_q(A)$ for $q\geq 1$.
We set $G_0(A)=\{1\}$, the one-element group.
Let $n\geq r\geq 0$ be integers.
We denote by $U_r(A^n) \subset M_{n,r}(A)$ the set of 
left invertible $n\times r$ matrices with entries in $A$,
and by $GU_r(A^n)\subset U_r(A^n)$ the subset of those left invertible matrices which can be completed to a matrix in $G_n(A)$.
For instance $U_0(A^n)=GU_0(A^n)=0$ is the one element set and $GU_n(A^n)= G_n(A)$.
By convention, $U_r(A^n)=GU_r(A^n)=\emptyset$ whenever $r<0$ or $r>n$.
By \cite[Lemma 2.1]{SuslinNesterenko} we have 
$U_r(A^n)=GU_r(A^n)$ for $r\leq n-\sr(A)$.

We define a complex $C(A^n)$ of abelian groups whose degree $r$ component
is the free abelian group $C_r(A^n) = \Z[GU_r(A^n)]$ generated by the 
set $GU_r(A^n)$.
For $i=1,...,r$ one has maps of abelian groups $\delta^i_r:C_r(A^n) \to C_{r-1}(A^n)$ defined on basis elements by $\delta^i_r(v_1,...,v_r) = (v_1,...,\hat{v}_i,...,v_r)$ omitting the $i$-th entry where $(v_i,...,v_r)$ is a left invertible matrix with $i$-th column the vector $v_i$.
We set 
\begin{equation}
\label{eqn:Crdifferential}
d_r=\sum_{i=1}^r(-1)^{i-1}\delta_r^i: C_r(A^n) \to C_{r-1}(A^n),
\end{equation} and it is standard that
$d_rd_{r+1}=0$.
This defines the chain complex $C(A^n)$.

\begin{lemma}
\label{lem:Acyclicity}
Let $A$ be a ring and $n \geq 0$ an integer.
Then for all $i\leq n-\sr (A)$ we have
$$H_i(C(A^n))= 0.$$
\end{lemma}

\begin{proof}
We check that the proof of \cite[Lemma 2.2]{SuslinNesterenko} goes through with $G_n(A)$ in place of $GL_n(A)$.
If we denote by $\tilde{C}(A^n)$ the complex with $\tilde{C}_q(A^n)=\Z[U_q(A^n)]$ in degree $q$ and differential given by the same formula as for ${C}(A^n)$, then we have an inclusion of complexes $C(A^n) \subset \tilde{C}(A^n)$ with 
$C_q(A^n)=\tilde{C}_q(A^n)$ for $q\leq n-\sr A$, by \cite[Lemma 2.1]{SuslinNesterenko}.
By \cite[2.6. Theorem (i)]{VanDerKallen:Invent} we have 
$H_i(\tilde{C}(A^n))= 0$ whenever $i\leq n-\sr A$.
Thus, it suffices to show that the boundary $dx$ of every $x\in U_{n-r+1}(A)$ is a boundary in ${C}(A^n)$ where $r = \sr(A)$.
By \cite[Lemma 2.1]{SuslinNesterenko}, there is a matrix $\alpha\in E_n(A)\subset G_n(A)$ such that 
$\alpha x = \left(\begin{smallmatrix} 1_{n-r} &  u\\  0 & v\end{smallmatrix}\right)$ 
with $u\in M_{n-r,1}(A)$ and $v\in M_{r,1}(A)$.
Left invertibility of $\alpha x$ implies that there are $T\in M_{n-r,r}(A)$ and $b\in M_{1,r}$ such that $u+Tv=0$ and $bv=1$.
The matrix $\beta= \left(\begin{smallmatrix} 1_{n-r} &  B\\  0 & 1_{r,r}\end{smallmatrix}\right)\left(\begin{smallmatrix} 1_{n-r} &  -T\\  0 & 1_{r,r}\end{smallmatrix}\right)$ is a product of elementary matrices where $B\in M_{n-r,r}(A)$ is the matrix all of whose rows equal $b$.
Then $\beta\alpha x = \left(\begin{smallmatrix} 1_{n-r} &  e_1+\cdots + e_{n-r}\\  0 & v\end{smallmatrix}\right)$ and
the matrix $w=(\beta\alpha x,e_{n-r+1})$ satisfies $\delta^{n-r+2}_{n-r+2}(w)=\beta \alpha x$.
For $i=1,...,n-r+1$, the matrix $\delta^i_{n-r+2}(w)$ can be completed to the matrix $(\delta^i_{n-r+2}(w),e_{n-r+2},...,e_n) \in SG_n(A)$.
It follows that 
$$y = x+(-1)^{n-r+2}d \alpha^{-1}\beta^{-1}w \in C_{n-r+1}(A^n)$$
and $dy=dx$.
\end{proof}

The group $GL_n(A)$ acts on $U_r(A^n)$ by left matrix multiplication, and so does its subgroup $G_n(A)$.
This makes the complex $C(A^n)$ into a complex of left $G_n(A)$-modules.
We may sometimes drop the letter $A$ in the notation $G_n(A)$, $SG_n(A)$ etc when the ring $A$ is understood.

\begin{lemma}
\label{lem:HMCAisTrivial}
For any ring $A$, right $G_n(A)$-module $M$ and integers $i,n \geq 0$ with $i\leq n-\sr (A)$ we have
$$H_{i}(M \stackrel{L}{\otimes}_{G_n(A)}C(A^n)) = 0.$$
\end{lemma}

\begin{proof}
This follows from Lemma \ref{lem:Acyclicity} in view of the spectral sequence
$$E_{p,q}^2= \Tor^{G_n}_p(M,H_q(C(A^n))) \Rightarrow H_{p+q}(M\stackrel{L}{\otimes}_{G_n}C(A^n)).$$
\end{proof}

For our arguments below we frequently need the following assumptions.

\begin{itemize}
\item[{\rm ($\ast$)}]
Let $m,n_0, t\geq 1$, $n\geq 0$ be integers such that
$A$ has an $\RR_m$-algebra structure,
 $n_0\cdot t <m$ and $0\leq n\leq n_0$, and
 $t$ is a multiple of every positive integer $\leq n_0$.
Set $\sigma = s_{m,-t}\in \Z[\bar{A}^*]$.
\end{itemize}
\vspace{1ex}

For an integer $r$, we denote by $C_{\leq r}(A^n)$ the subcomplex of $C(A^n)$ which is
$C_{\leq r}(A^n)_i = C_i(A^n)$ for $i\leq r$ and $C_{\leq r}(A^n)_i=0$ otherwise.
So, $C_{\leq r}(A^n)/C_{\leq r-1}(A^n)$ is  $C_r(A^n)$ placed in homological degree $r$.
This defines a filtration on $C(A^n)$ by complexes of $G_n(A)$-modules and thus a spectral sequence of $\Lambda$-modules
\begin{equation}
\label{eqn:SpSeq1}
E^1_{p,q}(A^n) = \Tor^{G_n}_{p}(\Lambda_{m,t},C_q(A^n)) \Rightarrow H_{p+q}(\Lambda_{m,t} \stackrel{L}{\otimes}_{G_n}C(A^n))
\end{equation}
with differential $d^r$ of bidegree $(r-1,-r)$.
The spectral sequence (\ref{eqn:SpSeq1}) comes with a filtration by $\Lambda$-modules
$$0 \subset F_{p+q,0} \subset F_{p+q-1,1} \subset F_{p+q-2,2} \subset \dots  \subset F_{0,p+q} = H_{p+q}(\Lambda_{m,t}\stackrel{L}{\otimes}_{G_n}C(A^n))$$
where $F_{p+q-s,s}$ is the image of
$$H_{p+q}(\Lambda_{m,t}\stackrel{L}{\otimes}_{G_n}C_{\leq s}(A^n)) \to H_{p+q}(\Lambda_{m,t}\stackrel{L}{\otimes}_{G_n}C(A^n)),$$
and $F_{p+q-s,s}/F_{p+q-s+1,s-1}\cong E^{\infty}_{p+q-s,s}(A^n)$.

\begin{lemma}
\label{lem:NS_2.3}
Assume {\rm ($\ast$)}.
Then the spectral sequence (\ref{eqn:SpSeq1}) has
$${E}^1_{p,q}(A^n)=\left\{ 
\renewcommand\arraystretch{1.5}
\begin{array}{ll}
H_p(SG_{n-q}(A),\Z), & 0\leq q< n,\ p\leq n_0\\
\Lambda_{m,t},& p=0,\ q=n\\
0, & q=n\ \text{and}\ p\neq 0,\ \text{or}\ q<0,\ \text{or}\ q>n.
\end{array}
\right.
$$
\end{lemma}

\begin{proof}
By definition, the group $G_n(A)$ acts transitively on the set $GU_q(A^n)$ with stabilizer at $(e_{n-q+1},...,e_n)\in GU_q(A^n)$ the subgroup $\Aff_{q,n-q}^G(A)$.
Recall (Shapiro's Lemma) that for any right $G$-module $M$ we have the quasi-isomorphism 
$$M\stackrel{L}{\otimes}_N\Z \stackrel{\sim}{\to} M\stackrel{L}{\otimes}_G\Z[G/N]$$ induced by the inclusions $N\subset G$ and $\Z\subset \Z[G/N]$.
For $M=\Lambda_{m,t}$, $G=G_n(A)$ and $N=\Aff_{q,n-q}^G(A)$ we therefore have quasi-isomorphisms
$$\Lambda_{m,t}\stackrel{L}{\otimes}_{\Aff_{q,n-q}^G}\Z
\stackrel{\sim}{\to}
\Lambda_{m,t}\stackrel{L}{\otimes}_{G_n}\Z[G_n/\Aff_{q,n-q}^G]
=\Lambda_{m,t}\stackrel{L}{\otimes}_{G_n}C_q(A^n).$$
In view of Corollary \ref{cor:HomologyOfAffineGps}, for
 $q<n$ we have the map of complexes
$$\Z\stackrel{L}{\otimes}_{SG_{n-q}}\Z {\longrightarrow}
\Lambda_{m,t}\stackrel{L}{\otimes}_{\Aff_{q,n-q}^G}\Z.
$$
which induces an isomorphism on homology in degrees $\leq n_0<m/t \leq m(n-q)/t$.
For $q=n$ we have $C_n(A^n)=\Z[G_n]$ and thus,
$$\Lambda_{m,t}\stackrel{L}{\otimes}_{G_n}C_n(A^n)= \Lambda_{m,t}\stackrel{L}{\otimes}_{G_n}\Z[G_n]=\Lambda_{m,t}.$$
\end{proof}

\begin{lemma}
\label{lem:NS_lemma_2.4}
Assume {\rm ($\ast$)}.
Then for $0<q<n$ and $p\leq n_0$, the differential $d^1_{p,q}: E^1_{p,q}(A^n) \to E^1_{p,q-1}(A^n)$ in the spectral sequence (\ref{eqn:SpSeq1}) is zero if $q$ is even and for $q$ odd it is the map 
$$H_p(SG_{n-q}(A),\Z) \to H_p(SG_{n-q+1}(A),\Z)$$
induced by the standard inclusion $SG_{n-q}(A)  \to SG_{n-q+1}(A)$.

For $q=n$ we have $d^1_{p,n}=0$ for $p\neq 0$, and $d^1_{0,n}=0$ for $n$ even and for $n$ odd $d^1_{0,n}$ is the map
$\Lambda_{m,t} \to \Z$
induced by the augmentation $\eps:\Lambda=\Z[\bar{A}^*] \to \Z$.
\end{lemma}

\begin{proof}
The differential $d^1_{p,q}$ is the map 
$$d^1_{p,q} = \sum_{j=1}^q(-1)^{j-1}(1\otimes \delta^j)_*: H_p(\Lambda_{m,t}\stackrel{L}{\otimes}_{G_n}\Z[GU_q]) \to H_p(\Lambda_{m,t}\stackrel{L}{\otimes}_{G_n}\Z[GU_{q-1}])$$
where $\delta^j:\Z[GU_q] \to Z[GU_{q-1}]$ is induced by the map $\delta^j:GU_q \to GU_{q-1}$ defined by $(v_1,...,v_q) \mapsto (v_1,...,\hat{v}_j,...,v_n)$.

Assume first that $q<n$.
Consider the diagram
$$\xymatrix{
\Z
\stackrel{L}{\otimes}_{SG_{n-q}}\Z
\ar[r]^{\hspace{-4ex}1\otimes u_q^n} \ar[d] &
\Z
\stackrel{L}{\otimes}_{SG_n}
\Z[GU_q(A^n)] \ar[r] \ar[d]^{1\otimes \delta^j} &
\Lambda_{m,t}
\stackrel{L}{\otimes}_{G_n}
\Z[GU_q(A^n)]\ar[d]^{1\otimes \delta^j}\\
\Z
\stackrel{L}{\otimes}_{SG_{n-q+1}}\Z
\ar[r]^{\hspace{-4ex}1\otimes u^n_{q-1}}  &
\Z
\stackrel{L}{\otimes}_{SG_n}
\Z[GU_{q-1}(A^n)] \ar[r]  &
\Lambda_{m,t}
\stackrel{L}{\otimes}_{G_n}
\Z[GU_{q-1}(A^n)]
}
$$
induced by the natural inclusions of groups and rings $SG_{n-q} \subset \Aff_{q,n-q}^{SG} \subset G_n$, $SG_{n-q} \subset SG_{n-q+1}$, $\Z \subset \Lambda_{m,t}$ and where 
$u^n_q:\Z \to \Z[GU_q(A^n)]$ sends $1$ to the left invertible matrix
$(e_{n-q+1},...,e_n)$.
The horizontal compositions induce the isomorphisms in Lemma \ref{lem:NS_2.3} upon taking homology in degrees $\leq n_0$.
We will show that the outer diagram commutes upon taking homology groups.
This implies the claim for $q<n$.
Since the right square commutes, it suffices to show that the left square commutes upon taking homology.

We will use the basic functoriality of group homology as recalled at the beginning of section \ref{Sec:HofAff}.
Upon taking homology, the left hand square of the diagram is the diagram
$$
\xymatrix{
H_*(SG_{n-q},\Z) \ar[rr]^{(i^n_{n-q},u^n_q)_*} \ar[d]_{(i^{n-q+1}_{n-q},1)_*} &&
H_*(SG_{n},\Z[GU_q]) \ar[d]^{(1,\delta^j)_*}\\
H_*(SG_{n-q+1},\Z) \ar[rr]_{(i^n_{n-q+1},u^n_{q-1})_*} && H_*(SG_{n},\Z[GU_{q-1}])
}
$$
where $u^n_r$ is the left invertible matrix $(e_{n-r+1},...,e_n)$ and $i^s_{r}:SG_r \to SG_s$ denotes the standard embedding for $r\leq s$.
Consider the matrix 
$$h = (e_1,...,e_{n-q},\tau e_{n-q+j}, \delta^j u^n_q)\in SG_n$$
where $\tau\in \{1,-1\}$ is chosen so that $\det(h)=1\in K_1(A)$.
Then $h\cdot i^n_{n-q} \cdot h^{-1} = i^n_{n-q}$ and $\delta^j(u^n_q) = h\cdot u^n_{q-1}$.
In view of the basic functoriality of group homology, 
the two compositions $(i^n_{n-q},\delta^ju^n_q)_*$ and $(i^n_{n-q},u^n_{q-1})_*$ are equal, and the square commutes.
Hence the lemma for $q<n$.

For $q=n$ and $p\neq 0$ we have $d^1_{p,n}=0$ since $E^1_{p,n}(A^n)=0$.
To prove the claim for $q=n$ and $p=0$, we need to show the commutativity of the diagram
$$\xymatrix{
\Lambda_{m,t} \ar[r]^{\hspace{-7ex}1\otimes u^n_n}_{\hspace{-7ex}\cong} \ar[d]_{\eps} &
\Lambda_{m,t} {\otimes}_{G_n}\Z[G_n(A)]\ar[d]^{1\otimes \delta^j}\\
\Z \ar[r]^{\hspace{-12ex}1\otimes u^n_{n-1}}&
\Lambda_{m,t} {\otimes}_{G_n}\Z[GU_{n-1}(A^n)].
}$$
Since (\ref{eqn:SpSeq1}) is a spectral sequence of $\Lambda_{m,t}$-modules,
this diagram is a diagram of $\Lambda_{m,t}$-modules. 
So, it suffices to show that the two compositions send $1\in \Lambda_{m,t}$ to the same element, that is, we need to see that
$$1\otimes \delta^j(u^n_n)=1\otimes u^n_{n-1}\in \Lambda_{m,t} {\otimes}_{G_n}\Z[GU_{n-1}(A^n)]$$
for all $j=1,...,n$.
Consider the matrix $h_j=(\tau  e_j, e_1,e_2,...,\hat{e}_j,...,e_n) \in SG_n(A)$ where $\tau \in \{1,-1\}$ is chosen so that $\det(h_j)=1\in K_1(A)$.
Then $\delta^j(u^n_n)=h_j\cdot u^n_{n-1}$ and thus
$$1\otimes \delta^j(u^n_n)=1\otimes h_j\cdot u^n_{n-1} = h_j\otimes u^n_{n-1} = 1\otimes u^n_{n-1} \in \Lambda_{m,t} {\otimes}_{G_n}\Z[GU_{n-1}(A^n)]$$
since $h_j\in SG_n(A)$ goes to $1\in \bar{A}^*$ under the map $G_n(A) \to \bar{A}^*$.
\end{proof}

\begin{proposition}
\label{prop:NS_Prop_2.6}
Assume {\rm ($\ast$)}.
The differentials $d^r_{p,q}$ in the spectral sequence (\ref{eqn:SpSeq1}) are zero for $r\geq 2$ and $p\leq n_0$.
\end{proposition}

\begin{proof}
We argue by induction on $n$.
For $n=0,1$, the differentials $d^r$, $r\geq 2$, are zero since $E_{p,q}^r(A^n)=0$ for $q\neq 0,1$.
Assume $n\geq 2$.
Similar to \cite{SuslinNesterenko}, consider the homomorphism of complexes 
$\psi:C(A^{n-2})[-2] \to C(A^n)$ defined in degree $q$ by $\psi=\psi_0-\psi_1+\psi_2$
where for $(v_1,...,v_{q-2}) \in GU_{q-2}(A^{n-2})$ the map $\psi_i$ is given by
$$
\renewcommand\arraystretch{1.5}
\begin{array}{rcl}
\psi_0(v_1,...,v_{q-2}) & = & (v_1,...,v_{q-2},e_{n-1},e_n) \\
\psi_1(v_1,...,v_{q-2}) & = & (v_1,...,v_{q-2},e_{n-1},e_n-e_{n-1}) \\
\psi_2(v_1,...,v_{q-2}) & = & (v_1,...,v_{q-2},e_{n},e_n-e_{n-1}).
\end{array}
$$
The map $\psi$ commutes with differentials and is compatible with the actions by $G_{n-2}(A)$ and $G_n(A)$ via the standard inclusion $G_{n-2}(A)\subset G_n(A)$.
Hence, $\psi$ induces a map of spectral sequences
$E(A^{n-2})[0,-2] \to E(A^n)$.
Denote by $E$ and $\tilde{E}$ the spectral sequences $E(A^n)$ and $E(A^{n-2})[0,-2]$.
From Lemma \ref{lem:NS_2.3}, we have
$$\tilde{E}^1_{p,q}=E^{1}_{p,q-2}(A^{n-2}) = \left\{
\renewcommand\arraystretch{1.5}
\begin{array}{ll}
H_p(SG_{n-q}(A),\Z)& 2 \leq q < n,\ p\leq n_0\\
\Lambda_{m,t} & q=n,\ p=0\\
0 & q=n\ \text{and}\ p\neq 0,\ \text{or}\ q<2,\ \text{or}\ q>n.
\end{array}
\right.
$$
The claim follows by induction on $r$ 
using the following lemma.
\end{proof}

\begin{lemma}
\label{lem:PsiIsIso}
Assume {\rm ($\ast$)}. 
Under the identifications of Lemma \ref{lem:NS_2.3}, the homomorphism $\psi:\tilde{E}^1_{p,q} \to E^1_{p,q}$ is the identity for $2\leq q \leq n$ and $p\leq n_0$.
\end{lemma}

\begin{proof}
Keeping the notations of the proof of Lemma \ref{lem:NS_lemma_2.4},
we will check the commutativity of the following diagrams for $j=0,1,2$ upon taking homology groups
$$
\xymatrix{
\Z\ \stackrel{L}{\otimes}_{SG_{n-q}}\Z \ar[r]^{\hspace{-8ex}1\otimes u^{n-2}_{q-2}} \ar[dr]_{1\otimes u^{n}_{q}} &
\Z \stackrel{L}{\otimes}_{SG_{n-2}}\Z[GU_{q-2}(A^{n-2})] \ar[r] \ar[d]^{1\otimes \psi_j}&
\Lambda_{m,t} \stackrel{L}{\otimes}_{G_{n-2}}\Z[GU_{q-2}(A^{n-2})] \ar[d]^{1\otimes \psi_j}\\
&
\Z \stackrel{L}{\otimes}_{SG_{n}}\Z[GU_{q}(A^{n})] \ar[r] &
\Lambda_{m,t} \stackrel{L}{\otimes}_{G_{n}}\Z[GU_{q}(A^{n})]
}$$
\begin{equation}
\label{eqn:lem:PsiIsIso1}
\xymatrix{
\Lambda_{m,t} \ar[rr]^{\hspace{-8ex}1\otimes u^{n-2}_{n-2}} \ar[drr]_{\hspace{-8ex}1\otimes u^{n}_{n}}  &&
\Lambda_{m,t} {\otimes}_{G_{n-2}}\Z[G_{n-2}] \ar[d]^{1\otimes \psi_j}\\
&&\Lambda_{m,t} {\otimes}_{G_{n}}\Z[G_{n}]
}
\end{equation}
where in the first diagram $2\leq q <n$, and the second diagram takes care of $q=n$.
Since $\psi=\psi_0-\psi_1+\psi_2$, commutativity of the diagrams on homology groups will imply the claim.

In the first diagram the right hand square commutes, and so we are left with showing the commutativity of the left hand square on homology, that is, the commutativity of the diagram
\begin{equation}
\label{eqn:lem:PsiIsIso2}
\xymatrix{
H_p(SG_{n-q},\Z) \ar[rr]^{\hspace{-8ex}(i^{n-2}_{n-q}, u^{n-2}_{q-2})_*} \ar[drr]_{\hspace{-3ex}(i^{n}_{n-q}, u^{n}_{q})_*} 
&& 
H_p(SG_{n-2},\Z[GU_{q-2}(A^{n-2})]) \ar[d]^{(i^n_{n-2},\psi_j)_*} 
\\
&&
H_p(SG_n,\Z[GU_q(A^{n})]).
}
\end{equation}
Commutativity for $j=0$ is clear.
For $j=1$ we consider the following matrix $h=(e_1,...,e_{n-1},e_n-e_{n-1})\in SG_n(A)$.
For $2\leq q <n$ we have $h\cdot i^n_{n-q}\cdot  h^{-1} = i^n_{n-q}$ and
$h\circ u_{q}^n = \psi_1(u^{n-2}_{q-2})$, this shows commutativity of diagram (\ref{eqn:lem:PsiIsIso2}) in this case.
For $j=2$, we replace the matrix $h$ with the matrix
$(e_1,...,e_{n-2},e_{n},e_n-e_{n-1})$ in $SG_n(A)$.
This finishes the proof of commutativity of (\ref{eqn:lem:PsiIsIso2}) for $j=0,1,2$.

To show commutativity of the second diagram (\ref{eqn:lem:PsiIsIso1}), note that it is a diagram of $\Lambda_{m,t}$-modules and that the horizontal and diagonal arrows are isomorphisms with inverses the multiplication maps
$\Lambda_{m,t}\otimes_G\Z[G] \stackrel{1\otimes \det}{\longrightarrow} \Lambda_{m,t}\otimes \Lambda_{m,t} \to \Lambda_{m,t}$.
Since $\psi_j(u^{n-2}_{n-2}) = 1 = u^n_n  \in \bar{A}^* \subset K_1(A)$, the claim follows.
\end{proof}

\begin{theorem}
\label{thm:SGHstability}
Let $A$ be a ring with many units.
Then the natural homomorphism
$$H_i(SG_{n-1}(A),\Z) \to H_i(SG_n(A),\Z)$$
is an isomorphism for $n\geq i + \sr(A) +1$ and surjective for $n\geq i+\sr(A)$.
\end{theorem}

\begin{proof}
Choose $n_0,m,t$ as in {\rm ($\ast$)}.
In particular $n\leq n_0$.
Then we have the spectral sequence (\ref{eqn:SpSeq1}) with
$E^1$-term given by Lemma \ref{lem:NS_2.3} and $d^1_{p,q}$ was computed in Lemma \ref{lem:NS_lemma_2.4} for $p\leq n_0$.
By Proposition \ref{prop:NS_Prop_2.6} and Lemma \ref{lem:HMCAisTrivial}
we have $E^2_{p,q}=E^{\infty}_{p,q}=0$ for $n\geq p+q +\sr(A)$ and $p\leq n_0$.
The claim follows.
\end{proof}

We can reformulate Theorem \ref{thm:SGHstability} in terms of elementary linear groups.
Recall \cite[\S 1]{Bass:StableK} that the group of elementary $r\times r$-matrices of a ring $A$ is the subgroup $E_r(A)$ of $GL_r(A)$ generated by the elementary matrices $e_{i,j}(a) = 1+ a\cdot e_ie_j^T$, $a\in A$.

\begin{lemma}
\label{lem:MaxPerf}
Let $A$ be a ring with many units.
Then for $n>\sr(A)$, we have $E_n(A)=SG_n(A)$, and this group is the commutator and the maximal perfect subgroup of $GL_n(A)$.
Moreover, the natural map $GL_n(A) \to K_1(A)$ is surjective for $n\geq \sr(A)$.
\end{lemma}

\begin{proof}
We clearly have $E_n(A) \subset SG_n(A)$.
It follows from the $GL_n$-version of Theorem \ref{thm:SGHstability} proved in \cite{SuslinNesterenko} that
the natural map $GL_n(A) \to K_1(A)$ is surjective for $r\geq \sr(A)$ and $GL_n(A)^{ab}=K_1(A)$ for $n>\sr(A)$.
Therefore, $GL_n(A)/SG_n(A) \cong K_1(A)$ for $r\geq \sr(A)$ and $SG_n(A)=[GL_n(A),GL_n(A)]$ for $n>\sr(A)$.

For the rest of the proof assume $n>\sr(A)$.
From \cite[Theorem 3.2]{Vaserstein:Stabilization}, we have $GL_n(A)/E_n(A)=K_1(A)$, hence $E_n(A)=SG_n(A)$.
Classically, the group $E_n(A)$ is perfect for $n\geq 3$ \cite[Corollary 1.5]{Bass:StableK}.
Alternatively, from Theorem \ref{thm:SGHstability}, we have an isomorphism
$$H_1(SG_n(A))\cong H_1(E(A)) = 0$$
since the infinite elementary linear group $E(A)$ is perfect and
$\colim_rSG_r(A) = E(A)$.
Hence, the group $E_n(A)$ is perfect (also when $n=2$).
Since the quotient $GL_n(A)/E_n(A)=K_1(A)$ is abelian, the group $E_n(A)$ is the maximal perfect subgroup of $GL_n(A)$.
\end{proof}

\begin{theorem}
\label{thm:Estability}
Let $A$ be a ring with many units.
Then the natural homomorphism
$$H_i(E_{n-1}(A),\Z) \to H_i(E_n(A),\Z)$$
is an isomorphism for $n\geq i + \sr(A) +1$ and surjective for $n\geq i+\sr(A)$.
\end{theorem}

\begin{proof}
This follows from Theorem \ref{thm:SGHstability} in view of Lemma \ref{lem:MaxPerf}.
\end{proof}

Denote by $BGL_n^+(A)$ the space obtained by applying Quillen's plus construction to the classifying space $BGL_n(A)$ of $GL_n(A)$ with respect to the maximal perfect subgroup of $GL_n(A)$.
The following proves a conjecture of Bass \cite[Conjecture XVI on p.\ 43]{BassConj} in the case of rings with many units.

\begin{theorem}
\label{thm:BassConjecture}
Let $A$ be a ring with many units, and let $n\geq 1$ be an integer. 
Then the natural homomorphism
$$\pi_iBGL_{n-1}^+(A) \to \pi_iBGL_n^+(A)$$
is an isomorphism for $n\geq i + \sr(A) +1$ and surjective for $n\geq i+\sr(A)$.
\end{theorem}

\begin{proof}
The case $i=0$ is trivial and
the case $i=1$ follows at once from Lemma \ref{lem:MaxPerf}.

Now, assume $i\geq 2$ and $n\geq i+\sr(A)\geq 2 +\sr(A)$.
Denote by $\F_n(A)$ the homotopy fibre of the map 
$BGL_{n-1}^+(A) \to BGL_n^+(A)$.
If follows from Lemma \ref{lem:MaxPerf}, that $\F_n$ is also the homotopy fibre of $BE_{n-1}^+(A) \to BE_n^+(A)$.
Since $E_{n-1}(A)$ and $E_n(A)$ are perfect, $\F_n(A)$ is connected and $\pi_1\F_n(A)$ is abelian as a quotient of $\pi_2BE_n^+(A)$.
From Theorem \ref{thm:Estability} and the (relative) Serre spectral sequence
$$E^2_{r,s}=H_r(BE_n^+(A),H_s(\pt,\F_n)) \Rightarrow H_{r+s}(BE_n^+(A),BE_{n-1}^+(A))$$
associated to the fibration 
$\F_n \to BE_{n-1}^+(A) \to BE_n^+(A)$
with  simply connected base we find
$$H_{i-1}(\F_n,\pt)=H_i(\pt,\F_n)=H_i(E_n(A),E_{n-1}(A))$$
for $i\leq n-\sr(A)+1$ where $\pt\in \F_n(A)$ denotes the base point of $\F_n(A)$.
By Hurewicz's Theorem, it follows from Theorem \ref{thm:Estability} that the natural map $$\pi_{i-1}(\F_n(A),\pt) \to H_{i-1}(\F_n(A),\pt)$$ is an isomorphism for $i\leq n-\sr(A)+1$.
In particular, we have
$$\pi_{i-1}(\F_n(A),\pt) = H_i(E_n(A),E_{n-1}(A))=0,\hspace{3ex}i\leq n-\sr(A).$$
Hence the result.
\end{proof}

Recall that for a commutative ring $A$ and integer $n\geq 1$,  the special linear group $SL_n(A)$ of $A$ is the kernel of the determinant map $GL_n(A) \to A^*$.

\begin{theorem}
\label{thm:BassConjectureBis}
Let $A$ be a local commutative ring with infinite residue field, and let $n\geq 2$ be an integer. 
Denote by $\F_n(A)$ the homotopy fibre of the map 
$BGL_{n-1}^+(A) \to BGL_n^+(A)$.
Then for $i\geq 1$ we have 
$$\pi_{i-1}\F_n(A) = \left\{\begin{array}{ll} 0, & i <n \\ H_i(SL_n(A),SL_{n-1}(A)), & i=n.\end{array}\right.$$
\end{theorem}

\begin{proof}
Note that for $A$ as in the theorem $E_r(A)=SL_r(A)$ is the maximal perfect subgroup of $GL_r(A)$ for all $r\geq 1$.
Now the proof is the same as for Theorem \ref{thm:BassConjecture}, the only improvement being that $\F_n$ is already the homotopy fibre of $BE_{n-1}(A)^+ \to BE_{n}(A)^+$ for $n\geq 2$.
\end{proof}

\begin{theorem}
\label{thm:SLstability}
Let $A$ be a commutative ring with many units and $n\geq 2$ an integer.
Then the natural homomorphism
$$H_i(SL_{n-1}(A),\Z) \to H_i(SL_n(A),\Z)$$
is an isomorphism for $n\geq i + \sr(A) +1$ and surjective for $n\geq i+\sr(A)$.
\end{theorem}

\begin{proof}
For $r\geq \sr(A)$, the natural map $GL_r(A) \to K_1(A)$ is surjective, by Lemma \ref{lem:MaxPerf}.
When $A$ is commutative we have an exact sequence of groups for $r\geq \sr(A)$
$$1\to SG_r(A) \to SL_r(A) \to SK_1(A) \to 0.$$
Hence, for $n\geq \sr(A)+1$ we have the associated Serre spectral sequence
$$H_i(SK_1(A),H_j(SG_nA,SG_{n-1}A)) \Rightarrow H_{i+j}(SL_n(A),SL_{n-1}(A)).$$
The result now follows from Theorem \ref{thm:SGHstability}.
\end{proof}

\section{Milnor-Witt $K$-theory}
\label{sec:KMW}

Let $A$ be a commutative ring.
Recall that we denote by $A^*$ the group of units in $A$.
Elements in the integral group ring $\Z[A^*]$ of $A^*$ corresponding to $a\in A^*$ are denoted by $\langle a\rangle$.
Note that $\langle 1 \rangle =1 \in \Z[A^*]$.
We denote by $\langle\langle a\rangle\rangle$ the element $\langle\langle a\rangle\rangle = \langle a\rangle -1 \in \Z[A^*]$.
Let $I[A^*]$ be the augmentation ideal in $\Z[A^*]$,
that is, $I[A^*]$ is the kernel of the ring homomorphism $\Z[A^*] \to \Z:\langle a\rangle \mapsto 1$.
We denote by $[a]$ the element $[a]=\langle a\rangle -1\in I[A^*]$.
Under the canonical embedding $I[A^*] \subset \Z[A^*]$, the element $[a]$ maps to $\langle\langle a\rangle\rangle$.

\begin{lemma}
\label{lem:IAPresentation}
The augmentation ideal $I[A^*]$ is the $\Z[A^*]$-module generated by symbols $[a]$ for $a\in A^*$ subject to the relation
$$[ab]=[a]+\langle a \rangle [b].$$
\end{lemma}

\begin{proof}
Clearly, the equation $[ab]=[a]+\langle a \rangle [b]$ holds in $I[A^*]$.
Let $\tilde{I}_{A^*}$ be the $\Z[A^*]$-module generated by symbols $[a]$ for $a\in A^*$ subject to the relation
$[ab]=[a]+\langle a \rangle [b]$.
Note that $[1]=0$ in $\tilde{I}_{A^*}$ because $[1\cdot 1]=[1]+\langle 1 \rangle [1]$ and $\langle 1 \rangle =1 \in \Z[A^*]$.

Consider the $\Z[A^*]$-module map 
$\tilde{I}_{A^*} \to I[A^*]: [a]\mapsto \langle a\rangle -1$.
The set consisting of $\langle a\rangle -1 \in I[A^*]$, $a\in A^*$, $a\neq 1$, defines a $\Z$-basis of $I[A^*]$.
This allows us to define a $\Z$-linear homomorphism
$I[A^*] \to \tilde{I}_{A^*}:\langle a\rangle -1 \mapsto [a]$.
Clearly, the composition $I[A^*] \to \tilde{I}_{A^*} \to I[A^*]$ is the identity.
Finally, the map $I[A^*] \to \tilde{I}_{A^*}$ is surjective because
$\langle b \rangle [a] = [ab]-[b]$.
\end{proof}

\begin{definition}
\label{dfn:KMWhat}
Let $A$ be a commutative ring.
We define the graded ring $\hat{K}^{MW}_*(A)$ 
as the tensor algebra of the augmentation ideal $I[A^*]$ (placed in degree $1$) over the group ring $\Z[A^*]$ modulo the Steinberg relation $[a][1-a]=0$ for $a,1-a\in A^*$:
$$\hat{K}^{MW}_*(A)=\bigoplus_{n\geq 0}\hat{K}^{MW}_n(A)=\Tens_{\Z[A^*]}(I[A^*])/[a][1-a].$$
\end{definition}

In view of Lemma \ref{lem:IAPresentation}, the graded ring $\hat{K}^{MW}_*(A)$ is the $\Z[A^*]$-algebra generated by symbols $[a]$, $a\in A^*$, in degree $1$ subject to the relations
\begin{enumerate}
\item
For $a,b\in A^*$ we have $[ab]=[a]+\langle a\rangle [b]$.
\item
For $a,1-a\in A^*$ we have the Steinberg relation $[a][1-a]=0$.
\end{enumerate}

It is convenient to write $[a_1,...,a_n]$, $h$ and $\eps$ for the following elements in $\hat{K}^{MW}_*(A)$
$$[a_1,...,a_n] = [a_1]\cdots [a_n] \in \hat{K}_n^{MW}(A),$$
$$h=1+\langle -1\rangle,\hspace{2ex}
\eps = -\langle -1\rangle \in \hat{K}_0^{MW}(A)$$
where $a_1,...,a_n\in A^*$.

\begin{lemma}
\label{Hatlem:basicKMWformulas}
Let $A$ be a commutative ring.
Then in the ring $\hat{K}^{MW}_*(A)$
we have for all $a,b,c,d\in A^*$ the following relations.
\begin{enumerate}
\item
\label{Hatlem:basicKMWformulas:1}
$[ab]=[a]+\langle a\rangle [b]$
\item
\label{Hatlem:basicKMWformulas:2}
$\langle 1 \rangle = 1$ and $[1]=0$, 
\item
\label{Hatlem:basicKMWformulas:3}
$\langle a b\rangle = \langle a\rangle \cdot \langle b\rangle$
and $\langle a\rangle$ is central in $\hat{K}^{MW}_*(A)$.
\item
\label{Hatlem:basicKMWformulas:2.5}
$[\frac{a}{b}]=[a]-\langle \frac{a}{b}\rangle [b]$, in particular, $[b^{-1}]=-\langle b^{-1}\rangle [b]$.
\item
\label{Hatlem:TowardsGWactionOnTwistedK:1}
$\langle\langle a \rangle\rangle [b] = \langle\langle b\rangle\rangle [a]$.
\item
\label{Hatlem:TowardsGWactionOnTwistedK:2}
If $a+b=1$ then 
$\langle\langle a\rangle\rangle[b,c]=0$.
\item
\label{Hatlem:TowardsGWactionOnTwistedK:3}
If $a+b=1$ then
$\langle\langle a\rangle\rangle \langle\langle b\rangle\rangle [c,d]=0$.
\item
\label{Hatlem:TowardsGWactionOnTwistedK:4}
$\langle\langle a\rangle\rangle [b,c] = \langle\langle a\rangle\rangle [c,b]$.
\end{enumerate}
\end{lemma}

\begin{proof}
Items (\ref{Hatlem:basicKMWformulas:1})- (\ref{Hatlem:basicKMWformulas:3}) follow from the definition of $\hat{K}^{MW}_*(A)$.

(\ref{Hatlem:basicKMWformulas:2.5}) 
We have $[a]=[\frac{a}{b}\cdot b] = [\frac{a}{b}]+ \langle \frac{a}{b}\rangle [b]$.

(\ref{Hatlem:TowardsGWactionOnTwistedK:1})
We have $[a]+\langle a\rangle [b] = [ab] = [b] + \langle b\rangle [a]$ which is 
$\langle\langle a \rangle\rangle [b] = \langle\langle b\rangle\rangle [a]$.

(\ref{Hatlem:TowardsGWactionOnTwistedK:2})
By (\ref{Hatlem:TowardsGWactionOnTwistedK:1}) and the Steinberg relation we have $\langle\langle a\rangle\rangle[b,c] = \langle\langle c\rangle\rangle[b,a] = 0$.

(\ref{Hatlem:TowardsGWactionOnTwistedK:3})
Similarly, we have
 $\langle\langle a\rangle \rangle \langle\langle b\rangle\rangle [c,d] =
\langle\langle c\rangle \rangle \langle\langle d\rangle\rangle [a,b] = 0$,
by (\ref{Hatlem:TowardsGWactionOnTwistedK:1}) and the Steinberg relation.

(\ref{Hatlem:TowardsGWactionOnTwistedK:4})
From (\ref{Hatlem:TowardsGWactionOnTwistedK:1}) and (\ref{Hatlem:basicKMWformulas:3}) we have
$\langle\langle a\rangle\rangle [b,c]= \langle\langle b\rangle\rangle [a,c]
=\langle\langle c\rangle\rangle [a,b] = \langle\langle a\rangle\rangle [c,b]$.
\end{proof}

\begin{lemma}
\label{Hatlem:basicKMWformulasTwo}
Let $A$ be either a field or a commutative local ring whose residue field has at least $4$ elements.
Then for all $a,b,c\in A^*$ the following relations hold in $\hat{K}^{MW}_*(A)$.
\begin{enumerate}
\item
\label{Hatlem:basicKMWformulas:4}
$[a][-a]=0$
\item
\label{Hatlem:basicKMWformulas:5}
$[a][a]=[a][-1]=[-1][a]$
\item
\label{Hatlem:basicKMWformulas:7}
$[a][b]=\eps [b][a]$ where $\eps =-\langle -1 \rangle$
\item
\label{lem:TowardsGWactionOnTwistedK:5}
$\langle\langle a \rangle\rangle\cdot h\cdot [b,c] = 0$
where $h=1+ \langle -1 \rangle$.
\item
\label{lem:TowardsGWactionOnTwistedK:6}
$[a^2,b]=h\cdot [a,b]$
\item
\label{lem:TowardsGWactionOnTwistedK:7}
$\langle\langle a^2\rangle\rangle [b,c]=0$, in particular, 
$\langle a^2\rangle[b,c]=[b,c]$.
\end{enumerate}
\end{lemma}

\begin{proof}
(\ref{Hatlem:basicKMWformulas:4})
First assume $\bar{a}\neq 1$ where $\bar{a}$ means reduction modulo the maximal ideal in $A$.
Then $1-a, 1-a^{-1}\in A^*$ and $-a = \frac{1-a}{1-a^{-1}}$.
Therefore,
$[a][-a]= [a][\frac{1-a}{1-a^{-1}}] = [a]([1-a]-\langle -a\rangle [1-a^{-1}]) = 
-\langle -a\rangle [a][1-a^{-1}] = -\langle -a\rangle \langle a\rangle [a^{-1}][1-a^{-1}] =0$ where the second to last equation is from Lemma \ref{Hatlem:basicKMWformulas} (\ref{Hatlem:basicKMWformulas:2.5}).
This already implies the case when $A$ is a field.

Now assume that $A$ is local whose residue field has at least $4$ elements.
Assume $\bar{a}=1$ and choose $b\in A^*$ with $\bar{b}\neq 1$.
Then $\bar{a}\bar{b}\neq 1$.
Therefore,
$0=[ab][-ab]=([a]+\langle a\rangle[b])[-ab] = [a][-ab]+\langle a\rangle[b][-ab]
= [a]([-a]+\langle -a\rangle [b])+\langle a\rangle [b]([a]+\langle a\rangle[-b])
=[a][-a]+\langle -a\rangle[a][b] +\langle a\rangle [b][a]$.
Hence, for all $\bar{b}\neq 1$ we have
$[a][-a]=-\langle -a\rangle [a][b]-\langle a \rangle [b][a]$.
Now, choose $b_1,b_2\in A^*$ such that $\bar{b}_1,\bar{b}_2,\bar{b}_1\bar{b}_2\neq 1$. 
This is possible if the residue field of $A$ has at least $4$ elements.
Then
$[a][-a]=-\langle -a\rangle[a][b_1b_2]-\langle a\rangle[b_1b_2][a]
= -\langle -a\rangle[a]([b_1]+\langle b_1\rangle [b_2])
 -\langle a\rangle ([b_1]+\langle b_1\rangle [b_2])[a]
= -\langle -a\rangle [a][b_1] - \langle a\rangle [b_1][a] 
+\langle b_1\rangle (-\langle -a\rangle [a][b_2]-\langle a\rangle [b_2][a])
= [a][-a]+\langle b_1 \rangle [a][-a]$.
Hence, $\langle b_1\rangle [a][-a]=0$.
Multiplying with $\langle b_1^{-1}\rangle $ yields the result.

(\ref{Hatlem:basicKMWformulas:5})
We have 
$[a][a]=[(-1)(-a)][a]=([-1]+\langle -1\rangle [-a])[a] = [-1][a]$.
Similarly, $[a][a] = [a][(-1)(-a)] = [a]([-1]+\langle -1\rangle [-a]) = [a][-1]$.

(\ref{Hatlem:basicKMWformulas:7})
We have 
$0 = [ab][-ba]=([a]+\langle a\rangle [b])[-ab]
= [a][-ab]+\langle a\rangle [b][-ab]
= [a]([-a]+\langle -a\rangle [b]) + \langle a\rangle [b]([a]+\langle a\rangle [-b]) 
=\langle -a\rangle [a][b] + \langle a\rangle [b][a]$.
Multiplying with $\langle a^{-1}\rangle$ yields $[a][b]=-\langle -1\rangle [b][a]$.

(\ref{lem:TowardsGWactionOnTwistedK:5})
From Lemma \ref{Hatlem:basicKMWformulas} (\ref{Hatlem:TowardsGWactionOnTwistedK:4}) and Lemma \ref{Hatlem:basicKMWformulasTwo} (\ref{Hatlem:basicKMWformulas:7}) we have
$\langle \langle a \rangle\rangle [b,c] = -\langle -1\rangle \langle \langle a \rangle\rangle [b,c]$ and thus $\langle \langle a \rangle\rangle (1+\langle -1\rangle)  [b,c]=0$.

(\ref{lem:TowardsGWactionOnTwistedK:6})
We have $[a^2,b]=[a,b]+\langle a\rangle [a,b]=2[a,b]+\langle\langle a\rangle\rangle [a,b] = 2[a,b]+\langle\langle b\rangle\rangle [a,a]=
2[a,b]+\langle\langle b\rangle\rangle [-1,a] = 2[a,b]+\langle\langle -1\rangle\rangle [a,b] = h\cdot [a,b]$.

(\ref{lem:TowardsGWactionOnTwistedK:7})
We have $\langle\langle a^2\rangle\rangle [b,c] = \langle\langle b\rangle\rangle [a^2,c] = \langle\langle b\rangle\rangle \cdot h \cdot [a,c] =0$.
\end{proof}

\begin{corollary}
Let $A$ be a either a field or a commutative local ring whose residue field has at least $4$ elements.
Then in $\hat{K}^{MW}_*(A)$ we have $[a_1,\dots,a_n]=0$ if $a_i+a_j = 1$ or $a_i+a_j=0$ for some $i\neq j$.
\end{corollary}

\begin{proof}
This follows from the Steinberg relation and Lemma \ref{Hatlem:basicKMWformulasTwo} (\ref{Hatlem:basicKMWformulas:4}) and (\ref{Hatlem:basicKMWformulas:7}).
\end{proof}

\begin{definition}
Let $A$ be a commutative ring.
We define the ring 
$$GW(A)$$
 as the quotient of the group ring $\Z[A^*]$
modulo the following relations.
\begin{enumerate}
\item
For all $a\in A^*$ we have $\langle\langle a \rangle\rangle h =0$.
\item
(Steinberg relation) For all $a,1-a\in A^*$ we have $\langle\langle a\rangle\rangle \langle\langle 1-a\rangle\rangle =0$.
\end{enumerate}
\end{definition}

\begin{definition}
Let $A$ be a commutative ring.
We define the $\Z[A^*]$-module 
$$V(A)$$
as the quotient of the augmentation ideal $I[A^*]$ modulo the relations
\begin{enumerate}
\item
For all $a,b\in A^*$ we have $\langle\langle a\rangle\rangle \cdot h\cdot [b] =0$.
\item
(Steinberg relation)
For all $a,1-a\in A^*$ we have $\langle\langle a\rangle\rangle [1-a] = 0$.
\end{enumerate}
\end{definition}

Since $\langle\langle a\rangle\rangle [b] = \langle\langle b\rangle\rangle [a]$ in $I[A^*]$ and hence in $V(A)$, we see that the $\Z[A^*]$-module $V(A)$ is naturally a $GW(A)$-module.

\begin{proposition}
\label{prop:hatKMWisTensGWV}
Let $A$ be commutative ring.
Then the natural surjections $\Z[A^*] \to GW(A)$ and $I[A^*]\to V(A)$ induce a surjective map of graded rings
$$\hat{K}^{MW}_*(A) \to \Tens_{GW(A)}\, V(A)\, /\, [a][1-a].$$
If $A$ is either a field or a local ring whose residue field has at least $4$ elements then this map is an isomorphism in degrees $\geq 2$.
\end{proposition}

\begin{proof}
It is clear that the map in the proposition is a surjective ring homomorphism.
Its kernel 
is the (homogeneous) ideal in $\hat{K}^{MW}_*(A)$ generated by the elements of the form $\langle\langle a \rangle\rangle h$ ($a\in A^*$), $\langle\langle a\rangle\rangle \langle\langle 1-a\rangle\rangle$ and $\langle\langle a\rangle\rangle [1-a]$ ($a,1-a\in A^*$).
If $A$ is a field or a local ring with residue field cardinality $\geq 4$, then
this ideal is zero in degrees $\geq 2$, in view of
Lemmas \ref{Hatlem:basicKMWformulas} (\ref{Hatlem:TowardsGWactionOnTwistedK:2}), (\ref{Hatlem:TowardsGWactionOnTwistedK:3}) and \ref{Hatlem:basicKMWformulasTwo} (\ref{Hatlem:basicKMWformulas:7}), (\ref{lem:TowardsGWactionOnTwistedK:5}). 
\end{proof}

\begin{remark}
\label{rmk:hatKMWTensGenlA}
Since the homomorphism of graded rings in Proposition \ref{prop:hatKMWisTensGWV}
is surjective for any commutative ring $A$, all formulas 
in Lemma \ref{Hatlem:basicKMWformulas} also hold in 
the target of that map.
\end{remark}

In case $A$ is a field the following definition is due to Hopkins and Morel
 \cite[Definition 3.1]{morel:book}.
For commutative local rings, the definition was also considered in \cite{GilleEtAl}.

\begin{definition}
\label{dfn:HopkinsMorelMilnorWittK}
Let $A$ be a commutative ring.
The Milnor-Witt $K$-theory of $A$ is the graded associative ring $K_*^{MW}(A)$ generated by symbols $[a]$, $a\in A^*$, of degree $1$ and one symbol $\eta$ of degree $-1$ subject to the following relations.

\begin{enumerate}
\item
For $a,1-a\in A^*$ we have $[a][1-a]=0$.
\item
For $a,b\in A^*$, we have $[ab]=[a]+[b]+\eta[a][b]$.
\item
For each $a\in A^*$, we have $\eta[a]=[a]\eta$, and
\item $\eta^2 [-1] +2\eta =0$.
\end{enumerate}
\end{definition}

\begin{definition}
\label{dfn:KMWtilde}
Let $A$ be a commutative ring and $n$ an integer.
Let $\tilde{K}_n^{MW}(A)$ be the abelian group generated by symbols of the form $[\eta^m,u_1,...,u_{n+m}]$ with $m\geq 0$, $n+m\geq 0$, $u_i\in A^*$, subject to the following three relations.
\begin{enumerate}
\item
$[\eta^m,u_1,...,u_{n+m}]=0$ if $u_i+u_{i+1}=1$ for some $i=1,...,n+m-1$.
\item
For all $a,b\in A^*$, $m\geq 0$ and $i=1,...n+m$ we have
$$
\renewcommand\arraystretch{1.5}
\begin{array}{l}
[\eta^m,u_1,...,u_{i-1},ab,u_{i+1},...]= \\
\phantom{}
[\eta^m,...,u_{i-1},a,u_{i+1},...] + [\eta^m,...,u_{i-1},b,u_{i+1},...] + [\eta^{m+1},...,u_{i-1},a,b,u_{i+1},...]
\end{array}
$$
\item
For each $m\geq 0$ and $i=1,...,n+m+2$ we have
$$[\eta^{m+2},u_1,...,u_{i-1},-1,u_{i+1},...,u_{n+m+2}] + 2[\eta^{m+1},...,u_{i-1},u_{i+1},...,u_{n+m+2}]=0.$$
\end{enumerate}
\end{definition}

We make $\tilde{K}_*^{MW}(A) = \bigoplus_n\tilde{K}_n^{MW}(A)$ into a graded ring with multiplication 
$$\tilde{K}_r^{MW}(A) \otimes \tilde{K}_s^{MW}(A) \to \tilde{K}_{r+s}^{MW}(A)$$
defined by 
$$[\eta^m,u_1,\dots,u_{r+m}]\otimes [\eta^{n},v_1,...,v_{s+n}] \mapsto 
[\eta^{m+n},u_1,\dots,u_{r+m},v_1,\dots v_{s+n}].$$
By going through the 3 relations in Definition \ref{dfn:KMWtilde} this map is well-defined as
map of abelian groups.
The multiplication is obviously associative and unital with unit $1 = [\eta^0]$.
We define a map of graded rings
$$K_*^{MW}(A) \to \tilde{K}_*^{MW}(A)$$
by sending $\eta$ to $[\eta]$ and $[u]$ to $[\eta^0,u]$.
It is easy to check that the defining relations for $K_*^{MW}(A)$ hold in $\tilde{K}_*^{MW}(A)$.

\begin{lemma}{\cite[Lemma 3.4]{morel:book}}
\label{lem:KMWntrivialPresentation}
For any commutative ring $A$ , the maps 
$$\tilde{K}_n^{MW}(A) \to K_n^{MW}(A):[\eta^m,u_1,...,u_{n+m}] \to \eta^m[u_1]\cdots [u_{n+m}]$$
define an isomorphism of graded rings
$$\tilde{K}^{MW}_*(A) \stackrel{\cong}{\longrightarrow}K^{MW}_*(A).$$
\end{lemma}

\begin{proof}
The composition 
$\tilde{K}_*^{MW}(A) \to K_*^{MW}(A) \to \tilde{K}_*^{MW}(A)$
is the identity, and
the first map is surjective.
\end{proof}

For $a\in A^*$ set $\langle a \rangle = 1 + \eta[a]\in K^{MW}_0(A)$ and $\langle\langle a\rangle \rangle = \langle a \rangle -1 = \eta[a] \in K^{MW}_0(A)$.

\begin{lemma}
\label{lem:basicKMWformulas}
Let $A$ be a commutative ring.
Then for all $a,b\in A^*$ we have in
$K^{MW}_*(A)$ the following.
\begin{enumerate}
\item
\label{lem:basicKMWformulas:1}
$[ab]=[a]+\langle a\rangle [b]$
\item
\label{lem:basicKMWformulas:2}
$\langle 1 \rangle = 1$ and $[1]=0$, 
\item
\label{lem:basicKMWformulas:2.5}
$[\frac{a}{b}]=[a]-\langle \frac{a}{b}\rangle [b]$, in particular, $[b^{-1}]=-\langle b^{-1}\rangle [b]$.
\item
\label{lem:basicKMWformulas:3}
$\langle a b\rangle = \langle a\rangle \cdot \langle b\rangle$
and $\langle a\rangle$ is central in $K^{MW}_*(A)$.
\item
\label{lem:KMWmodG:1}
The map $\Z[A^*] \to K^{MW}_0(A):\langle a\rangle \mapsto \langle a \rangle$ is a surjective ring homomorphism making $K^{MW}_*(A)$ into a $\Z[A^*]$-algebra.
\item
\label{lem:K0MWformulas:2}
$\langle \langle a \rangle\rangle \cdot h = 0$ where $h=1+\langle -1\rangle$.
\item
\label{lem:K0MWformulas:3}
If $a+b=1$ then $\langle\langle a\rangle\rangle [b]=0$.
\end{enumerate}
\end{lemma}

\begin{proof}
(\ref{lem:basicKMWformulas:1})
We have
$[ab] = [a] + [b] + \eta[a][b] = [a]+(1+\eta[a])[b]=[a]+\langle a\rangle [b]$.

(\ref{lem:basicKMWformulas:2})
We have $[-1]=[-1]+[1] + \eta[-1][1]$,
hence $0 = \eta [1] + \eta^2[-1][1] = \eta[1] -2\eta [1] = -\eta[1]$.
This shows $\langle 1 \rangle = 1$.
Now, $[1]=[1\cdot 1]=[1]+\langle 1\rangle [1] = [1]+[1]$ from which we obtain $[1]=0$.

(\ref{lem:basicKMWformulas:2.5})
We have $[a]=[\frac{a}{b}b] = [\frac{a}{b}]+\langle\frac{a}{b}\rangle [b]$.

(\ref{lem:basicKMWformulas:3})
We have
$\langle a \rangle \cdot \langle b \rangle = (1+ \eta[a])(1+\eta[b]) = 
1+ \eta([a] + [b] + \eta [a][b]) = 1+\eta[ab] = \langle ab\rangle$.
Moreover,
$\langle a \rangle [b] = (1+\eta[a])[b]=[b]+\eta[a][b]=[ab]-[a]=[ba]-[a]$ whereas
$[b]\langle a\rangle = [b](1+\eta[a])=[b]+\eta[b][a]=[ba]-[a]$.
Hence $\langle a\rangle[b]=[b]\langle a \rangle$.

(\ref{lem:KMWmodG:1})
It is clear that $\Z[A^*] \to K^{MW}_0(A)$ is a ring homomorphism.
By Lemma \ref{lem:KMWntrivialPresentation}, the group $K^{MW}_0(A)$ is additively generated by $[\eta^{0}]$ and $[\eta,a]$, equivalently, by $[\eta^0] = 1 = \langle 1\rangle$ and $[\eta^0] + [\eta,a] = \langle a \rangle$.
Hence, the map of rings is surjective.
The ring $K^{MW}_*(A)$ is a $\Z[A^*]$-algebra since $K_0^{MW}(A)$ is central in $K^{MW}_*(A)$.

(\ref{lem:K0MWformulas:2})
We have
$\langle a \rangle \cdot  h = (1+\eta[a])\cdot h = h$ because $\eta h = 0$.
Hence $\langle\langle a\rangle\rangle\cdot h = 0$

(\ref{lem:K0MWformulas:3})
If $a+b=1$ then $\langle\langle a\rangle\rangle [b] = \eta[a][b]=0$.
\end{proof}

\begin{lemma}
\label{lem:a2is1inGW}
Let $A$ be a commutative ring.
Then the following map defines an isomorphism of rings 
$$GW(A) \stackrel{\cong}{\longrightarrow} K^{MW}_0(A):\langle a\rangle \mapsto \langle a\rangle$$
\end{lemma}

\begin{proof}
By Lemma \ref{lem:basicKMWformulas} the map in the lemma is a surjective ring homomorphism.
Using Lemma \ref{lem:KMWntrivialPresentation}, we define the inverse by
$$\tilde{K}_0^{MW}(A) \to GW(A): [\eta^m,a_1,\dots,a_m]\mapsto \prod_{i=1}^m \langle\langle a_i\rangle\rangle$$
It is easy to check that this also defines a surjective ring homomorphism.
Since the composition 
$GW(A) \to \tilde{K}_0^{MW}(A) \to GW(A)$ is the identity, we are done.
\end{proof}

\begin{lemma}
\label{lem:SomeVARln}
Let $A$ be a commutative ring.
Then for $a,b\in A^*$ we have in $V(A)$
\begin{enumerate}
\item
\label{lem:SomeVARln:1}
$\langle\langle a\rangle\rangle[b] = \langle\langle b\rangle\rangle [a]$
\item
\label{lem:SomeVARln:2}
$h[b] = 2[b]+\langle\langle b\rangle\rangle [-1]$
\end{enumerate}
\end{lemma}

\begin{proof}
(\ref{lem:SomeVARln:1}) 
The equation holds in $I[A^*]$ and hence in its quotient $V(A)$.

(\ref{lem:SomeVARln:2})
We have
$h[b] = [b] + \langle -1 \rangle [b] = 2[b] + \langle\langle -1\rangle\rangle [b] = 2[b] + \langle\langle b \rangle\rangle [-1]$.
\end{proof}

\begin{lemma}
\label{lem:VisKMW1}
Let $A$ be a commutative ring.
Then we have an isomorphism of $GW(A)=K_0^{MW}(A)$-modules
$$V(A) \stackrel{\cong}{\longrightarrow} K_1^{MW}(A): [a] \mapsto [a].$$
\end{lemma}

\begin{proof}
By Lemma \ref{lem:basicKMWformulas} the map in the lemma is well-defined.
Using Lemma \ref{lem:KMWntrivialPresentation} we define the inverse by 
$$\tilde{K}_1^{MW}(A) \to V(A): [\eta^m,u_1,\dots,u_{m+1}] \mapsto \left(\prod_{i=1}^m \langle\langle u_i\rangle\rangle\right) [u_{m+1}].$$
This map preserves the relations of Definition \ref{dfn:KMWtilde}, the last one follows from Lemma \ref{lem:SomeVARln} (\ref{lem:SomeVARln:2}) which implies
$0 = \langle\langle a\rangle\rangle h [b] = \langle\langle a\rangle\rangle \langle\langle b\rangle\rangle [-1] + 2\langle\langle a\rangle\rangle [b]\in V(A)$.
\end{proof}

For $A$ a field, the following is a remark (without proof) in \cite{morel:book}.

\begin{proposition}
\label{prop:KMWisTensGWV}
Let $A$ be a commutative ring.
The isomorphisms $GW(A)\cong K_0^{MW}(A)$ and $V(A)\cong K_1^{MW}(A)$ in Lemmas \ref{lem:a2is1inGW} and \ref{lem:VisKMW1}
extend to an isomorphism of graded rings
$$\Tens_{GW(A)}V(A)/\{[a][1-a]|\ a,1-a\in A^*\} \stackrel{\cong}{\longrightarrow} K^{MW}_{\geq 0}(A).$$
\end{proposition}

\begin{proof}
Using Lemma \ref{lem:KMWntrivialPresentation}, the inverse is given by the ring map
$$
\renewcommand\arraystretch{2.5}
\begin{array}{rcl}
\tilde{K}^{MW}_{\geq 0}(A) & \to & \Tens_{GW(A)}V(A)/\{[a][1-a]|\ a,1-a\in A^*\}\\
\phantom{}[\eta^m,u_1,\dots,u_{m+n}] & \mapsto & \displaystyle{\left(\prod_{i=1}^m \langle\langle u_i\rangle\rangle\right) [u_{m+1}]\cdots [u_{m+n}]}
\end{array}
$$
This map is well-defined and indeed the required inverse in view of
Lemma \ref{lem:SomeVARln} and Remark \ref{rmk:hatKMWTensGenlA}.
\end{proof}

Now we come to the main result of this section.
For fields, a related but different presentation is given in \cite{HutchinsonTao:AugmIdeal}.

\begin{theorem}
\label{thm:KhatIsK}
Let $A$ be either a field or a local ring whose residue field has at least $4$ elements.
Then the homomorphism of graded $\Z[A^*]$-algebras
$$\hat{K}^{MW}_*(A) \to K^{MW}_*(A): [a] \mapsto [a]$$
induces an isomorphism for all $n\geq 2$
$$\hat{K}^{MW}_n(A) \stackrel{\cong}{\longrightarrow} K^{MW}_n(A).$$
\end{theorem}

\begin{proof}
The theorem follows from Propositions \ref{prop:hatKMWisTensGWV} and \ref{prop:KMWisTensGWV}.
\end{proof}

\begin{proposition}
\label{prop:KMWnSnPresentation}
Let $A$ be either a field or a local ring with residue field cardinality at least $4$.
Let $n\geq 1$ an integer.
Then for $a_i\in A^*$ $(1\leq i \leq n)$ and $\lambda_i\in A^*$ $(1\leq i \leq n)$ with $\bar{\lambda}_i\neq \bar{\lambda}_j$ for $i\neq j$ ($\bar{\lambda}$ denotes reduction of $\lambda$ modulo the maximal ideal), the following relation holds in $\hat{K}^{MW}_*(A)$:
$$
\renewcommand\arraystretch{2}
\begin{array}{rl}
& [\lambda_1 a_1,\dots, \lambda_n a_n] - [a_1,\dots,a_n]\\
= &
\sum_{i=1}^n \eps^{i+n}\cdot \langle a_i\rangle \cdot [(\lambda_1 -\lambda_i)a_1,\dots,\widehat{(\lambda_i-\lambda_i)a_i},\dots (\lambda_n-\lambda_i)a_n,\lambda_i].
\end{array}
$$
\end{proposition}

\begin{proof}
We will prove the statement by induction on $n$.
For $n=1$ this is Lemma \ref{Hatlem:basicKMWformulas} (\ref{Hatlem:basicKMWformulas:1}).
For $n\geq 2$ we have
$$
\renewcommand\arraystretch{2}
\begin{array}{cl}
 &
\sum_{i=1}^{n-1} \eps^{i+n}\cdot \langle a_i\rangle \cdot [(\lambda_1 -\lambda_i)a_1,\dots,\widehat{(\lambda_i-\lambda_i)a_i},\dots, (\lambda_n-\lambda_i)a_n,\lambda_i]\\
\stackrel{(1)}{=} &
\sum_{i=1}^{n-1} \eps^{i+n}\cdot \langle a_i\rangle \cdot [(\lambda_1 -\lambda_i)a_1,\dots,\widehat{(\lambda_i-\lambda_i)a_i},\dots, \lambda_n a_n,\lambda_i]\\
 &
+ \sum_{i=1}^{n-1} \eps^{i+n}  \langle a_i\rangle   \langle \lambda_na_n\rangle  [(\lambda_1 -\lambda_i)a_1,\dots,
\widehat{\phantom{12}},
\dots, (\lambda_{n-1}-\lambda_i)a_{n-1},1-\lambda_i/\lambda_n,\lambda_i]\\
\stackrel{(2)}{=} &
\left(\sum_{i=1}^{n-1} \eps^{i+n-1}  \langle a_i\rangle   [(\lambda_1 -\lambda_i)a_1,\dots,\widehat{(\lambda_i-\lambda_i)a_i},\dots, (\lambda_{n-1}-\lambda_i) a_{n-1},\lambda_i]\right) [\lambda_na_n]\\
 &
+ \langle \lambda_na_n\rangle \sum_{i=1}^{n-1} \eps^{i+n}  \langle a_i\rangle    [(\lambda_1 -\lambda_i)a_1,\dots,
\widehat{\phantom{12}},
\dots, (\lambda_{n-1}-\lambda_i)a_{n-1},1-\lambda_i/\lambda_n,\lambda_n]\\
\stackrel{(3)}{=} &
[\lambda_1a_1,\dots,\lambda_na_n] - [a_1,\dots,a_{n-1},\lambda_na_n]\\
& +
\eps \langle \lambda_na_n\rangle  
\sum_{i=1}^{n-1} \eps^{i+n-1}  \langle a_i\rangle   [(\frac{\lambda_i}{\lambda_n}-\frac{\lambda_1}{\lambda_n})a_1,\dots,
\widehat{\phantom{12}},
\dots, (\frac{\lambda_i}{\lambda_n}-\frac{\lambda_{n-1}}{\lambda_n})a_{n-1},1-\frac{\lambda_i}{\lambda_n},\lambda_n] \\
\stackrel{(4)}{=} &
[\lambda_1a_1,\dots,\lambda_na_n] - [a_1,\dots,a_n] - \langle a_n\rangle [a_1,\dots,a_{n-1},\lambda_n]\\
&+
\eps \langle \lambda_na_n\rangle \left( 
[(1-\frac{\lambda_1}{\lambda_n})a_1,\dots,(1-\frac{\lambda_{n-1}}{\lambda_n})a_{n-1},\lambda_n]
-[a_1,\dots,a_{n-1},\lambda_n]\right)\\
\stackrel{(5)}{=} &
[\lambda_1a_1,\dots,\lambda_na_n] - [a_1,\dots,a_n] - \langle a_n\rangle [a_1,\dots,a_{n-1},\lambda_n]\\
&-
 \langle a_n\rangle  
[(\lambda_1 - \lambda_n)a_1,\dots,(\lambda_{n-1} - \lambda_n)a_{n-1},\lambda_n]
+ \langle a_n\rangle [a_1,\dots,a_{n-1},\lambda_n]\\
= & 
[\lambda_1a_1,\dots,\lambda_na_n] - [a_1,\dots,a_n] - 
 \langle a_n\rangle [(\lambda_1 - \lambda_n)a_1,\dots,(\lambda_{n-1} - \lambda_n)a_{n-1},\lambda_n].
\end{array}
$$
Here, equation $(1)$ follows from 
$$[(\lambda_n-\lambda_i)a_n] = [(1-\lambda_i/\lambda_n)\cdot \lambda_n a_n]
= [\lambda_n a_n] + \langle \lambda_n a_n\rangle [1-\lambda_i/\lambda_n],$$
equation $(2)$ follows from
$[\lambda_i]=[\lambda_n\frac{\lambda_i}{\lambda_n}] = [\lambda_n] + \langle \lambda_n\rangle [\frac{\lambda_i}{\lambda_n}]$ together with the Steinberg relation which yields
$$[1 - \lambda_i/\lambda_n,\lambda_i] = [1 - \lambda_i/\lambda_n,\lambda_n].$$
Equation $(3)$ follows from the induction hypothesis and
$$[-a/\lambda,\lambda] = [a,\lambda]+\langle a\rangle [-1/\lambda,\lambda] = 
[a,\lambda]+\langle -a/\lambda\rangle [-\lambda,\lambda] = [a,\lambda].$$
Equation $(4)$ follows from the induction hypothesis.
Equation $(5)$ follows from $[-a/\lambda,\lambda] = [a,\lambda]$ and 
$\langle -1\rangle[a,\lambda]=\langle -1\rangle [(-\lambda)(-\frac{a}{\lambda})][\lambda] = \langle -1\rangle ([-\lambda] + \langle -\lambda\rangle [-a/\lambda])[\lambda] = \langle \lambda\rangle [-a/\lambda,\lambda]=\langle \lambda \rangle [a,\lambda].$
\end{proof}

\section{The obstruction to further stability}
\label{sec:Obstruction}

The purpose of this section is to prove Theorems \ref{thm:BassConjLocalRings} and \ref{thm:main:SLnLocalA} from the Introduction.

In subsection \ref{subsec:SpSeqLocRings}, $A$ can be any commutative ring, unless otherwise stated.
In the remaining subsections, $A$ will be a local commutative ring with infinite residue field $k$.
In this case, $A$ has many units, its stable rank is $\sr(A)=1$, and we write
$\bar{a}\in k$ for the reduction of $a\in A$ modulo the maximal ideal in $A$.

\subsection{Multiplicative properties of the spectral sequence}
\label{subsec:SpSeqLocRings}

Let $A$ be a commutative ring.
Recall from Section \ref{sec:Stability} the complexes $C(A^n)$ with $C_r(A^n)$ the free abelian  group generated by the set $U_r(A^n)$ of left invertible $n\times r$-matrices with entries in $A$.
Matrix multiplication makes $C(A^n)$ into a complex of left $GL_n(A)$-modules.
Recall also the spectral sequence 
\begin{equation}
\label{eqn:SpSeqLocalA}
E^1_{p,q}(A^n) = \Tor^{GL_n}_{p}(\Z[A^*],C_q(A^n)) \Rightarrow H_{p+q}(\Z[A^*] \stackrel{L}{\otimes}_{GL_n}C(A^n))
\end{equation}
with differential $d^r$ of bidegree $(r-1,-r)$.
This is the spectral sequence $E^1_{p,q}(A^n) \Rightarrow G_{p+q}(A^n)$ of the exact couple 
\begin{equation}
\label{eqn:ExCouple}
\xymatrix{
\bigoplus_{p,q}E^r_{p,q}(A^n) \ar[r]^i & \bigoplus_{p,q}D^r_{p,q}(A^n) \ar[r]^{\rho} \ar[d]^{j} & \bigoplus_{p+q}G_{p+q}(A^n)\\
& \bigoplus_{p,q}D^r_{p,q}(A^n) \ar[ul]^k \ar[ur]_{\rho} & 
}
\end{equation}
where
$$
\renewcommand\arraystretch{2}
\begin{array}{lcl}
D^1_{p,q}(A^n) &=& H_{p+q}(\Z[A^*]\stackrel{L}{\otimes}_{GL_n}C_{\leq q}(A^n))\\
E^1_{p,q}(A^n) &=& H_{p+q}(\Z[A^*]\stackrel{L}{\otimes}_{GL_n}C_{q}(A^n)[q]) = H_{p}(\Z[A^*]\stackrel{L}{\otimes}_{GL_n}C_{q}(A^n))\\
G_{p+q}(A^n) &=& H_{p+q}(\Z[A^*]\stackrel{L}{\otimes}_{GL_n}C(A^n)).
\end{array}
$$
For $r=1$, the maps $i,j,k$ are the maps of the long exact sequence of homology groups associated with the exact sequence of complexes
$$0 \to C_{\leq q-1}(A^n) \to C_{\leq q}(A^n) \to C_{q}(A^n)[q] \to 0$$
and the map $\rho$ is induced by the inclusion
$C_{\leq q}(A^n) \subset C(A^n)$.
The derived couple is obtained by keeping $G$ and replacing $D$ with $\im(j)$, $E$ with the homology of $(E,k\circ i)$, and $i,j,k,\rho$ with certain induced maps.

The spectral sequence comes with a filtration of the abuttment
\begin{equation}
\label{eqn:LocalSpSeqFiltration}
0 \subset F_{p+q,0}(A^n) \subset F_{p+q-1,1}(A^n) \subset \dots  \subset F_{0,p+q}(A^n) = G_{p+q}(A^n)
\end{equation}
where $F_{p,q}(A^n)$ is the image of $\rho: D_{p,q}(A^n) \to G_{p+q}(A^n)$ (that image is independent of $r$) and an exact sequence
\begin{equation}
\label{eqn:FiltrationExSeq}
0\to F_{p+1,q-1}(A^n) \to F_{p,q}(A^n) \to  E^{\infty}_{p,q}(A^n) \to 0.
\end{equation}
For large $r$ (depending on $(p,q)$), the map $i$ is zero, $E^r_{p,q}=E^{\infty}_{p,q}$, and the exact sequence (\ref{eqn:FiltrationExSeq}) is the exact sequence
$$0\to D^r_{p+1,q-1}(A^n) \to D^r_{p,q}(A^n) \to  E^r_{p,q}(A^n) \to 0.$$
\vspace{1ex}

In this section, we are interested in the degree $n$-part of the spectral sequence $E(A^n)$, and we set
$$
S_n(A) = H_n(\Z[A^*]\stackrel{L}{\otimes}_{GL_n}C(A^n))\hspace{3ex}(=G_n(A^n))
.$$
For instance, we have 
$$\renewcommand\arraystretch{2}
\begin{array}{ll}
S_0(A)=  \Z[A^*],&   S_1(A)=  I[A^*]\\
\end{array}
$$
since $\Z[GL_0(A)]=\Z$, and $C(A^1)$ is the augmentation complex $\eps:\Z[A^*] \to \Z$ with $\Z$ placed in degree $0$.

\begin{remark}
If $A$ is a local ring with infinite residue field, we
let $Z_n(A^n)$ be the $GL_n(A)$-module
$$Z_n(A^n)=\ker(d_n:C_n(A^n) \to C_{n-1}(A^n)).$$
From Lemma \ref{lem:Acyclicity} or Lemma \ref{lem:QuillenAcyclicity} below, the canonical map 
$Z_n(A^n)[n] \to C(A^n)$ is a quasi-isomorphism of complexes of $GL_n(A)$-modules, and thus, 
\begin{equation}
\label{eqn:DfnSn}
S_n(A)=\Z[A^*]{\otimes}_{GL_n}Z_n(A^n) = \left\{
\renewcommand\arraystretch{1.5}
\begin{array}{ll}
\Z{\otimes}_{SL_n}Z_n(A^n)= SL_n\backslash Z_n(A^n),& n\geq 1\\
\Z[A^*], & n=0.
\end{array}
\right.
\end{equation}
\end{remark}
\vspace{2ex}

Let $A$ be a commutative ring.
Since the complex $C(A^n)$ is concentrated in degrees between $0$ and $n$, the spectral sequence $E(A^n)$ has two edge maps.
Set 
\begin{equation}
\label{eqn:DfnBn}
\renewcommand\arraystretch{2}
\begin{array}{rl}
\B_n(A) &=H_n(\Z[A^*]\stackrel{L}{\otimes}_{GL_n}\Z) = \left\{
\renewcommand\arraystretch{1.5}
\begin{array}{ll}
H_n(SL_nA,\Z),& n\geq 1\\
\Z[A^*], & n=0
\end{array}
\right.\\
& (= D^1_{n,0}(A^n)=E^1_{n,0}(A^n)).
\end{array}
\end{equation}
Then the {\em incoming edge map} is the map $B_n(A)\to S_n(A)$ induced by the inclusion $\Z = C_0(A^n) \subset C(A^n)$.
The {\em outgoing edge map} is the map 
$$S_n(A) \to H_n(\Z[A^*]\stackrel{L}{\otimes}_{GL_n}\Z[GL_n][n])=\Z[A^*]$$
induced by the projection $C(A^n) \to C_n(A^n)[n]=\Z[GL_n(A)][n]$.

\begin{remark}
Let $A$ be a local ring with infinite residue field.
In the description (\ref{eqn:DfnSn}) of $S_n(A)$, the outgoing edge map is the determinant map
$$\det: S_n(A) \to \Z[A^*]:z\mapsto \det(z)$$
sending the class of a cycle $z=\sum_in_i\cdot [\alpha_i]\in Z_n(A^n)$, $\alpha_i\in GL_n(A)$, $n_i\in \Z$, to its determinant 
$$\det(z) = \sum_i n_i\langle \det \alpha_i\rangle \hspace{2ex}\in\hspace{1ex} \Z[A^*].$$
\end{remark}

Let $A$ be a commutative ring.
We will need some multiplicative properties of the spectral sequence which we explain now.
For a group $G$, denote by $E(G)$ the standard contractible simplicial set with free $G$-action on the right, and by $\Z E(G)$ the chain complex associated with the free simplicial abelian group generated by $E(G)$.
So, $\Z E(G)$ is the standard free $\Z[G]$-resolution of the trivial $G$-module $\Z$ with $\Z E_q(G) = \Z[G^{q+1}]$ and differential $d_q=\sum_{i=0}^q(-1)^i\delta^i$ defined by $\delta^i(g_0,...,g_q) =(g_0,...,\hat{g_i},...g_q)$. 
The group $G$ acts from the right on $\Z E(G)$ by the formula $(g_0,...,g_q)g=(g_0g,...,g_qg)$.
Recall that the shuffle map
$$\nabla:\Z E(G_1)  \otimes \Z E(G_2) \to \Z E(G_1\times G_2)$$
is a quasi-isomorphism and a lax monoidal transformation and thus makes the usual unit and associativity diagrams commute; see for instance \cite[\S 2.2 and 2.3]{ShipSchwede:Equivalences}.

Let $\Lambda$ be a commutative ring which is flat over $\Z$.
Let $G$ be a group equipped with a group homomorphism $G \to \Lambda^*$ to the group of units of $\Lambda$, that is, a ring homomorphism $\Z[G] \to \Lambda$.
We consider $\Lambda$ as a right $G$-module via this ring homomorphism.
We let $G$ act on $\Lambda E(G)=\Lambda \otimes \Z E(G)$ via the formula
$(\lambda \otimes (g_0,...,g_q))\cdot g = \lambda g \otimes (g_0g,...,g_qg)$.
Then $\Lambda E(G)$ is a free resolution of $\Lambda$ in the category of $\Lambda[G]$-modules.
Let $M$ be a bounded complex of $G$-modules.
Then $\Lambda \otimes M$ is a bounded complex of $\Lambda[G]$-modules.
Since $\Lambda$ is flat over $\Z$, the canonical map
$$\Lambda \stackrel{L}{\otimes}_{\Lambda G} (\Lambda \otimes M) \longrightarrow 
\Lambda \stackrel{L}{\otimes}_{G}  M$$
induced by the isomorphism 
$$\Lambda {\otimes}_{\Lambda G} (\Lambda \otimes M) \stackrel{\cong}{\longrightarrow} \Lambda {\otimes}_{G} M:  \lambda_1 \otimes (\lambda_2 \otimes m)  \mapsto \lambda_1\lambda_2 \otimes m$$
is a quasi-isomorphism.
This can be checked using a $\Z[G]$-projective resolution of $M$.
In particular, 
$$\Lambda E(G)\otimes_{\Lambda[G]}(\Lambda \otimes M)\cong \Lambda E(G)\otimes_{G} M$$
represents $\Lambda\stackrel{L}{\otimes}_GM$.

Let $G_1,G_2$ be groups equipped with group homomorphisms $G_1,G_2\to \Lambda^*$.
Since $\Lambda$ is commutative, we obtain a group homomorphism $G_1\times G_2\to \Lambda^*$.
Let $M_1,M_2$ be bounded complexes of $G_1$ and $G_2$-modules, respectively.
Then $M_1\otimes M_2$ is a complex of $G_1\times G_2$-modules.
The cross product
$$\times: \Tor_p^{G_1}(\Lambda,M_1) \otimes_{\Lambda}\Tor_q^{G_2}(\Lambda,M_2)
\longrightarrow \Tor_{p+q}^{G_1\times G_2}(\Lambda,M_1\otimes M_2)$$
is the map on homology induced by the map of complexes
$$
\renewcommand\arraystretch{2}
\begin{array}{cl}
& \left\{ \Lambda E(G_1)\otimes_{G_1}M_1\right\} \otimes_{\Lambda} \left\{ \Lambda E(G_2)\otimes_{G_2}M_2\right\}\\
\cong & \left\{ \Lambda E(G_1)\otimes_{\Lambda}\Lambda E(G_2) \right\} \otimes_{G_1\times G_2}(M_1\otimes M_2)\\
\stackrel{\nabla \otimes 1}{\longrightarrow}&
\Lambda E(G_1\times G_2) \otimes_{G_1\times G_2}(M_1\otimes M_2).
\end{array}
$$
Since the shuffle map is unital and associative, so is the cross product.
\vspace{1ex}

We apply the previous considerations to $\Lambda = \Z[A^*]$, $G_1= GL_m(A)$, $G_2=GL_n(A)$, $M_1=C(A^m)$, and $M_2=C(A^n)$.
The cross product together with the
$GL_m(A)\times GL_n(A)$-equivariant map of complexes \cite[\S 3]{SuslinNesterenko}
\begin{equation}
\label{eqn:Cprods}
C(A^m)\otimes_{\Z}C(A^n) \to C(A^{m+n}):x\otimes y\mapsto (x,y)
\end{equation}
given on basis elements by concatenation of sequences,
then defines a product
$$S_m(A) \otimes_{A^*}S_n(A) \to S_{m+n}(A)$$ wich makes 
$$S(A) = \bigoplus_{n\geq 0}S_n(A)$$
into an associative and unital $\Z[A^*]$-algebra.

\begin{remark}
Let $A$ be a local ring with infinite residue field.
In terms of cycles (\ref{eqn:DfnSn}), the product is given by
$$\left(\sum_im_i[\alpha_i]\right) \cdot \left(\sum_j n_j[\beta_j]\right) = \sum_{i,j}m_in_j\left[\begin{smallmatrix}\alpha_i & 0 \\ 0 & \beta_j\end{smallmatrix}\right]$$
where $\alpha_i\in GL_m(A)$ and $\beta_j\in GL_n(A)$, $m_i,n_j\in \Z$.
In particular, the outgoing edge map $\det:S(A) \to \Z[A^*]$ is a ring homomorphism.
\end{remark}

Let $A$ be a commutative ring.
The map of complexes (\ref{eqn:Cprods}) restricts to maps 
$$C_{\leq r}(A^m)\otimes C_{\leq s} (A^n) \to C_{\leq r+s}(A^{m+n})$$
which, together with the cross products, define pairings
$$D^1_{m-r,r}(A^m) \otimes_{A^*} D^1_{n-s,s}(A^n) \to D^1_{m+n-r-s,r+s}(A^{m+n})$$
that are suitably associative and unital.
In particular, 
$$\B(A)= \bigoplus_{n\geq 0}\B_n(A) = \bigoplus_{n\geq 0}D^1_{n,0}(A^n)$$
is a unital and associative $\Z[A^*]$-algebra, and for all $r\geq 0$ the graded $A^*$-module
$$\bigoplus_{n\geq 0} D^1_{n-r,r}(A^n)$$
is a graded $\B(A)$ bimodule.
The map $\B(A) \to S(A)$ is a map of graded $\Z[A^*]$-algebras,
and the maps
$$\B(A) \to \bigoplus_{n\geq 0} D^1_{n-r,r}(A^n) \twoheadrightarrow \bigoplus_{n\geq 0} F_{n-r,r}(A^n) \hookrightarrow S(A)$$
are $\B(A)$-bimodule maps.
Since all maps in (\ref{eqn:ExCouple}) are $\B(A)$-bimodule maps, the resulting spectral sequence 
$$\bigoplus_{n\geq 0} E^1_{p,q}(A^n) \Rightarrow \bigoplus_{n\geq 0}G_{p+q}(A^n)$$
is a spectral of $\B(A)$-bimodules.
\vspace{1ex}

For a group $G$ and a subgroup $H\leq G$, write $C(G,H)$ for the complex
of $G$-modules $\Z[G/H] \to \Z: gH \mapsto 1$ with $\Z$ placed in degree $0$.
By Shapiro's Lemma, we have a canonical isomorphism
$H_n(G,H;\Z)\cong H_n(G,C(G,H))$.

For $n\geq 2$, the inclusion $SL_{n-1}(A) \subset \Aff^{SL}_{1,n-1}(A)$ defines a
map of complexes of $SL_n(A)$-modules
\begin{equation}
\label{eqn:PreCanMapToSn}
C(SL_n(A),SL_{n-1}(A)) \to C(SL_nA,\Aff^{SL}_{1,n-1}A) = C_{\leq 1}(A^n).
\end{equation}

\begin{lemma}
\label{Lem:DandF}
Let $A$ be a commutative ring with many units.
Let $2\leq n \leq n_0$ and $\sigma\in \Z[A^*]$ as in $(\ast)$.
\begin{enumerate}
\item
\label{Lem:DandF:1}
The map (\ref{eqn:PreCanMapToSn}) induces an isomorphism of $A^*$-modules
$$H_n(SL_nA,SL_{n-1}A) \stackrel{\cong}{\longrightarrow} \sigma^{-1}D^1_{n-1,1}(A).$$
\item
\label{Lem:DandF:2}
The natural surjetion $D^1_{n-1,1}(A) \to F_{n-1,1}(A)$ induces an isomorphism $A^*$-modules
$$\sigma^{-1}D^1_{n-1,1}(A) \stackrel{\cong}{\longrightarrow} \sigma^{-1}F_{n-1,1}(A).$$
\end{enumerate}
\end{lemma}

\begin{proof}
It follows from Theorem \ref{thm:HomologyOfAffineGps} and the five lemma that the map (\ref{eqn:PreCanMapToSn}) induces an isomorphism 
$$H_p(SL_n(A),SL_{n-1}(A)) \to \sigma^{-1}H_p(SL_nA,\Aff^{SL}_{1,n-1}A)$$
for all $p\leq n_0$.
This proves (\ref{Lem:DandF:1}).

For part (\ref{Lem:DandF:2}), consider the exact sequence
$$D^1_{n,0} \stackrel{j}{\to} D^1_{n-1,1} \stackrel{k}{\to} E^1_{n-1,1} \stackrel{i}{\to} D^1_{n-1,0}$$
and the isomorphism $k:D^1_{n-1,0} \cong E^1_{n-1,0}$.
By Lemma \ref{lem:NS_lemma_2.4}, the differential $d=ki:\sigma^{-1}E^1_{n-1,2} \to \sigma^{-1}E^1_{n-1,1}$ is zero.
Hence, $\sigma^{-1}E^2_{n-1,1}$ is the kernel of the map 
$i:\sigma^{-1}E^1_{n-1,1} \to \sigma^{-1}D^1_{n-1,0}$.
The exact sequence above induces the exact sequence
$0\to \im(j) \to D^1_{n-1,1} \to \ker(d) \to 0$
which, after localization at $\sigma$, is
$0\to \sigma^{-1}D^2_{n,0} \to \sigma^{-1}D^1_{n-1,1} \to \sigma^{-1}E^2_{n-1,1} \to 0$.
By Proposition \ref{prop:NS_Prop_2.6}, we have $\sigma^{-1}D^2_{n,0}=\sigma^{-1}F_{n,0}$ and $\sigma^{-1}E^2_{n-1,1}=\sigma^{-1}E^{\infty}_{n-1,1}$.
Now, the last exact sequence maps to 
$0 \to \sigma^{-1}F_{n,0} \to \sigma^{-1}F_{n-1,1}\to \sigma^{-1}E^{\infty}_{n-1,1} \to 0$.
By the five lemma, we are done.
\end{proof}

\subsection{Presentation and decomposability}

\begin{center}
{\em In this subsection, $A$ is a commutative local ring with infinite residue field.}
\end{center}

In order to obtain a presentation of $S_n(A)$, we need to recall from \cite{SuslinNesterenko} the definition of the complex of $GL_n(A)$-modules $\tilde{C}(A^n)$.
A sequence $(v_1,...,v_r)$ of $r$ vectors in $A^n$ is said to be in general position if any $\min(r,n)$ of the vectors $v_1,...,v_r$ span a free submodule of rank $\min(r,n)$.
A {\em rank $r$ general position sequence} in $A^n$ is a sequence $(v_1,...,v_r)$ of $r$ vectors in $A^n$ which are in general position.
Note that $(v_1,...,v_r)$ is in general position in $A^n$ if and only if 
their reduction $(\bar{v}_1,...,\bar{v}_r)$ modulo the maximal ideal of $A$ is in general position in $k^n$.
This is because a set of vectors $v_1,...,v_s$ spans a free submodule of rank $s$ if and only if the matrix $(v_1,...,v_s)$ has a left inverse.

Let $V=(v_1,...,v_r)$ be a general position sequence, we call a vector $w\in A^n$ {\em transversal to $V$} if
$(V,w) = (v_1,...,v_r,w)$ is also in general position.

\begin{lemma}
\label{lem:Transversality}
Let $(A,m,k)$ be a local ring with infinite residue field $k$.
Let $V^1,...,V^s$ be a finite set of rank $r$ general position sequences in $A^n$.
Then there is an element $e\in A^n$ which is transversal to $V^1,...,V^s$.
\end{lemma}

\begin{proof}
Since a set of vectors is in general position in $A^n$ if and only if it is modulo $m$, we can assume $A=k$ is an infinite field.
Let $\V$ be the union of the vectors occurring in the sequences $V^1$,...,$V^s$.
Let $r_0=\min(r,n)-1$.
Each subset of $\V$ of cardinality $r_0$ generates
a $k$-linear subspace of $k^n$ of dimension $\leq r_0$.
Since $k$ is infinite and  $r_0<n$  there is $e\in A^n$ 
which is not in any of these finitely many subspaces.
Any such $e$ is transversal to $V^1,...,V^s$.
\end{proof}

Let $\tilde{U}_r(A^n)$ be the set of sequences $(v_1,...,v_r)$ of vectors
$v_1,...,v_r$ which are in general position in $A^n$.
For integers $r$, $n$ with $n\geq 0$, let $\tilde{C}_r(A^n)=\Z[\tilde{U}_r^n(A)]$ be the free abelian group with basis the rank $r$ general position sequences 
$(v_1,...,v_r)$ in $A^n$.
For instance, $\tilde{C}_r(A^n) = 0$ for $r<0$, $\tilde{C}_0(A^n)=\Z$ generated by the empty sequence, and $\tilde{C}_n(A^n)=\Z[GL_n(A)]$.
For $i=1,...,r$ one has maps $\delta^i_r:\tilde{C}_r^n \to \tilde{C}_{r-1}^n$ defined on basis elements by $\delta^i_r(v_1,...,v_r) = (v_1,...,\hat{v}_i,...,v_r)$ omitting the $i$-th entry.
We set $d_r=\sum_{i=1}^r(-1)^{i-1}\delta_r^i: \tilde{C}_r \to \tilde{C}_{r-1}$, and it is standard that $d_rd_{r+1}=0$.
This defines the chain complex $\tilde{C}(A^n)$.
The group $GL_n(A)$ acts on this complex by left matrix multiplication.

\begin{lemma}
\label{lem:QuillenAcyclicity}
Let $A$ be a local ring with infinite residue field.
Then the complex $\tilde{C}(A^n)$ is acyclic, that is, for all $r\in \Z$ we have
$$H_r(\tilde{C}(A^n))= 0.$$
\end{lemma}

\begin{proof}
Let $\xi=\sum_{i=1}^s n_iV^i \in \tilde{C}_r(A^n)$ where $V^i$ are rank $r$ general position sequences.
By Lemma \ref{lem:Transversality} we can choose $e\in A^n$ which is transversal to $V^1,...,V^s$.
Set $(\xi,e) = \sum_{i=1}^s n_i(V^i,e) \in \tilde{C}_{r+1}(A^n)$.
If $d_r\xi=0$ 
then 
$$
\renewcommand\arraystretch{2}
\begin{array}{ll}
d_{r+1}(\xi,e) & = \sum_{j=1}^{r+1}(-1)^{j-1}\delta_{r+1}^j(\xi,e) = \sum_{j=1}^{r}(-1)^{j-1}\delta_{r+1}^j(\xi,e) +(-1)^{r}\delta_{r+1}^{r+1}(\xi,e)\\ 
& = (d_r\xi,e) + (-1)^{r}\delta_{r+1}^{r+1}(\xi,e) = (-1)^{r}\xi.
\end{array}$$
This shows that $\xi$ is a boundary.
\end{proof}

In the following proposition, we consider the empty symbol $[]$ as a symbol, the unique symbol of length zero.

\begin{proposition}
\label{prop:SnPresentation}
For $n\geq 0$, the $\Z[A^*]$-module $S_n(A)$ has the following presentation.
Generators are the symbols $[a_1,...,a_n]$ with $a_i\in A^*$.
A system of defining relations has the form
$$
\renewcommand\arraystretch{2}
\begin{array}{rl}
& [\lambda_1 a_1,\dots, \lambda_n a_n] - [a_1,\dots,a_n]\\
= &
\sum_{i=1}^n \eps^{i+n}\cdot \langle a_i\rangle \cdot [(\lambda_1 -\lambda_i)a_1,\dots,\widehat{(\lambda_i-\lambda_i)a_i},\dots (\lambda_n-\lambda_i)a_n,\lambda_i]
\end{array}
$$
where $\lambda_i\in A^*$ and $\bar{\lambda}_i\neq \bar{\lambda}_j \in k$ for $i\neq j$ and $\eps = -\langle -1\rangle \in \Z[A^*]$.
\end{proposition}

\begin{proof}
For $n=0$, the module given by the presentation is generated by the empty symbol $[]$ subject to the trivial relation $[]-[]=0$.
Hence, this module is $\Z[A^*]$ which is $S_0(A)$.
For $n=1$, the module given by the presentation is generated by symbols $[a]$ for $a\in A^*$ subject to the relation $[\lambda a]-[a]=\langle a\rangle[\lambda]$ for $a,\lambda \in A^*$.
By Lemma \ref{lem:IAPresentation}, this module is the augmentation ideal $I[A^*]$ which is $S_1(A)$.

For $n\geq 2$, the proof is the same as in \cite[Theorem 3.3]{HutchinsonTao:Stability}.
We have 
$$S_n(A)=H_0(\Z[A^*]\stackrel{L}{\otimes}_{GL_n}Z_n(A^n)) =H_0(SL_n,Z_n(A^n))$$ since $\det:GL_n(A) \to A^*$ is surjective with kernel $SL_n(A)$ (Shapiro's Lemma).
In view of Lemma \ref{lem:QuillenAcyclicity} we have an exact sequence of $GL_n(A)$-modules
$$\tilde{C}_{n+2}(A^n) \stackrel{d_{n+2}}{\longrightarrow} \tilde{C}_{n+1}(A^n) \to Z_n(A^n).$$
Taking $SL_n(A)$-coinvariants yields a presentation of $S_n(A)$ as $A^*=GL_n(A)/SL_n(A)$-module.

As $GL_n(A)$-sets we have equalities
$$\tilde{U}_{n+1}(A^n) = \bigsqcup_{a_1,...,a_n\in A^*} GL_n(A)\cdot (e_1,...,e_n,a)$$
where $a=a_1e_n+\cdots +a_ne_n$ and
$$\tilde{U}_{n+2}(A^n) = \bigsqcup_{\stackrel{a_1,...,a_n}{b_1,...,b_n}\in A^*} GL_n(A)\cdot (e_1,...,e_n,a,b)$$
where $a=a_1e_n+\cdots +a_ne_n$, $b=b_1e_n+\cdots +b_ne_n$ with $\bar{a}_i\bar{b}_j\neq \bar{a}_j\bar{b}_i \in k$ for $i\neq j$.
In other words $b_i=\lambda_ia_i$ for some $\lambda_i\in A^*$ such that 
$\bar{\lambda}_i\neq \bar{\lambda}_j$.

For $i=1,...,n$ we have
$$d^i_{n+2}(e_1,...,e_n,a,b) = (e_1,...,\widehat{e_i},...,e_n,a,b) = 
M^i(a)\cdot(e_1,...,e_n,c)$$
where $M^i(a)$ is the matrix $(e_1,...,\widehat{e_i},...,e_n,a)$
and 
$$c=M^i(a)^{-1}b = ((\lambda_1 -\lambda_i)a_1,\dots,\widehat{(\lambda_i-\lambda_i)a_i},\dots (\lambda_n-\lambda_i)a_n,\lambda_i).$$
Since $\det M^i(a) = (-1)^{n+i}a_i$ we have the following presentation for $S_n(A)$ as $\Z[A^*]$-module.
Generators are the symbols $[a] = [a_1,...,a_n]$ with $a=a_1e_1+\cdots + a_ne_n$ where $a_i\in A^*$,
the symbol $[a]$ representing the $SL_n(A)$-orbit of $(e_1,...,e_n,a)$.
The relations are 
$$0=d_{n+2}(e_1,...,e_n,a,b) =
(-1)^{n}[b] - (-1)^n[a] + \sum_{i=1}^n(-1)^{i-1}\langle (-1)^{n+i}a_i\rangle [M^i(a)^{-1}b]$$
where $b_i=\lambda_ia_i$ with $\bar{\lambda}_i\neq \bar{\lambda}_j$ for $i\neq j$.
This can be written as
$$[b] - [a] = \sum_{i=1}^n\eps^{i+n}\langle a_i\rangle [M^i(a)^{-1}b].$$
\end{proof}

\begin{remark}
\label{rmk:detComp}
We compute the determinant of $[a_1,...,a_n]$ as
$$
\renewcommand\arraystretch{1.5}
\begin{array}{rcl}
\det [a_1,...,a_n] & = & (-1)^{n}\langle 1\rangle  + \sum_{i=1}^{n}(-1)^{i+1}\langle \det(e_1,...,\hat{e_i},...,e_n,a)\rangle\\
&=& (-1)^{n}\langle 1\rangle + \sum_{i=1}^{n}(-1)^{i+1}\langle (-1)^{n+i}a_i\rangle
\end{array}
$$
where $a=a_1e_1+\cdots + a_ne_n$.
For instance, 
$$
\renewcommand\arraystretch{1.5}
\begin{array}{lcl}
\det[a,b] & = & \langle -a\rangle - \langle b \rangle +\langle 1\rangle\\
\det [a,b,c] & = & \langle a\rangle - \langle -b\rangle + \langle c\rangle - \langle 1 \rangle.
\end{array}
$$
\end{remark}

Let $S_{\geq 1}(A) = \bigoplus_{n\geq 1}S_n(A) \subset S(A)$ be the ideal of positive degree elements,
and define the graded ideal $S(A)^{dec}$ of decomposable elements as
$$S(A)^{dec}=S_{\geq 1}(A)\cdot S_{\geq 1}(A)\subset S(A).$$
The algebra of indecomposable elements is the quotient
$$S(A)^{ind}=S(A)/S(A)^{dec}.$$

We will show in Proposition \ref{prop:Decomposibility} that $\sigma^{-1}S_n^{ind}(A)=0$ for $n\geq 3$.
The proof requires the complexes $C(V,W)$ from \cite[\S 3]{SuslinNesterenko} which we recall now.
Let $V$ and $W$ be finitely generated free $A$-modules.
Let $U_m(V,W)=U_m(V)\times W^m \subset U_m(V\oplus W)$ be the set of sequences 
$$\left(\begin{matrix}v_1\\w_1\end{matrix}\right),
\left(\begin{matrix}v_2\\w_2\end{matrix}\right),\dots,
\left(\begin{matrix}v_m\\w_m\end{matrix}\right)$$
with $v_i\in V$ and $w_i\in W$ such that
$(v_1,v_2,...,v_m)$ can be completed to a basis of $V$.
Let $C_m(V,W)=\Z[U_m(V,W)]$ be the free abelian group generated by $U_m(V,W)$, and define a differential $d:C_m(V,W) \to C_{m-1}(V,W)$ as in (\ref{eqn:Crdifferential}).
In particular, 
$C(V,W)$ is a subcomplex of $C_m(V\oplus W)$, and $C(V,0)=C(V)$.
Note that $C_m(V,W)=0$ for $m > \rk V$ where $\rk V$ denotes the rank of the free $A$-module $V$.
The group 
$$\Aff(V,W)=\left( \begin{matrix} 
GL(V) & 0 \\ \Hom(V,W) & 1_W
\end{matrix}\right)
$$
acts on $U_m(V,W)$ by multiplication from the left making the complex
$C(V,W)$ into a complex of left $\Aff(V,W)$-modules.
For $\rk V = n$ let 
$$Z_n(V,W)=\ker (d:C_n(V,W) \to C_{n-1}(V,W)).$$

\begin{lemma}
If $n=\rk V$, then the canonical map of complexes of $\Aff(V,W)$-modules
$$Z_n(V,W)[n] \to C(V,W)$$
is a quasi-isomorphism.
\end{lemma}

\begin{proof}
This is \cite[Corollary 3.6]{SuslinNesterenko}.
\end{proof}

The determinant map $\Aff(V,W) \to A^*$ makes 
$\Z[A^*]$ into a right $\Aff(V,W)$-module.
Write $S(V,W)$ for the group 
$$S(V,W) = \Tor_n^{\Aff(V,W)}(\Z[A^*],C(V,W)) = \Z[A^*]\otimes_{\Aff(V,W)}Z_n(V,W).$$

\begin{lemma}
\label{lem:SVWisSV}
Let $V$, $W$ be finitely generated free $A$-modules with $n=\rk V$.
Let $n \leq n_0$ and $\sigma\in \Z[A^*]$ as in $(\ast)$.
Then for $0 \leq p \leq n_0$, the inclusion $C(V)=C(V,0)  \subset C(V,W)$ induces isomorphisms
$$\sigma^{-1} \Tor_p^{GL(V)}(\Z[A^*],C(V)) \stackrel{\cong}{\longrightarrow} 
\sigma^{-1} \Tor_p^{\Aff(V,W)}(\Z[A^*],C(V,W)).$$
In particular, the inclusion $0 \subset W$ induces an isomorphism $A^*$-modules
$$\sigma^{-1}S(V) \stackrel{\cong}{\longrightarrow} \sigma^{-1}S(V,W).$$
\end{lemma}

\begin{proof}
The proof is the same as in \cite[Proposition 3.9]{SuslinNesterenko}
replacing \cite[Theorem 1.11]{SuslinNesterenko} with Theorem \ref{thm:HomologyOfAffineGps}.
\end{proof}

If $n=\rk V$ and $m=\rk W$ , then
concatenation of sequences defines a map
$$
\renewcommand\arraystretch{1.5}
\begin{array}{rclcl}
Z_n(V,W) & \otimes & Z_m(W) & \to & Z_{n+m}(V\oplus W)\\
\left( \begin{smallmatrix} v_1 & \dots & v_n \\ w_1 & \dots & w_n \end{smallmatrix}\right) & \otimes & (z_1,...,z_m) &  \mapsto & \left( \begin{smallmatrix}  v_1 & \dots & v_n & 0 & \dots & 0  \\  w_1 & \dots & w_n & z_1 & \dots & z_m  \end{smallmatrix}\right)
\end{array}$$ 
which induces a well-definied  multiplication 
$$S(V,W) \otimes Z_m(W) \to S(V\oplus W).$$
Similarly, concatenation defines a product
$$
\renewcommand\arraystretch{1.5}
\begin{array}{rclcl}
Z_m(W) & \otimes & S(V,W) &  \to & S(V\oplus W)\\
(z_1,...,z_m) & \otimes & \left( \begin{smallmatrix} v_1 & \dots & v_n \\ w_1 & \dots & w_n \end{smallmatrix}\right) & \mapsto & \left( \begin{smallmatrix} 0 & \dots & 0 & v_1 & \dots & v_n \\ z_1 & \dots & z_m & w_1 & \dots & w_n \end{smallmatrix}\right).
\end{array}$$

\begin{proposition}
\label{prop:ProductSimplification}
Let $V$, $W$ be finitely generated free $A$-modules with $\rk W=m$.
If $\rk V \leq n_0$ and $\sigma\in \Z[A^*]$ as in $(\ast)$, then
the map $p:S(V,W) \to S(V)$ induced by the unique map of $A$-modules $W \to 0$ yields commutative diagrams
$$\xymatrix{
\sigma^{-1} S(V,W) \otimes Z_m(W) \ar[d]_{p\otimes 1}^{\cong} \ar[r]& \sigma^{-1} S(V\oplus W) & Z_m(W)  \otimes \sigma^{-1} S(V,W) \ar[l] \ar[d]^{1 \otimes p}_{\cong} \\
\sigma^{-1} S(V) \otimes Z_m(W) \ar[ur] & & Z_m(W) \otimes \sigma^{-1} S(V) \ar[ul]
}
$$
\end{proposition}

\begin{proof}
This is because the diagrams obviously commutes when the isomorphism 
$p:\sigma^{-1}S(V,W) \to \sigma^{-1}S(V)$ is replaced with its inverse given in Lemma \ref{lem:SVWisSV}.
\end{proof}

\begin{lemma}
\label{lem:IndecGens}
Let $n_0 \geq n\geq 2$ and $\sigma \in \Z[A^*]$ as in $(\ast)$.
Then for $a_1,...,a_n,b\in A^*$ and $1\leq i\leq n$ we have in $\sigma^{-1}S(A)^{ind}$
$$[a_1,...,ba_i,...,a_n]=\langle b\rangle [a_1,...,a_n].$$
\end{lemma}

\begin{proof}
Using Proposition \ref{prop:ProductSimplification}, the proof is the same as in \cite[Theorem 6.2]{HutchinsonTao:Stability} and \cite[Proposition 3.19]{SuslinNesterenko}.
We omit the details; the case $n=3$ is explained in more details in Lemma \ref{lem:S3prods} below.
\end{proof}

We will need a slightly more precise version of Lemma \ref{lem:IndecGens} when $n=3$.

\begin{lemma}
\label{lem:S3prods}
Let $n_0 \geq 3$ and $\sigma \in \Z[A^*]$ as in $(\ast)$.
Then for all $a,b,c,\lambda \in A^*$, the following holds in $\sigma^{-1}S_3(A)$ modulo $S_1(A)\cdot S_2(A)$.
\begin{enumerate}
\item
\label{lem:S3prods:A}
$[\lambda a, b,c] = \langle \lambda \rangle [a,b,c]$
\item
\label{lem:S3prods:B}
$[a, \lambda b,c] = \langle \lambda \rangle [a,b,c] + S_2(A)\cdot [c]$
\item
\label{lem:S3prods:C}
$[a, b, \lambda c] = \langle \lambda \rangle [a,b,c] + S_2(A)\cdot [\lambda]$
\item
\label{lem:S3prods:D}
$[a, b,c] = \langle abc \rangle [1,1,1] + S_2(A)\cdot [c]$.
\end{enumerate}
\end{lemma}

\begin{proof}
To prove (\ref{lem:S3prods:A}), write $x = \lambda a e_1 + b e_2 + c e_3$ and note that
$$
\renewcommand\arraystretch{1.5}
\begin{array}{rcl}
[\lambda a, b, c] - \langle \lambda \rangle [a,b,c] 
&=& d\{(e_1,e_2,e_3,x) - (\lambda e_1,e_2,e_3,x)\} \\
&=& (e_1-\lambda e_1) \cdot d(e_2,e_3,x)\\
&=& -[\lambda] \cdot [b,c] \hspace{6ex}\in S_1\cdot S_2
\end{array}$$
where the last equality follows from Proposition \ref{prop:ProductSimplification}.

To prove (\ref{lem:S3prods:B}), write $x =  a e_1 + \lambda b e_2 + c e_3$, $u = a e_1 + \lambda b e_2$ and note that
$$
\renewcommand\arraystretch{1.5}
\begin{array}{rcl}
[a, \lambda b, c] - \langle \lambda \rangle [a,b,c] 
&=& d\{(e_1,e_2,e_3,x) - (e_1, \lambda e_2,e_3,x)\} \\
&=& (e_2-\lambda e_2) \cdot d(u,e_3,x) - d(e_1,u)(e_2-\lambda e_2) \cdot d(e_3,x)\\
& & +  \{ u (e_2-\lambda e_2) + (e_2-\lambda e_2)u\} \cdot d(e_3,x)\\
&\in & S_1\cdot S_2 + S_2\cdot [c] + S_2\cdot [c]
\end{array}
$$
where the last line follows from Proposition \ref{prop:ProductSimplification}.

To prove (\ref{lem:S3prods:C}), write $x =  a e_1 +  b e_2 + \lambda c e_3$ and note that
$$
\renewcommand\arraystretch{1.5}
\begin{array}{rcl}
[a,  b,\lambda c] - \langle \lambda \rangle [a,b,c] 
&=& d\{(e_1,e_2,e_3,x) - (e_1, e_2, \lambda e_3,x)\} \\
&=& (e_2-e_1) \cdot \{ x (e_3-\lambda e_3) + (e_3-\lambda e_3)x\}\\
&& - \{ (e_1,e_2) + (e_2-e_1)x\} \cdot (e_3-\lambda e_3)\\
& \in & S_1\cdot S_2 + S_2 \cdot [\lambda]
\end{array}
$$
where the last line follows from Proposition \ref{prop:ProductSimplification}.

Finally, to prove (\ref{lem:S3prods:D}), we note that modulo $S_1\cdot S_2$ we have
$$
\renewcommand\arraystretch{1.5}
\begin{array}{rcl}
[a,  b, c] &=&  \langle a \rangle [1,b,c] \\
&=&  \langle ab \rangle [1,1,c] + S_2\cdot [c]\\
&=&  \langle abc \rangle [1,1,1] + S_2\cdot [c]
\end{array}
$$
in view of (\ref{lem:S3prods:A}), (\ref{lem:S3prods:B}) and (\ref{lem:S3prods:C}).
\end{proof}

The rest of the section is devoted to the proof of the following.

\begin{proposition}
\label{prop:Decomposibility}
For $3\leq n \leq n_0$ and $\sigma \in \Z[A^*]$ as in $(\ast)$ we have 
$$\sigma^{-1}S_n(A)^{dec}=\sigma^{-1}S_n(A).$$
\end{proposition}

Keep the hypothesis of Proposition \ref{prop:Decomposibility}.
Let $\Sigma_1(A)$ be the free $\Z[A^*]$-module generated by the set of units $a\in A^*$ with $\bar{a}\neq 1$.
Define $\Sigma(A)$ as the tensor algebra of $\Sigma_1(A)$ over $\Z[A]$ with $\Sigma_1(A)$ placed in degree $1$.
So, $\Sigma_n(A)$ is the free $\Z[A^*]$-module generated by symbols 
$[a_1,...,a_n]$ with $a_i\in A^*$ such that $\bar{a}_i\neq 1$, and multiplication is given by concatenation of symbols.
Similarly, let $\tilde{\Sigma}_1(A)$ be the free $\Z[A^*]$-module generated by the set of all units $a\in A^*$.
Define $\tilde{\Sigma}(A)$ as the tensor algebra of $\tilde{\Sigma}_1(A)$ over $\Z[A]$ with $\tilde{\Sigma}_1(A)$ placed in degree $1$.
Consider the diagram of graded $\Z[A^*]$-module maps
\begin{equation}
\label{eqn:EvenDecomp}
\xymatrix{
\Sigma(A) \ar[r]^L \ar@{=}[d] & \tilde{\Sigma}_{2*}(A) \ar[d]_p \ar[r]^{\Pi} & \Z[A^*][T^2] \ar[d]_q\\
\Sigma(A) \ar[r]^R & S_{2*}(A) \ar[r] & S_{2*}(A)^{ind}
}
\end{equation}
where the maps are defined as follows.
The map $L:\Sigma(A) \to \tilde{\Sigma}_{2*}(A)$ is the $\Z[A^*]$-algebra map
induced by the $\Z[A^*]$-module homomorphism $\Sigma_1(A) \to \tilde{\Sigma}_2(A)$
defined on generators $a\in A^*$ of $\Sigma_1(A)$ by
$$L(a) = \langle -1\rangle [1-a,1] - \langle a\rangle [ 1- a^{-1},a^{-1}] +[1,1].$$
The map 
$R:\Sigma(A) \to S_{2*}(A)$ is the $\Z[A^*]$-algebra homomorphism 
induced by the $\Z[A^*]$-module homomorphism 
$$\Sigma_1(A) \to S_2(A):[a]\mapsto R(a)=[1,a].$$
The map $\Pi:\tilde{\Sigma}_{2*}(A) \to \Z[A^*][T^2]$ is the $\Z[A^*]$-algebra homomorphism which is the even part of 
$$\tilde{\Sigma}(A) \to \Z[A^*][T]: [a_1,...,a_n]\mapsto \langle a_1\cdots a_n\rangle T^n.$$
The middle and right vertical maps are the $\Z[A^*]$-module homomorphisms
$$p:\tilde{\Sigma}_n(A) \to S_n(A):[a_1,...,a_n]\mapsto [a_1,...,a_n]$$
and 
$$q:\Z[A^*]\cdot T^n \to S_n(A)^{ind}:T^n\mapsto [1,...,1].$$

\begin{lemma}
The diagram (\ref{eqn:EvenDecomp}) commutes in degrees $\leq n_0$ after inverting $\sigma\in \Z[A^*]$.
\end{lemma}

\begin{proof}
Commutativity of the right hand square follows from Lemma \ref{lem:IndecGens}.
The proof of the commutativity of the left hand square is the same is in \cite[Lemma 6.6]{HutchinsonTao:Stability} and we omit the details.
\end{proof}

\begin{lemma}
Proposition \ref{prop:Decomposibility} is true for $n\geq 4$ even.
\end{lemma}

\begin{proof}
Set $d=n/2$.
So $d\geq 2$.
The composition of the two lower horizontal maps in diagram (\ref{eqn:EvenDecomp}) 
is zero in degree $d$ since $\Sigma$ is decomposable in degrees $\geq 2$.
The right vertical map in that diagram is surjective in degree $d$ after inverting $\sigma$, by Lemma \ref{lem:IndecGens}.
For any $r\geq 1$, when extending scalars along $\Z[A^*] \to \Z[A^*/A^{r*}]$, the composition of the top two horizontal arrows in the diagram becomes surjective.
For if $a\in A$ is a unit with $\bar{a}^r\neq 1$
then
$1-a^{-r} = a^r-1$ in $A^*/A^{r*}$, and hence
$$\Pi L(a^r)=\langle -1\rangle \langle 1-a^r\rangle - \langle a^r\rangle \langle 1-a^{-r}\rangle\langle a^{-r}\rangle + 1 = 1 \in \Z[A^*/A^{r*}].$$
It follows that $(\Pi\circ L)\otimes_{A^*}\Z[A^*/A^{r*}]$ is surjective in degree $1$ which implies surjectivity in degrees $n_0 \geq n \geq 1$.
We have thus shown that
\begin{equation}
\label{eqn:SindmodA*r}
\sigma^{-1}S_n(A)^{ind}\otimes_{A^*}\Z[A^*/A^{r*}]=0
\end{equation}
for all $r\geq 1$.

Taking the image of the filtration 
(\ref{eqn:LocalSpSeqFiltration})
in $\sigma^{-1}S_n(A)^{ind}$ defines the filtration 
$0 \subset F_{n,0}^{ind} \subset ... \subset F_{0,n}^{ind} = \sigma^{-1}S_n(A)^{ind}$
with quotients $E^{\infty}_{p,q}(A^n)^{ind}=F_{p,q}^{ind}/F_{p+1,q-1}^{ind}$ subquotients of 
$\sigma^{-1}E^1_{p,q}(A^n)$.
Since the edge map $\det: S_n(A) \to E^{\infty}_{0,n}=E^1_{0,n}=\Z[A^*]$ sends the decomposable element $[-1,1]^d$ to $\det([-1,1]^d)=(\det[-1,1])^d = \langle 1\rangle ^d= \langle 1 \rangle$, we have $E^{\infty}_{0,n}(A^n)^{ind}=0$.
Moreover, $E^{\infty}_{1,n-1}(A^n)^{ind}=0$ as $\sigma^{-1}E^1_{1,n-1}(A^n)=0$.
For $2\leq r \leq n$, the $A^*$-module $E^{\infty}_{r,n-r}(A^n)^{ind}$ is a subquotient of 
$\sigma^{-1}E^1_{r,n-r}(A^n)=H_r(SL_rA)$ and hence a $\Z[A^*/A^{r*}]$-module.
Using (\ref{eqn:SindmodA*r}), induction on $r$ shows that 
$E^{\infty}_{r,n-r}(A^n)^{ind}=0$ for all $r$.
Hence, $\sigma^{-1}S_n(A)^{ind}=0$.
\end{proof}

\begin{lemma}
\label{Lem:EndPfDecomp}
Proposition \ref{prop:Decomposibility} is true for $n$ odd.
\end{lemma}

\begin{proof}
Let $n=2d+1$ be odd with $d\geq 1$.
Choose $\lambda_n,\lambda_1,...,\lambda_d \in A^*$ such that $\bar{\lambda}_i\neq \bar{\lambda}_j$ for $i\neq j$, $\bar{\lambda}_n\neq \bar{\lambda}_{i}+\bar{\lambda}_j$ for $i,j=1,...,d$.
This is possible since $A$ has infinite residue field.
For $i=d+1,...,2d$ set $\lambda_i=\lambda_n-\lambda_{n-i}$.
Then $\lambda_i\in A^*$ and $\bar{\lambda}_i\neq \bar{\lambda}_j$ for $1\leq i\neq j \leq n$.
For $1\leq i, j < n$ we have $\lambda_j-\lambda_i = -(\lambda_{n-j}-\lambda_{n-i})$.
Since $n$ is odd, for $i\neq n$ we therefore find 
$$(\lambda_n-\lambda_i)\lambda_i\prod_{j\neq i,n}(\lambda_j-\lambda_i) = -\lambda_{n-i}(\lambda_n-\lambda_{n-i})\prod_{j\neq n-i,n}(\lambda_{j}-\lambda_{n-i}).$$
The equation from Proposition \ref{prop:SnPresentation} with $a_1=...=a_n=1$ together with Lemma \ref{lem:IndecGens} then implies that $-1=0$ in the group $\sigma^{-1}S_n(A)^{ind}$.
Hence this group is zero.
\end{proof}

\subsection{The Steinberg relation and $H_2(SL_2A)$}

\begin{center}
{\em In this subsection, $A$ is a commutative local ring with infinite residue field.}
\end{center}

\begin{definition}
We define $\bar{S}(A)$ as the quotient left $S(A)$-module
$$\bar{S}(A)=S(A)/\left(S(A)\cdot [-1,1]\right).$$
\end{definition}

Recall the canonical maps
\begin{equation}
\label{eqn:CanMapToSn}
H_n(SL_n(A),SL_{n-1}(A)) \to D^1_{n-1,1}(A) \rightarrow F_{n-1,1}(A^{n-1}) \subset S_n(A) \to \bar{S}_n(A).
\end{equation}
all but the last of which were defined in Subsection \ref{subsec:SpSeqLocRings}, and the last is the natural surjection.

\begin{lemma}
\label{lem:RelHisSn}
Let $2\leq n \leq n_0$ and $\sigma \in \Z[A^*]$ as in $(\ast)$.
Then the map (\ref{eqn:CanMapToSn}) induces isomorphisms of $A^*$-modules
$$H_n(SL_n(A),SL_{n-1}(A)) \stackrel{\cong}{\longrightarrow} \sigma^{-1}F_{n-1,1}(A^n) \stackrel{\cong}{\longrightarrow} \sigma^{-1}\bar{S}_n(A).$$
\end{lemma}

\begin{proof}
The first isomorphism was proved in Lemma \ref{Lem:DandF}.

For the second isomorphism, recall from the proof of Proposition \ref{prop:NS_Prop_2.6} the map of complexes $\psi:C(A^{n-2})[-2] \to C(A^{n})$ which 
induces maps of short exact sequences
$$\renewcommand\arraystretch{2}
\begin{array}{rcccccccl}
\sigma^{-1}F_{n-s+3,s-3}(A^{n-2}) &\rightarrowtail & \sigma^{-1}F_{n-s+2,s-2}(A^{n-2}) & \twoheadrightarrow & \sigma^{-1}E^{\infty}_{n-s+2,s-2}(A^{n-2})  \\
\downarrow \psi && \downarrow \psi && \downarrow \psi \\
 \sigma^{-1}F_{n-s+1,s-1}(A^{n}) & \rightarrowtail & \sigma^{-1}F_{n-s,s}(A^{n}) & \twoheadrightarrow & \sigma^{-1}E^{\infty}_{n-s,s}(A^{n}).
\end{array}$$
For $s=2$, the right vertical map is an isomorphism and the upper left corner is $0$.
It follows that the the middle map is injective with cokernel the lower left corner $\sigma^{-1}F_{n-1,1}(A^n)$.
Since the right vertical map is an isomorphism for $s\geq 2$, it follows by induction on $s$ that $\psi:\sigma^{-1}F_{n-s+2,s-2}(A^{n-2}) \to \sigma^{-1}F_{n-s,s}(A^{n})$ is injective with cokernel $\sigma^{-1}F_{n-1,1}(A^n)$. 
The case $s=n$ is the isomorphism $\sigma^{-1}F_{n-1,1}(A^n) \to \sigma^{-1}S_n/\psi(S_{n-2})$.
Since the map $\psi$ is right multiplication with $[-1,1]$, we have
$\psi(S_{n-2})=S_{n-2}\cdot [-1,1]$ and we are done.
\end{proof}

\begin{remark}
\label{rem:NS:3.14}
\cite[Remark 3.14]{SuslinNesterenko}.
Let $n\geq 1$.
We describe a standard procedure which allows us to represent an arbitrary element in ${S}_n(A)$ as a sum of generators $[a_1,...,a_n]$.
Take an arbitrary cycle 
$$x = \sum_i n_i(\alpha_i) \in Z_n(A^n)$$
with $n_i\in \Z[A^*]$ and $\alpha_i\in GL_n(A)$
and find a vector $v\in A^n$ in general position with the column vectors of $\alpha_i$ for all $i$.
Then 
$$x = (-1)^n d(x,v) = (-1)^n\sum n_id(\alpha_i,v).$$
We have
$$(\alpha_i,v) = \alpha_i\cdot (e_1,...,e_n,\alpha_i^{-1}v) \equiv \langle \alpha_i\rangle \cdot (e_1,...,e_n,\alpha_i^{-1}v) \mod SL_n(A)$$
where $\langle \alpha_i\rangle \in \Z[A^*]$ denotes the determinant of $\alpha_i$.
Hence,
$$x = \sum_i n_i(\alpha_i) = (-1)^n\sum_i n_i\langle \alpha_i \rangle [\alpha_i^{-1}v].$$
\end{remark}

\begin{proposition}
\label{prop:SToKMW}
The map 
$$T_n: S_n(A) \to \hat{K}^{MW}_n(A):[a_1,...,a_n]\mapsto [a_1,...,a_n]$$
defines a map of $\Z[A^*]$-algebras
$T: S(A) \to \hat{K}^{MW}(A)$
sending $S(A)\cdot [-1,1]$ to zero.
In particular, it induces a map of left $S(A)$-modules
$$\bar{S}(A) \to \hat{K}^{MW}(A).$$
\end{proposition}

\begin{proof}
The map $S_n(A) \to \hat{K}^{MW}_n(A)$ given by the formula in the proposition
is a well-defined map of $\Z[A^*]$-modules in view of Propositions \ref{prop:SnPresentation} and \ref{prop:KMWnSnPresentation}.
In order to check multiplicativity of this map, take
$x = \sum_i m_i(\alpha_i) \in Z_m(A^m)$ and $y = \sum_j n_j(\beta_j) \in Z_n(A^n)$
with $m_i,n_j\in \Z[A^*]$ and $\alpha_i\in GL_m(A)$ and $\beta_j\in GL_n(A)$.
Choose vectors $v\in A^m$ (and $w\in A^n$) which are in general position with respect to $\alpha_i$ (and $\beta_j$ respectively).
Then $(v,w)\in A^{m+n}$ is in general position with respect to the frames 
$\alpha_i\oplus \beta_j = \left(\begin{smallmatrix}\alpha_i& 0 \\ 0 & \beta_j\end{smallmatrix}\right)$.
By Remark \ref{rem:NS:3.14}, we have in $S(A)$
$$
x\cdot y  =  \sum m_in_j(\alpha_i\oplus \beta_j) = 
\sum m_in_j \langle \alpha_i \oplus \beta_j \rangle [\alpha_i^{-1}v,\beta_j^{-1}w]
$$
whereas 
$$x = \sum_i m_i(\alpha_i)  = \sum_i m_i\langle \alpha_i\rangle [\alpha_i^{-1}v]$$
$$y = \sum_j n_j(\beta_j)  = \sum_j n_j\langle \beta_j\rangle [\beta_j^{-1}w].$$
This proves multiplicativity.
Since $[-1,1]=0\in \hat{K}_2^{MW}$, we are done.
\end{proof}

\begin{lemma}
\label{lem:SnprodOfGens}
For arbitrary $a_1,...,a_n \in A^*$, the following formula holds in ${S}_n(A)$
$$[a_1]\cdots [a_n] = \sum_{1 \leq i_1 < \cdots < i_k \leq n} (-1)^k\left\langle {\prod_{s=1}^k a_{i_s}}\right\rangle [a_1,\dots,1,\dots,1,\dots,a_n]$$
where the summand $[a_1,\dots,1,\dots,1,\dots,a_n]$, corresponding to the index $(i_1,\dots,i_k)$, is obtained from $[a_1,...,a_n]$ by replacing $a_{i_s}$ with $1$ for $s=1,...,k$.
\end{lemma}

\begin{proof}
We have $[a] = d(e_1,ae_1) = (ae_1) - (e_1)$.
Hence,
$$
\renewcommand\arraystretch{2}
\begin{array}{rcl}
[a_1]\cdots[a_n] & = & \prod_{i=1}^n ((a_ie_i) - (e_i))\\
& = &
\sum_{1 \leq i_1 < \cdots < i_k \leq n} (-1)^{n-k}(e_1,\dots,a_{i_1}e_{i_1},\dots \dots,a_{i_k}e_{i_k},....,e_n)
\end{array}$$
The vector $v=a_1e_1+\cdots + a_ne_n$ is in general position with respect to this cycle.
Hence,
$$[a_1]\cdots[a_n] = \sum_{1 \leq i_1 < \cdots < i_k \leq n} (-1)^{k}
\langle \alpha_i \rangle [\alpha_{i_1,...,i_k}^{-1}v]$$
where $\alpha_{i_1,...,i_k}$ is the matrix 
$$\alpha_{i_1,...,i_k} = (e_1,\dots,a_{i_1}e_{i_1},\dots \dots,a_{i_k}e_{i_k},....,e_n)\in GL_n(A).$$
Since $v = \alpha_{i_1,...,i_k} \cdot (a_1,...,1,...,1,....,a_n)^T$ and
$\det \alpha_{i_1,...,i_k} = a_{i_1}\cdots a_{i_k}$ we are done.
\end{proof}

\begin{lemma}
\label{lem:InvertibleElementInH2SL2}
For $\lambda \in A^*$ such that $\bar{\lambda}\neq 1$, the element 
$$s(\lambda) = 1 - \langle 1-\lambda\rangle - \langle \lambda\rangle \in \Z[A^*]$$
 acts as a unit on the $A^*$-module $H_2(SL_2(A))$.
\end{lemma}

\begin{proof}
The proof is essentially contained in \cite[\S 2]{Mazzoleni}.
Let $(a_1,...,a_r)$ be a sequence of units $a_i\in A^*$.
Then a sequence $(v_1,...,v_r)$ of vectors $v_i\in A^n$ is in general position if and only if the sequence $(a_1v_1,...,a_rv_r)$ is in general position.
For $r\geq 1$ we can therefore define the set 
$\ppp\tilde{U}_r(A)$ as the quotient of the set $\tilde{U}_r(A)$ by the equivalence relation $(v_1,...,v_r)\sim  (a_1v_1,...,a_rv_r)$.
We define the complex $\ppp\tilde{C}(A^n)$ by 
$$\ppp\tilde{C}_r(A^n)=\Z[ \ppp\tilde{U}_{r+1}(A^n)]$$
 for $r\geq 0$ and $\ppp\tilde{C}_r(A^n)=0$ for $r<0$ (note the shift in degree compared to $\tilde{C}_r(A^n)$). 
The differential $\ppp\tilde{C}_r(A^n) \to \ppp\tilde{C}_{r-1}(A^n)$ is given by the same formula as for $\tilde{C}_r(A^n)$.
The action of $GL_n(A)$ on $A^n$ makes the complex $\ppp\tilde{C}(A^n)$ into a complex of $GL_n(A)$-modules.
The unique map $\ppp\tilde{U}_1 \to \pt$ defines a map of complexes
$\ppp\tilde{C}(A^n) \to \Z$ of $GL_n(A)$-modules where $\pt$ is the one-element set.
The proof of Lemma \ref{lem:QuillenAcyclicity} shows that this map of complexes induces an isomorphism on homology.
Hence, for $n\geq 1$ the homology of 
\begin{equation}
\label{eqn:UnitActionResolution}
\Z \stackrel{^L}{\otimes}_{SL_n} \ppp\tilde{C}(A^n) \simeq \Z[A^*]\stackrel{L}{\otimes}_{GL_n}\ppp\tilde{C}(A^n)
\end{equation}
computes the homology of $SL_n(A)$.
Let $\ppp\tilde{C}_{\leq r}(A^n) \subset \ppp\tilde{C}(A^n)$ be the subcomplex with 
$\ppp\tilde{C}_{\leq r}(A^n)_i = \ppp\tilde{C}_i(A^n)$ for $i\leq r$ and zero otherwise.
This defines a filtration on 
(\ref{eqn:UnitActionResolution})
by the complexes  $\Z[A^*]\stackrel{L}{\otimes}_{GL_n}\ppp\tilde{C}_{\leq r}(A^n)$ of $\Z[A^*]$-modules and thus a spectral sequence of $\Z[A^*]$-modules
$$E^1_{p,q}(A^n) = H_p(\Z[A^*]\stackrel{L}{\otimes}_{GL_n}\ppp\tilde{C}_q(A^n)) \Rightarrow H_{p+q}(SL_n(A))$$
with differentials $d^r$ of bidegree $(r-1,-r)$.
For $1\leq q \leq n$, the group $SL_n(A)$ acts transitively on the set $\ppp\tilde{U}_q(A^n)$ with stabilizer at $(e_{n-q+1},...,e_n)$ the group 
$$\renewcommand\arraystretch{2}
\begin{array}{l}
\ppp\Aff^{SL}_{q,n-q}(A) = \\
\left\{\left(\begin{smallmatrix}M&0\\ N& D\end{smallmatrix}\right)|\ M\in GL_{n-q}(A),\ D\in T_q(A), N \in M_{q,n-q}(A),\ \det M\det D =1\right\}
\end{array}$$
where $(A^*)^q = T_q(A)\subset GL_q(A)$ is the subgroup of diagonal matrices.
By \cite[Lemma 9]{Hutchinson:Matsumoto} (whose proof works for local rings with infinite residue field), the inclusion of groups
$$\left\{\left(\begin{smallmatrix}M&0\\ 0& D\end{smallmatrix}\right)|\ M\in GL_{n-q}(A),\ D\in T_q(A),\ \det M\det D =1\right\} \subset \ppp\Aff^{SL}_{q,n-q}(A)$$
induces an isomorphism on integral homology groups.
For $n=2$ and $q=1,2$, the left hand side is $A^*$ and thus, its homology has trivial $A^*$-action.
Thus, $A^*$ acts trivially on $E_{p,q}^1(A^2)$ for $q\leq 1$.
It follows that $A^*$ acts trivially on $E_{p,q}^{\infty}(A^2)$ for $q\leq 1$.
In particular, the element $s(\lambda)$ acts as $-1$, hence as a unit on $E_{p,q}^{\infty}(A^2)$ for $q\leq 1$.
To finish the proof of the lemma, it suffices to show that $s(\lambda)$ acts as a unit on the cokernel of the $\Z[A^*]$-module map
\begin{equation}
\label{eqn:sIsInvertibleInH2}
d^1:E^1_{0,3}(A^2) \to E^1_{0,2}(A^2).\end{equation}
As a $GL_2(A)$-set we have
$$\ppp\tilde{U}_3(A^2) = GL_2(A)/D_2(A)\cdot (e_1,e_2,e_1+e_2)$$
where $D_2(A)=A^*\cdot 1$ is the group of invertible scalar matrices.
It follows that we have isomorphisms of $A^*$-modules
$$E^1_{0,2}(A^2) = \Z[A^*]\otimes_{GL_2}\Z[\ppp\tilde{U}_3(A^2)]
\cong \Z[A^*]\otimes_{D_2}\Z \cong \Z[A^*/A^{2*}]$$
where $1\in \Z[A^*/A^{2*}]$ corresponds to 
$1\otimes (e_1,e_2,e_1+e_2)\in \Z[A^*]\otimes_{GL_2}\Z[\ppp\tilde{U}_3(A^2)]$.
As a $GL_2(A)$-set we have
$$\ppp\tilde{U}_4(A^2) =\bigsqcup_{a,b\in A^*,\ \bar{a}\neq \bar{b}} GL_2(A)/D_2(A)\cdot (e_1,e_2,e_1+e_2,ae_1+be_2).$$
The map (\ref{eqn:sIsInvertibleInH2}) sends the element $(e_1,e_2,e_1+e_2,ae_1+be_2)$ to
the following element in $E^1_{0,2}(A^2)$
$$\renewcommand\arraystretch{2}
\begin{array}{l}
\phantom{=} \left(\begin{smallmatrix}0&1&a\\ 1&1&b\end{smallmatrix}\right)
- \left(\begin{smallmatrix}1&1&a\\ 0&1&b\end{smallmatrix}\right)
+ \left(\begin{smallmatrix}1&0&a\\ 0&1&b\end{smallmatrix}\right)
-\left(\begin{smallmatrix}1&0&1\\ 0&1&1\end{smallmatrix}\right)\\
= \langle -1 \rangle \left(\begin{smallmatrix}-1&1\\ 1&0\end{smallmatrix}\right) \left(\begin{smallmatrix}0&1&a\\ 1&1&b\end{smallmatrix}\right)
- \left(\begin{smallmatrix}1&-1\\ 0&1\end{smallmatrix}\right) \left(\begin{smallmatrix}1&1&a\\ 0&1&b\end{smallmatrix}\right)
+ \left(\begin{smallmatrix}1&0&a\\ 0&1&b\end{smallmatrix}\right)
-\left(\begin{smallmatrix}1&0&1\\ 0&1&1\end{smallmatrix}\right)\\
= \langle -1\rangle \left(\begin{smallmatrix}1&0&b-a\\ 0&1&a\end{smallmatrix}\right)
- \left(\begin{smallmatrix}1&0&a-b\\ 0&1&b\end{smallmatrix}\right)
+ \left(\begin{smallmatrix}1&0&a\\ 0&1&b\end{smallmatrix}\right)
-\left(\begin{smallmatrix}1&0&1\\ 0&1&1\end{smallmatrix}\right)\\
=(\langle -(b-a)a\rangle - \langle(a-b)b\rangle +\langle ab\rangle -\langle 1\rangle ) \cdot (e_1,e_2,e_1+e_2).
\end{array}
$$
It follows that the cokernel of the map (\ref{eqn:sIsInvertibleInH2}) is the quotient of $\Z[A^*/A^{2*}]$ modulo the $A^*$-submodule generated by 
$\langle (a-b)a\rangle - \langle(a-b)b\rangle +\langle ab\rangle -\langle 1\rangle$ whenever $a,b\in A^*$ with $\bar{a}\neq \bar{b}$.
Setting $a=1$ and $b=\lambda$, we see that $s(\lambda) =-\langle \lambda(1-\lambda)\rangle$ acts invertibly on the cokernel of (\ref{eqn:sIsInvertibleInH2}).
\end{proof}

The following proposition shows that the Steinberg relation holds in $\sigma^{-1}\bar{S}_2(A)$.

\begin{proposition}
\label{prop:SteinbergInSbar}
Let $2\leq n_0$ and let $\sigma\in \Z[A^*]$ as in $(\ast)$.
Then for any $\lambda \in A^*$ such that $\bar{\lambda}\neq 1$ we have
in the $A^*$-module $\sigma^{-1}S_2(A)$ the following equality 
$$[\lambda][1-\lambda] = \langle\langle \lambda \rangle\rangle \langle \langle 1-\lambda \rangle\rangle \cdot [-1,1].$$
\end{proposition}

\begin{proof}
By Lemma \ref{lem:SnprodOfGens}, we have
$$[\lambda][1-\lambda] = [\lambda,1-\lambda] - \langle 1-\lambda\rangle [\lambda,1] - \langle \lambda \rangle [1,1-\lambda] + \langle \lambda(1-\lambda)\rangle [1,1].$$
Recall from Proposition \ref{prop:SnPresentation} that for $\alpha,\beta, s,t\in A^*$ with $\bar{\alpha}\neq \bar{\beta}$ we have
\begin{equation}
\label{lem:eqn:NS.3.17.1b}
[\alpha s,\beta t] - [s,t] = 
\langle t\rangle [(\alpha-\beta)s,\beta] -\langle -s\rangle [(\beta-\alpha)t,\alpha].
\end{equation}
Setting $\alpha = 1-b\lambda^{-1}$, $\beta=1$, $s=a\lambda$, $t=b$ (where $\bar{\lambda} \neq \bar{b}$) in (\ref{lem:eqn:NS.3.17.1b}), we obtain
\begin{equation}
\label{lem:SteinbergInSbar:eqn:1}
[a\lambda -ab,b]-[a\lambda,b] = \langle b\rangle[-ab,1]-\langle -a\lambda\rangle [b^2\lambda^{-1},1-b\lambda^{-1}].
\end{equation}
Setting $\alpha=\lambda$, $\beta = b$, $s=a$ and $t=1$ (where $\bar{\lambda} \neq \bar{b}$) in (\ref{lem:eqn:NS.3.17.1b}), we obtain
\begin{equation}
\label{lem:SteinbergInSbar:eqn:2}
[a\lambda,b]-[a,1]=[a\lambda - ab,b] - \langle -a\rangle[b-\lambda,\lambda].
\end{equation}
Adding equations (\ref{lem:SteinbergInSbar:eqn:1}) and (\ref{lem:SteinbergInSbar:eqn:2}), cancelling common summands and multiplying with $\langle -a^{-1}\rangle$ we obtain
\begin{equation}
\label{lem:SteinbergInSbar:eqn:3}
 \langle -a^{-1}b\rangle [-ab,1]+\langle -a^{-1}\rangle [a,1] = \langle \lambda\rangle [b^2\lambda^{-1},1-b\lambda^{-1}] + [b-\lambda,\lambda].
\end{equation}
for any $a,b,\lambda \in A^*$ with $\bar{\lambda}\neq \bar{b}$.
Note that the right hand side of (\ref{lem:SteinbergInSbar:eqn:3}) is independent of $a\in A^*$.
Hence, so is the left hand side and thus it equals its evaluation at $a=-1$, that is, we have
$$
 \langle - a^{-1}b\rangle [-ab,1]+\langle -a^{-1}\rangle [a,1]
= \langle b\rangle [b,1]+ [-1,1].
$$
Therefore,
\begin{equation}
\label{lem:SteinbergInSbar:eqn:4}
[-ab,1] = - \langle b^{-1}\rangle [a,1]
 + \langle -a\rangle [b,1]+ \langle -ab^{-1}\rangle[-1,1],
\end{equation}
for any $a,b\in A^*$ (as we can always choose a $\lambda \in A^*$ with $\bar{\lambda}\neq \bar{b}$).
Replacing $b$ with $-b$ in equation (\ref{lem:SteinbergInSbar:eqn:3}) we see that the expression
$\langle a^{-1}b\rangle [ab,1]+\langle -a^{-1}\rangle [a,1] $ 
does not depend on $a$.
In particular, it equals its evaluation at $a=1$, and we have
$$\langle a^{-1}b\rangle [ab,1]+\langle -a^{-1}\rangle [a,1] 
=\langle b\rangle [b,1]+\langle -1\rangle [1,1],$$
that is,
\begin{equation}
\label{lem:SteinbergInSbar:eqn:5}
[ab,1]= - \langle -b^{-1}\rangle [a,1] 
+ \langle a\rangle [b,1]+\langle -ab^{-1}\rangle [1,1].
\end{equation}
For $b=1$ this yields
\begin{equation}
\label{lem:SteinbergInSbar:eqn:5b}
[a,1]= - \langle -1\rangle [a,1] 
+ \langle a\rangle [1,1]+\langle -a\rangle [1,1].
\end{equation}
Setting $s=a$, $\alpha = \lambda$, $t=\beta =1$ (where $\bar{\lambda}\neq 1$) in (\ref{lem:eqn:NS.3.17.1b}), we obtain
$$[\lambda a,1] - [a,1] = [(\lambda-1)a,1] - \langle -a\rangle [(1-\lambda),\lambda]
$$
that is,
\begin{equation}
\label{lem:SteinbergInSbar:eqn:6}
\langle -a\rangle [1-\lambda,\lambda] =  [(\lambda -1)a,1] -[\lambda a,1] + [a,1].
\end{equation}
Together with equations (\ref{lem:SteinbergInSbar:eqn:4}) and (\ref{lem:SteinbergInSbar:eqn:5}) this yields
$$\renewcommand\arraystretch{1.5}
\begin{array}{rcl}
&&\langle -a\rangle [1-\lambda,\lambda] \\
&=& - \langle (1-\lambda)^{-1}\rangle [a,1] + \langle -a\rangle [1-\lambda,1] + \langle -a(1-\lambda)^{-1}\rangle [-1,1] \\
&& + \langle -\lambda^{-1}\rangle [a,1] -\langle a \rangle [\lambda,1] - \langle -a\lambda^{-1}\rangle [1,1] +[a,1]\\
&\underset{\text{(\ref{lem:SteinbergInSbar:eqn:5b})}}{=}&
(1-\langle \lambda^{-1}\rangle -\langle (1-\lambda)^{-1}\rangle )\cdot [a,1]\\
&&+\langle -a\rangle \cdot ([1-\lambda,1]+ \langle (1-\lambda)^{-1}\rangle [-1,1] -\langle -1\rangle [\lambda ,1] + \langle- \lambda^{-1}\rangle [1,1]).
\end{array}
$$
It follows that the expression
$$(1-\langle (1-\lambda)^{-1} \rangle -\langle \lambda ^{-1}\rangle )\cdot \langle -a^{-1}\rangle [a,1]$$
does not depend on $a$.
In particular,
\begin{equation}
\label{eqn:FirstFactor}
\{1-\langle (1-\lambda)^{-1} \rangle -\langle \lambda ^{-1}\rangle \}\cdot \{\langle -a^{-1}\rangle [a,1] - [-1,1]\} = 0 \in S_2(A).
\end{equation}
Since $\det [a,1] = \langle -a\rangle$, we have
$$\langle -a^{-1}\rangle [a,1] - [-1,1] \in \sigma^{-1} \ker (\det:S_2 \to \Z[A^*]) = \sigma^{-1}F_{2,0} = H_2(SL_2A),$$
where the equality $\sigma^{-1}F_{2,0}(A^2) = H_2(SL_2A)$ follows from Lemma \ref{lem:RelHisSn}.

By  Lemma \ref{lem:InvertibleElementInH2SL2}
the first factor in (\ref{eqn:FirstFactor}) is invertible in $H_2(SL_2A)$ as square units act trivially on that group.
Hence, we have $\langle -a^{-1}\rangle [a,1] - [-1,1] = 0 \in S_2(A)$ for all $a\in A^*$ and therefore,
\begin{equation}
\label{lem:SteinbergInSbar:eqn:7}
[a,1]=\langle a\rangle [1,1]\in \sigma^{-1}S_2(A).
\end{equation}
Setting $s=t=\alpha = 1$, $\beta = \lambda$ with $\bar{\lambda}\neq 1$ in (\ref{lem:eqn:NS.3.17.1b}) yields
$$[1,\lambda]-[1,1]=[(1-\lambda),\lambda]-\langle -1\rangle [\lambda -1,1].$$
Putting this into (\ref{lem:SteinbergInSbar:eqn:6}) with $a=1$ yields the equation in $\sigma^{-1}S_2(A)$
\begin{equation}
\label{lem:SteinbergInSbar:eqn:8}
[1,\lambda]= -\langle -1\rangle [\lambda,1] +[1,1]+\langle -1\rangle [1,1].
\end{equation}
Finally, putting $a=-1$ in (\ref{lem:SteinbergInSbar:eqn:6}) and using 
(\ref{lem:SteinbergInSbar:eqn:8}) 
we find for $\lambda \in A^*$ with $\bar{\lambda}\neq 1$ the equation in $\sigma^{-1}S_2(A)$
$$
\renewcommand\arraystretch{1.5}
\begin{array}{rcl}
&&
[1-\lambda,\lambda] - \langle \lambda \rangle [1-\lambda,1] -\langle 1-\lambda\rangle [1,\lambda] + \langle(1-\lambda)\lambda\rangle [1,1]\\
&\underset{\text{(\ref{lem:SteinbergInSbar:eqn:8})}}{=}&[1-\lambda,1]+[-1,1]-[-\lambda,1] - \langle \lambda \rangle [1-\lambda,1] \\
&&-\langle 1-\lambda\rangle (-\langle -1\rangle [\lambda,1] +[1,1]+\langle -1\rangle [1,1]) + \langle(1-\lambda)\lambda\rangle [1,1]\\
&\underset{\text{(\ref{lem:SteinbergInSbar:eqn:7})}}{=}&
\langle -1\rangle (
1
- \langle \lambda\rangle 
+ \langle \lambda(1-\lambda)\rangle 
- \langle 1- \lambda \rangle 
) \cdot [1,1]\\
&=&(
1
- \langle \lambda\rangle 
+ \langle \lambda(1-\lambda)\rangle 
- \langle 1- \lambda \rangle 
) \cdot [-1,1]\\
&= & \langle\langle \lambda \rangle\rangle \langle \langle 1-\lambda \rangle\rangle \cdot [-1,1].
\end{array}
$$
Replacing $\lambda$ with $1-\lambda$ yields the desired result.
\end{proof}

\begin{lemma}
Let $2\leq n \leq n_0$ and $\sigma \in \Z[A^*]$ as in $(\ast)$.
Then the localization map induces an isomorphism
$$\hat{K}_n^{MW}(A) \stackrel{\cong}{\longrightarrow} \sigma^{-1}\hat{K}_n^{MW}(A).$$
\end{lemma}

\begin{proof}
Since $n_0\geq 2$, the number $t$ in ($\ast$) is even and $\sigma=s_{m,-t}$ acts as $1$ on $\hat{K}_n^{MW}(A)$ for $n\geq 2$ since square units act as $1$ on it (Lemma \ref{Hatlem:basicKMWformulasTwo} (\ref{lem:TowardsGWactionOnTwistedK:7})) and $\eps(\sigma)=1$.
Therefore, $\hat{K}_n^{MW}(A) = \sigma^{-1}\hat{K}_n^{MW}(A)$.
\end{proof}

\begin{corollary}
\label{cor:S2isKMW}
Let $2\leq n \leq n_0$ and $\sigma \in \Z[A^*]$ as in $(\ast)$.
The following map is well-defined and an isomorphism $A^*$-modules
$$\hat{K}^{MW}_2(A) \stackrel{\cong}{\longrightarrow} \sigma^{-1}\bar{S}_2(A):[a,b]\mapsto [a]\cdot [b].$$
The inverse isomorphism is the map 
$T_2: \sigma^{-1}\bar{S}_2(A) \to \sigma^{-1}K_2^{MW}(A)=K_2^{MW}(A)$
from Proposition \ref{prop:SToKMW}.
\end{corollary}

\begin{proof}
The map is well-defined, by Proposition \ref{prop:SteinbergInSbar}.
It is surjective, by Lemma \ref{lem:IndecGens}.
It follows from the multiplicativity of the map in Proposition \ref{prop:SToKMW} that the composition 
$\hat{K}^{MW}_2(A) \to \sigma^{-1}\bar{S}_2(A) \to \hat{K}^{MW}_2(A)$
is the identity.
This proves the claim.
\end{proof}

We have proved the following presentation of $H_2(SL_2A)$.
Different presentations were given in \cite[Corollaire 5.11]{matsumoto}, \cite[Theorem 9.2]{moore}, \cite[Theorem 3.4]{vdKH2SL2}.

\begin{theorem}
\label{thm:H2andKMW2}
Let $A$ be a commutative local ring with infinite residue field.
Let $I[A^*]\subset \Z[A^*]$ be the augmentation ideal, and write $[a]$ for $\langle a\rangle -1 \in I[A^*]$.
Then there is an isomorphism of $A^*$-modules
$$H_2(SL_2A,\Z) \cong I[A^*]\otimes_{A^*}I[A^*]/\{[a]\otimes [1-a]|\ a,1-a\in A^*\}.$$
\end{theorem}

\begin{proof}
Recall that the right hand side is $\hat{K}_2^{MW}(A)$, by definition.
The isomorphism in the theorem is the composition of isomorphisms
$$H_2(SL_2A) \stackrel{\cong}{\longrightarrow} \sigma^{-1}F_{1,1}(A^2) \stackrel{\cong}{\longrightarrow} \sigma^{-1}\bar{S}_2 \stackrel{\cong}{\longrightarrow} 
\hat{K}_2^{MW}(A)$$
from Lemma \ref{lem:RelHisSn} and Corollary \ref{cor:S2isKMW}.
\end{proof}

\begin{remark}
\label{rmk:S2SymbolsProds}
For any $a,b\in A^*$ we have in $\sigma^{-1}S_2(A)$ the following equality
$$[a,b] = [a]\cdot [b] + (\langle -a\rangle + \langle a \rangle - \langle ab\rangle) [-1,1].$$
This follows from the isomorphism
$$(\det, T_2): \sigma^{-1}S_2(A) \stackrel{\cong}{\longrightarrow} \sigma^{-1}\Z[A^*] \oplus \hat{K}_2^{MW}(A).$$
\end{remark}
\vspace{2ex}

\subsection{Centrality of $[-1,1]$  and $H_n(SL_nA,SL_{n-1}A)$}

\begin{center}
{\em In this subsection, $A$ is a commutative local ring with infinite residue field.}
\end{center}

\begin{definition}
For $\lambda \in A^*$ with $\bar{\lambda}\neq 1$, we define the element
$\beta_{\lambda}\in S_3(A)$ as
$$\renewcommand\arraystretch{1.5}
\begin{array}{rcl}
\beta_{\lambda} & = & [1,1-\lambda, \lambda]-[1,1-\lambda,1] + [1,-\lambda,1]\\
&& -[1-\lambda,\lambda,1]+ [1-\lambda,1,1]-[-\lambda,1,1].
\end{array}$$
\end{definition}

Note that $\det \beta_{\lambda} =0$, by Remark \ref{rmk:detComp}.

\begin{lemma}
\label{lem:BetaCommutator}
For all $a, \lambda \in A^*$ such that $\bar{\lambda}\neq 1$ we have in $S_3(A)$ the equality
$$ [-1,1]\cdot [a] - [a]\cdot [-1,1]= \langle\langle a \rangle\rangle \cdot \beta_{\lambda}.$$
Moreover, $\beta_{\lambda}=\beta_{\mu}$ for all $\lambda,\mu \in A^*$ with $\bar{\lambda},\bar{\mu}\neq 1$.
\end{lemma}

\begin{proof}
Let $u$ be the element $u=d(e_2,e_3,e_3-e_2) \in Z_2(Ae_2+Ae_3)$.
We have
$$
\renewcommand\arraystretch{1.5}
\begin{array}{rcl}
&& [-1,1]\cdot [a] - [a]\cdot [-1,1] \\
&=&  d(e_1,e_2,e_2-e_1)\cdot (ae_3-e_3) - (ae_1-e_1)\cdot d(e_2,e_3,e_3-e_2)\\
&=& u \cdot (ae_1-e_1) - (ae_1-e_1)\cdot u \\
&=& \{(e_1)\cdot u-u\cdot(e_1)\} - \{(ae_1)\cdot u-u\cdot (ae_1)\} \\
&=& \langle\langle a\rangle\rangle \{u\cdot (e_1)- (e_1)\cdot u\}.
\end{array}
$$
We need to show that $\beta_{\lambda} = u\cdot (e_1) - (e_1)\cdot u$.
Note that the right hand side is independent of $\lambda$.
The vector $v = (1,-\lambda,1)^T$ is in general position with respect to the vectors $e_1,e_2,e_3,e_3-e_2$ occuring in $u(e_1) - (e_1)u$.
Therefore, we obtain the equality
$u(e_1) - (e_1)u = d\{(e_1,u,v)-(u,e_1,v)\}$.
In $\tilde{C}_4(A^3)/SL_3(A)$ we have 
$$
\renewcommand\arraystretch{1.5}
\begin{array}{rcl}
(e_1,u,v)
&=& (e_1,e_3,e_3-e_2,v) - (e_1,e_2,e_3-e_2,v) + (e_1,e_2,e_3,v)\\
&=& (e_1,e_2,e_3,v_1) - (e_1,e_2,e_3,v_2) + (e_1,e_2,e_3,v_3)
\end{array}
$$
where $v_1 = (1,1-\lambda,\lambda)^T$, $v_2 = (1,1-\lambda,1)^T$ and $v_3=v$ since matrix multiplication with $(e_1,e_3,e_3-e_2), (e_1,e_2,e_3-e_2) \in SL_3(A)$ yields $v=(e_1,e_3,e_3-e_2)v_1 = (e_1,e_2,e_3-e_2)v_2$.
Similarly, we have in $\tilde{C}_4(A^3)/SL_3(A)$
$$
\renewcommand\arraystretch{1.5}
\begin{array}{rcl}
(u,e_1,v)
&=& (e_3,e_3-e_2,e_1,v) - (e_2,e_3-e_2,e_1,v) + (e_2,e_3,e_1,v)\\
&=& (e_1,e_2,e_3,v_4) - (e_1,e_2,e_3,v_5) + (e_1,e_2,e_3,v_6)
\end{array}
$$
where $v_4 = (1-\lambda,\lambda,1)^T$, $v_5 = (1-\lambda,1,1)^T$ and $v_6=(-\lambda,1,1)^T$ since matrix multiplication with $(e_3,e_3-e_2,e_1), (e_2,e_3-e_2,e_1), (e_2,e_3,e_1) \in SL_3(A)$ yields 
$v=(e_3,e_3-e_2,e_1)v_4= (e_2,e_3-e_2,e_1)v_5 = (e_2,e_3,e_1)v_6$.
The result follows.
\end{proof}

In view of the independence of $\lambda$ we will write $\beta$ for $\beta_{\lambda}$.
Write $S^0_n(A)$ for the kernel of the determinant map $\det:S_n(A) \to \Z[A^*]$.

\begin{lemma}
\label{lem:BetaInS1S20}
Let $n_0\geq 3$ and $\sigma\in \Z[A^*]$ as in $(\ast)$.
Then 
$$\beta \in \sigma^{-1}S_1(A)\cdot S_2^0(A) \subset \sigma^{-1}S_3(A).$$
\end{lemma}

\begin{proof}
To simplify notation I will write $S_n$ for $\sigma^{-1}S_n(A)$.
From the definition of $\beta$ and  Lemma \ref{lem:S3prods} (\ref{lem:S3prods:D}) we have in $S_3$ modulo $S_2[\lambda] +S_1S_2$ the equality
$$\beta = \{ \langle (1-\lambda)\lambda\rangle -\langle 1-\lambda\rangle + \langle -\lambda \rangle - \langle (1-\lambda)\lambda\rangle + \langle 1-\lambda\rangle - \langle -\lambda\rangle \}\cdot [1,1,1] =0.$$
Therefore, $\beta \in S_2[\lambda] +S_1S_2$.
By Remark \ref{rmk:S2SymbolsProds}, we have $S_2=\Z[A^*]\cdot [-1,1]+S_1S_1$, and hence,
$\beta \in \Z[A^*]\cdot [-1,1][\lambda] +S_1S_2$.
Since $\det [-1,1]=1$, we have the decomposition $S_2=S_2^0+ \Z[A^*]\cdot[-1,1]$, 
and thus, $\beta = r \cdot [-1,1][\lambda] +c\cdot [-1,1]$ modulo $S_1S_2^0$ for some $r\in \Z[A^*]$ and $c\in S_1$ (depending on $\lambda$).
Since $\det \beta =0$ we can compare determinants and use $\det (S_1S_2^0)=\det(S_1)\det(S_2^0)=0$ to find $c=-r[\lambda]$.
Hence, for all $\lambda \in A^*$ with $\bar{\lambda}\neq 1$ there is 
$r\in \Z[A^*]$ such that 
\begin{equation}
\label{eqn:BetaIsLambdaBeta}
\beta = r\cdot \langle\langle \lambda \rangle\rangle \cdot \beta \mod S_1S_2^0.
\end{equation}
Since $\det \beta =0$, we have $\beta \in S_3^0 = \sigma^{-1}F_{2,1}(A^3)\cong H_3(SL_3A,SL_2A)$.
Since square units act trivially on $H_2(SL_2A)$ and cube units act trivially on $H_3(SL_3A)$, the exact sequence
$$H_3(SL_3A) \to H_3(SL_3A,SL_2A) \to H_2(SL_2A)$$
implies that for all $a,b\in A^*$ we have 
$\langle \langle a^2\rangle\rangle \langle\langle b^3\rangle\rangle \cdot H_3(SL_3A,SL_2A)=0$.
In particular, $\langle \langle a^2\rangle\rangle \langle\langle b^3\rangle\rangle \cdot \beta =0$.
Now choose $a,b\in A^*$ such that $\bar{a}^2,\bar{b}^3\neq 1$.
This is possible since $A$ has infinite residue field.
From (\ref{eqn:BetaIsLambdaBeta}) we infer that
$$\beta = r_1r_2\langle \langle a^2\rangle\rangle \langle\langle b^3\rangle\rangle \cdot \beta = 0 \mod S_1S_2^0.$$
Hence, $\beta \in S_1S_2^0$.
\end{proof}

\begin{lemma}
\label{lem:BetaInF30}
Let $n_0\geq 3$ and $\sigma\in \Z[A^*]$ as in $(\ast)$.
Then 
$$\beta \in \sigma^{-1}F_{3,0}(A^3).$$
\end{lemma}

\begin{proof}
To simplify, write $F_{p,q}$, $E^s_{p,q}$, $D^s_{p,q}$
 and $S_n$ also for their localizations at $\sigma$.
Since $\det \beta =0$, we have $\beta\in \ker(\det) = F_{1,2}(A^3) = F_{2,1}(A^3)$.
So, we have to show that $\beta$ is sent to zero under the map (when $n=3$)
\begin{equation}
\label{eqn:BetaInF20}
F_{n-1,1}(A^n) \to E^{\infty}_{n-1,1}(A^n)= E^2_{n-1,1}(A^n)\subset E^1_{n-1,1}(A^n)
\end{equation}
which is well-defined for $n\leq n_0$, by Lemma \ref{lem:NS_lemma_2.4} and Proposition \ref{prop:NS_Prop_2.6}.
By Lemma \ref{Lem:DandF} (\ref{Lem:DandF:2}), this is also the map
$D^1_{n-1,1}(A^n) \to E^1_{n-1,1}(A^n)$.
For $n=1$, this map is the map $I[A^*] \to \Z[A^*]:[a] \mapsto \langle\langle a\rangle\rangle$.
Taking the direct sum over $n$, the map (\ref{eqn:BetaInF20}) is part of a $\B(A)$-bimodule map
$$\bigoplus_{0\leq n\leq n_0} F_{n-1,1}(A^n)  \to \bigoplus_{0\leq n \leq n_0} E^1_{n-1,1}(A^n)=\bigoplus_{0\leq n\leq n_0}\sigma^{-1}\Tor^{GL_{n}}_{n}(\Z[A^*],C_1(A^n)[1]),$$
by Subsection \ref{subsec:SpSeqLocRings}.
Consider the embedding
$$GL_{n-1}(A) \to \left(\begin{matrix}1 & \ast \\ 0 & GL_nA \end{matrix}\right):
M \mapsto \left(\begin{matrix}1 &  \\ 0 & M \end{matrix}\right)
$$
of $GL_{n-1}(A)$ into the stabilizer at $e_1$ of the $GL_n(A)$-action on $U_1(A^n)$.
By Theorem \ref{thm:HomologyOfAffineGps}, the inclusion induces an isomorphism
$$\bigoplus_{0\leq n\leq n_0}\sigma^{-1}\Tor^{GL_{n-1}}_{n-1}(\Z[A^*],\Z) \stackrel{\cong}{\longrightarrow} \bigoplus_{0\leq n\leq n_0}\sigma^{-1}\Tor^{GL_{n}}_{n}(\Z[A^*],C_1(A^n)[1])$$
of $A^*$-modules.
This map is also a map of right $\B(A)$-modules.
The composition of the $\B(A)$-bimodule map and the inverse of the right $\B(A)$-module map defines a right $\B(A)$-module map
$$\delta= \bigoplus_{0\leq n\leq n_0}\delta_n: \bigoplus_{0\leq n\leq n_0} F_{n-1,1}(A^n)  \to \bigoplus_{0\leq n \leq n_0} \sigma^{-1}\Tor^{GL_{n-1}}_{n-1}(\Z[A^*],\Z).$$
We have to show that $\delta_3(\beta)=0$.

I claim that the following diagram commutes (all localized at $\sigma$):
\begin{equation}
\label{eqn:DiagEtaCommutes}
\xymatrix{
S_1S_2^0\hspace{1ex} \ar@{^(->}[r] \ar@{^(->}[d]& F_{2,1} \hspace{1ex} \ar@{^(->}[r] & S_3 \ar[r]^{\hspace{-3ex}T_3} & K^{MW}_3(A) \ar[d]^{\eta}\\
F_{2,1}  \ar[r]^{\hspace{-3ex}\delta_3} & H_2(SL_2A) \hspace{1ex}\ar@{^(->}[r] & S_2 \ar[r]^{\hspace{-3ex}T_2}
 & K_2^{MW}(A)
}
\end{equation}
To see this, note that the $A^*$-module $S_1S_2^0$ is generated by products $[a]\cdot \gamma$ where 
$a\in A^*$ and $\gamma\in S^0_2$.
We have $S_2^0=\B_2(A)$, and thus, $\gamma \in \B_2(A)$.
Since $\delta$ is a right $\B(A)$-module morphism and $T$ is a map of rings, we have
$$T_2\delta_3([a]\cdot \gamma) = T_2(\delta_1([a])\cdot \gamma) =T_2(\langle\langle a\rangle\rangle \cdot \gamma) =  \langle\langle a\rangle\rangle T_2(\gamma)$$
since $\delta_1([a])=\langle\langle a\rangle\rangle$ as shown above.
On the other hand
$$\eta\cdot T_3([a]\cdot \gamma) = \eta\cdot T_1([a]) \cdot T_2(\gamma) = \eta\cdot [a]\cdot T_2(\gamma) = \langle\langle a \rangle\rangle T_2(\gamma).$$
So, the diagram does indeed commute.

Since $[1]=0\in \hat{K}_1^{MW}(A)$ and $T$ is multiplicative, the definition of $\beta$ yields $T_3(\beta) = 0$. 
By commutativity of diagram (\ref{eqn:DiagEtaCommutes}), we have $T_2\delta_3(\beta)=0$.
By (the proof of) Theorem \ref{thm:H2andKMW2},
the map $T_2:H_2(SL_2A) \to K_2^{MW}(A)$ is an isomorphism.
Hence, $\delta_3(\beta)=0$.
\end{proof}

\begin{lemma}
\label{lem:A*actionOnFA3}
Let $n_0\geq 3$ and $\sigma\in \Z[A^*]$ as in $(\ast)$. 
\begin{enumerate}
\item
\label{lem:A*actionOnFA3:A}
$A^{2*}$ acts trivially on $\sigma^{-1}S_1(A)S_2^0(A)$.
\item
\label{lem:A*actionOnFA3:B}
$A^*$ acts trivially on $F_{3,0}(A)\cap \sigma^{-1}S_1(A)S_2^0(A)$.
\end{enumerate}
In particular, for all $a\in A^*$, we have
$$\langle\langle a \rangle \rangle \cdot \beta =0\hspace{1ex} \in \sigma^{-1}S_3(A).$$
\end{lemma}

\begin{proof}
Since square units act trivially on $\sigma^{-1}S_2^0 = H_2(SL_2A)$,
this shows part (\ref{lem:A*actionOnFA3:A}).
The group of cube units $A^{3*}$ acts trivially on 
$\sigma^{-1}F_{3,0}(A^3)$ as this group is a quotient of $H_3(SL_3A)$.
Hence, if $\gamma \in F_{3,0}(A)\cap \sigma^{-1}S_1(A)S_2^0(A)$, then 
$\langle a^2\rangle \gamma = \gamma = \langle a^3 \rangle \gamma$ and thus,
$\langle a \rangle \gamma = \gamma$.
This proves part (\ref{lem:A*actionOnFA3:B}).

By Lemmas \ref{lem:BetaInS1S20} and \ref{lem:BetaInF30}, we have
$\beta \in F_{3,0}(A)\cap \sigma^{-1}S_1(A)S_2^0(A)$.
Hence, for all $a\in A^*$, we have $\langle a \rangle \beta = \beta$, that is,
$\langle\langle a \rangle\rangle \cdot \beta =0$.
\end{proof}

\begin{proposition}
\label{prop:EcommutesWithA}
Let $n_0\geq 3$ and $\sigma\in \Z[A^*]$ as in $(\ast)$. 
Then for all $a\in A^*$, we have in $\sigma^{-1}S_3(A)$ the equality
$$[a]\cdot [-1,1] = [-1,1]\cdot [a].$$
\end{proposition}

\begin{proof}
This follows from Lemmas \ref{lem:BetaCommutator} and \ref{lem:A*actionOnFA3}.
\end{proof}

For a graded ring $R=\bigoplus_{n\geq 0}R_n$, denote by $R_{\leq n}$ the quotient ring which is $R_i$ in degrees $i\leq n$ and $0$ otherwise.

\begin{corollary}
\label{cor:Centrality}
Let $n_0\geq n\geq 3$ and $\sigma\in \Z[A^*]$ as in $(\ast)$. 
Then the element $[-1,1]$ is central in $\sigma^{-1}S_{\leq n}(A)$.
In particular, 
$$\sigma^{-1}\bar{S}_{\leq n} = \sigma^{-1}S_{\leq n}(A)/S_{\leq n-2}(A)[-1,1]$$ is a quotient algebra of $\sigma^{-1}S(A)$.
\end{corollary}

\begin{proof}
Centrality follows from Proposition \ref{prop:EcommutesWithA}
in view of Proposition \ref{prop:Decomposibility} and Lemma \ref{lem:IndecGens}.
The claim follows.
\end{proof}

\begin{proposition}
\label{prop:SnisKMWn}
Let $2\leq n \leq n_0$ and $\sigma \in \Z[A^*]$ as in $(\ast)$.
Then the following map is a well-defined isomorphism $A^*$-algebras
$$\hat{K}^{MW}_{\leq n}(A) \stackrel{\cong}{\longrightarrow} \sigma^{-1}\bar{S}_{\leq n}(A):[a,b]\mapsto [a]\cdot [b].$$
with inverse the map 
$T_{\leq n}: \sigma^{-1}\bar{S}_{\leq n}(A) \to \sigma^{-1}K_{\leq n}^{MW}(A)$
from Proposition \ref{prop:SToKMW}.
\end{proposition}

\begin{proof}
The map is well-defined, by Proposition \ref{prop:SteinbergInSbar}.
It is surjective, by Lemma \ref{lem:IndecGens} and Proposition \ref{prop:Decomposibility}.
It follows from the multiplicativity of the map $T$  that the composition 
$\sigma^{-1}\hat{K}^{MW}_{\leq n}(A) \to \sigma^{-1}\bar{S}_{\leq n}(A) \to \sigma^{-1}\hat{K}^{MW}_{\leq n}(A)$
is the identity.
This proves the claim.
\end{proof}

The following proves Theorem \ref{thm:main:SLnLocalA} in view of Theorem \ref{thm:KhatIsK}.
Recall our convention for $SL_0(A)$ from the Introduction so that
 $H_*(SL_nA) = H_*(GL_nA,\Z[A^*])$ for $n\geq 0$.
We set $SL_n(A)=GL_n(A) = \emptyset$ for $n<0$.

\begin{theorem}
\label{thm:HnIsKMW}
Let $A$ be a commutative local ring with infinite residue field.
\begin{enumerate}
\item
Then $H_i(SL_n(A),SL_{n-1}(A)) = 0$ for $i<n$ and 
the maps in Proposition \ref{prop:SToKMW} and Lemma \ref{lem:RelHisSn} induce isomorphisms of $A^*$-modules for $n\geq 0$
$$H_n(SL_n(A),SL_{n-1}(A)) \cong \hat{K}_n^{MW}(A).$$
\item
If $n$ is even, then the natural map
$$H_n(SL_n(A))\to H_n(SL_n(A),SL_{n-1}(A))$$
is surjective and inclusion $SL_{n-1}A\subset SL_nA$ induces an isomorphism
$$H_{n-1}(SL_{n-1}A) \cong H_{n-1}(SL_{n}A).$$
\end{enumerate}
\end{theorem}

\begin{proof}
For $i<n$, the vanishing of homology follows from Theorem \ref{thm:Estability} as $\sr(A)=1$ and $E_n(A)=SL_n(A)$.
 
The statement of the theorem is clear for $n\leq 1$.
So, assume $n\geq 2$.
For $2\leq n\leq n_0$, 
the identification with Milnor-Witt $K$-theory follows from Proposition \ref{prop:SnisKMWn}
together with Lemma \ref{lem:RelHisSn}.

For the second part, assume $n$ is even.
We have maps of graded $\Z[A^*]$-algebras
$$\Tens_{\Z[A^*]} H_2(SL_2(A)) \to \B(A) = \bigoplus_{n\geq 0}H_n(SL_n(A)) \to S(A) \stackrel{T}{\to} \hat{K}^{MW}(A)$$
where $H_2(SL_2(A))$ is placed in degree $2$ and $H_0(SL_0A) = \Z[A^*]$, by our convention for $SL_0(A)$.
The composition is an isomorphism in degrees $0$ and $2$.
Since the target ring $\hat{K}^{MW}(A)$ is generated in degree $1$, its even part is generated in degree $2$, and the composition is surjective in even degrees.
The claim follows.
\end{proof}

\begin{theorem}
\label{thm:BassConjLocalRingsInText}
Let $A$ be a commutative local ring with infinite residue field.
Then, for $i\geq 0$, the natural homomorphism
$$\pi_iBGL_{n-1}^+(A) \to \pi_iBGL_n^+(A)$$
is an isomorphism for $n\geq i + 2$ and surjective for $n \geq i+1$.
Moreover, for $n\geq 2$ there is an exact sequence
$$\pi_nBGL_{n-1}^+(A) \to \pi_nBGL_{n}^+(A) \to K^{MW}_n(A) \to \pi_{n-1}BGL_{n-1}^+(A)
\to \pi_{n-1}BGL_{n}^+(A).$$
\end{theorem}

\begin{proof}
The theorem follows from Theorem \ref{thm:HnIsKMW} in view of Theorem \ref{thm:BassConjectureBis}.
\end{proof}

\subsection{Prestability}

\begin{center}
{\em In this subsection, $A$ is a commutative local ring with infinite residue field.}
\end{center}

This subsection is devoted to an explicit computation of the kernel and cokernel of the stabilization map in homology at the edge of stabilization as was done in \cite{HutchinsonTao:Stability} for characteristic zero fields.

Assume that $A$ is a local ring for which the Milnor conjecture on bilinear forms holds, that is, 
the ring homomorphism defined by Milnor \cite{milnor:KMpaper} is an isomorphism
\begin{equation}
\label{eqn:MilnorConj}
K^M_*(A)/2 \cong \bigoplus_{n\geq 0} I^n(A)/I^{n+1}(A)
\end{equation}
where $I(A)\subset W(A)$ is the fundamental ideal in the Witt ring of $A$.
By the work of Voevodsky and collaborators \cite{voevodskyCollab} and its extension by Kerz \cite[Theorem 7.10]{Kerz}, the map (\ref{eqn:MilnorConj}) is an isomorphism if $A$ is local and contains an infinite field of characteristic not $2$.
The map is also an isomorphism for any henselian local ring $A$ with $\frac{1}{2}\in A$ as both sides agree with their value at the residue field of $A$.
Using the isomorphism (\ref{eqn:MilnorConj}) we obtain a commutative diagram 
\begin{equation}
\label{eqn:MorelSquare}
\xymatrix{
K^{MW}_n(A) \ar@{->>}[d] \ar[r]^{\eta^n} & I^n(A) \ar@{->>}[d]\\
K^M_n(A) \ar@{->>}[r] & K^M_n(A)/2.
}
\end{equation}
In the following theorem we will assume this diagram to be cartesian.
By \cite{morel:puissances}, this is the case for fields whose characteristic is different from $2$.
This was generalized in \cite{GilleEtAl} to commutative local rings containing an infinite field of characteristic different from $2$.
So, our theorem holds in this case.

\begin{theorem}
\label{thm:KerCokerOdd}
Let $A$ be a commutative local ring with infinite residue field.
Assume that the map (\ref{eqn:MilnorConj}) is an isomorphism and that 
the diagram (\ref{eqn:MorelSquare}) is cartesian for all $n\geq 0$.
Then for $n\geq 3$ odd we have exact sequences
$$H_n(SL_{n-1}A) \to H_n(SL_nA) \to 2K^M_n(A) \to 0,$$

$$0 \to I^n(A) \to H_{n-1}(SL_{n-1}A) \to H_{n-1}(SL_{n}A) \to 0.$$
\end{theorem}

The proof requires the following lemma.

\begin{lemma}
\label{lem:PreHutchinsonTaoComps}
Let $A$ be a local ring with infinite residue field for which
(\ref{eqn:MilnorConj}) is an isomorphism and the square
(\ref{eqn:MorelSquare}) is cartesian for all $n\geq 0$.
Then the following hold.
\begin{enumerate}
\item
\label{item1:lem:PreHutchinsonTaoComps}
The following sequence is exact
$$K^{MW}_n(A) \stackrel{h_n}{\to} K^{MW}_n(A) \stackrel{\eta_n}{\to} K^{MW}_{n-1}(A) \to K^M_{n-1}(A) \to 0$$
where $h_n$ and $\eta_n$ are multiplication with $h=1+\langle -1\rangle$ and $\eta$, respectively.
\item
\label{item2:lem:PreHutchinsonTaoComps}
For $n\geq 3$ odd, under the isomorphism of Theorem \ref{thm:HnIsKMW}, the boundary map 
$$\partial_n: H_n(SL_nA,SL_{n-1}A) \to H_{n-1}(SL_{n-1}A,SL_{n-2}A)$$
of the triple $(SL_nA,SL_{n-1}A,SL_{n-2}A)$ is multiplication with $\eta$.
\item
\label{item3:lem:PreHutchinsonTaoComps}
The following square is bicartesian for $n\geq 3$ odd
$$\xymatrix{
H_{n-1}(SL_{n-1}A) \ar@{->>}[r]  \ar@{->>}[d] & H_{n-1}(SL_{n}A) \ar@{->>}[d]\\
H_{n-1}(SL_{n-1}A,SL_{n-2}A) \ar@{->>}[r]  & H_{n-1}(SL_nA, SL_{n-2}A).
}$$
\end{enumerate}
\end{lemma}

\begin{proof}
The sequence in (\ref{item1:lem:PreHutchinsonTaoComps}) is exact since it is isomorphic to the exact sequence
\begin{equation}
\label{eqn:MWexSeqAsFibProd}
I^n\times_{k^M_n}K^M_n \stackrel{(0,2)}{\longrightarrow} I^n\times_{k^M_n}K^M_n \stackrel{(1,0)}{\longrightarrow} I^{n-1}\times_{k^M_{n-1}}K^M_{n-1} \longrightarrow K^M_{n-1} \to 0
\end{equation}
in view of the cartesian square (\ref{eqn:MorelSquare}).

We prove (\ref{item2:lem:PreHutchinsonTaoComps}).
Under the isomorphism $\hat{K}^{MW}_n(A)\cong K^{MW}_n(A)$ for $n\geq 2$ proved in Theorem \ref{thm:KhatIsK}, multiplication by $\eta\in K^{MW}_{-1}(A)$ corresponds to the map 
$$\eta_n:\hat{K}^{MW}_n(A) \to \hat{K}^{MW}_{n-1}(A): [a_1,...,a_n]\mapsto \langle\langle a_n\rangle \rangle [a_1,...,a_{n-1}].$$
We will show that for all odd $n\geq 1$ the map in  (\ref{item2:lem:PreHutchinsonTaoComps}) is $\eta_n$.
This is clear for $n=1$.
The map
\begin{equation}
\label{eqn:BtoKMW}
\B(A)=\bigoplus_{n\geq 0}H_n(SL_nA) \to \bigoplus_{n\geq 0} \hat{K}^{MW}_n(A)
\end{equation}
is a $\Z[A^*]$-algebra homomorphism which is surjective in even degrees (Theorem \ref{thm:HnIsKMW}).
Moreover, the maps in (\ref{item2:lem:PreHutchinsonTaoComps}) assemble to a map of left $B$-modules
$$\partial:\bigoplus_{n\geq 0} H_n(SL_nA,SL_{n-1}A) \to \bigoplus_{n\geq 0} H_n(SL_n(A),SL_{n-1}A) \cong \hat{K}^{MW}(A).$$
For $n\geq 3$ odd and $[a_1,...,a_n]\in H_n(SL_nA,SL_{n-1}A)=\hat{K}^{MW}_n(A)$ choose a lift $b\in \B_{n-1}(A)$ of $[a_1,...,a_{n-1}]$ for the map (\ref{eqn:BtoKMW}).
Then
$$\partial_n([a_1,...,a_n]) = b\cdot \partial_1([a_n]) = b\cdot \langle \langle a_n\rangle\rangle = \langle \langle a_n\rangle\rangle [a_1,...,a_{n-1}].$$
This proves (\ref{item2:lem:PreHutchinsonTaoComps}).

We prove (\ref{item3:lem:PreHutchinsonTaoComps}).
The horizontal maps in diagram (\ref{item3:lem:PreHutchinsonTaoComps}) are surjective because we have $H_{n-1}(SL_nA,SL_{n-1}A)=0$.
The left vertical map is surjective, by Theorem \ref{thm:HnIsKMW}.
It follows that the right vertical map is also surjective.
The total complex of the square in (\ref{item3:lem:PreHutchinsonTaoComps}) is part of a Mayer-Vietoris type long exact sequence with boundary map the composition
$$H_n(SL_nA,SL_{n-2}A) \stackrel{\alpha_n}{\to} H_n(SL_nA,SL_{n-1}A) \stackrel{\delta_n}{\to} H_{n-1}(SL_{n-1}A).$$
Thus, the square in (\ref{item3:lem:PreHutchinsonTaoComps}) is bicartesian if and only if $\delta_n\alpha_n=0$.
From (\ref{item1:lem:PreHutchinsonTaoComps}) and (\ref{item2:lem:PreHutchinsonTaoComps}) we have
$\im(\alpha_n) = \ker(\eta_n) = \im(h_n)$.
Hence, the square in (\ref{item3:lem:PreHutchinsonTaoComps}) is bicartesian if and only if $\delta_nh_n=0$.
For $n=3$, the square in (\ref{item3:lem:PreHutchinsonTaoComps}) is bicartesian
since the vertical maps are isomorphisms.
In particular, $\delta_3h_3=0$.
Now, the map
$$\delta: \bigoplus_{n\geq 0} H_n(SL_n(A),SL_{n-1}A) \to \bigoplus_{n\geq 0}H_n(SL_nA)$$
is a left $\B(A)$-module map.
Take $[a_1,...,a_n]\in K^{MW}_n(A)=H_n(SL_nA,SL_{n-1}A)$ where $n\geq 5$ is odd.
Then the element $[a_1,...,a_{n-3}]\in \hat{K}^{MW}_{n-3}(A)$ lifts to $b\in \B_{n-3}(A)$
and
$$\delta_n(h\cdot [a_1,...,a_n]) = b\cdot \delta_3h_3 ([a_{n-2},a_{n-1},a_n]) = b\cdot 0 =0$$
\end{proof}

\begin{proof}[Proof of Theorem \ref{thm:KerCokerOdd}]
From Lemma \ref{lem:PreHutchinsonTaoComps} we have (using the notation of that lemma)
$$\ker\left(H_{n-1}(SL_{n-1}A) \to H_{n-1}(SL_nA)\right) = \im(\partial_n) = \im(\eta_n) = I^n(A)$$
and
$$\coker\left(H_n(SL_{n-1}A) \to H_n(SL_nA)\right) = \ker(\partial_n) = \ker(\eta_n) = \im(h_n)=2K^M_n(A)$$
where the last equality follows because of the exact sequence of Lemma \ref{lem:PreHutchinsonTaoComps} (\ref{item1:lem:PreHutchinsonTaoComps})
being isomorphic to (\ref{eqn:MWexSeqAsFibProd}).
\end{proof}

\section{Euler class groups}
\label{sec:EulerClGps}

Let $X$ be a separated noetherian scheme, and denote by $\Open_X$ the category of Zariski open subsets of $X$ and inclusions thereof as morphisms.
For a simplicial presheaf $F:\Open_X^{op}\to \sSets$ on the small Zariski site of $X$, we denote by 
$$[X,F]_{\Zar} = \pi_0(F_{\Zar}(X))$$
the set of maps from $X$ to $F$ in the homotopy category of simplicial presheaves on $X$ for the Zariski-topology \cite{BrownGersten}.
This is the set of 
path components of the simplicial set $F_{\Zar}(X)$ where $F \to F_{\Zar}$ is a map of simplicial presheaves which induces a weak equivalence of simplicial sets $F_x \to (F_{\Zar})_x$ on all stalks, $x\in X$, and $F_{\Zar}$ is object-wise weakly equivalent to its fibrant model in the Zariski topology \cite{BrownGersten}.
The latter means that $F_{\Zar}(\emptyset)$ is contractible and $F_{\Zar}$ sends a square
\begin{equation}
\label{eqn:MV}
\renewcommand\arraystretch{1.5}
\begin{array}{ccc}
U\cup V & \leftarrow & U \\
\uparrow & & \uparrow \\
V & \leftarrow & U\cap V
\end{array}
\end{equation}
of inclusions of open subsets of $X$ to a homotopy cartesian square of simplicial sets; see \cite[Theorem 4]{BrownGersten}.

\begin{example}
\label{ex:BGLn}
For the presheaf $BGL_n$ defined by $U \mapsto BGL_n(\Gamma(U,O_X))$ write
$B_{\Zar}GL_n$ for $(BGL_n)_{\Zar}$.
Then $B_{\Zar}GL_n(X)$ is a (functorial) model of the classifying space $B\Vect_n(X)$ of the category $\Vect_n(X)$ of rank $n$-vector bundles on $X$ with isomorphisms as morphisms.
For the inclusion of the automorphisms of $O_X^n$ into $\Vect_n(X)$ induces a map of simplicial presheaves
$BGL_n \to B\Vect_n$ which is a weak equivalence at the stalks of $X$ as vector bundles over local rings are free.
Moreover, $B_{\Zar}GL_n = B\Vect_n$ sends the squares (\ref{eqn:MV}) to homotopy cartesian squares, by an application of Quillen's Theorem B \cite{quillen:higherI}, for instance.
Hence, 
$$\Phi_n(X)=[X,BGL_n]_{\Zar}$$
is the set $\pi_0B\Vect_n(X)$ of isomorphism classes of rank $n$ vector bundles on $X$.
\end{example}

\begin{example}
Similar to Example \ref{ex:BGLn},  
the simplicial presheaf $BSL_n$ defined by $U \mapsto BSL_n(\Gamma(U,O_X))$ has a model $B_{\Zar}SL_n$ where $B_{\Zar}SL_n(X)$ is the classifying space $B\Vect^+_n(X)$  of the category $\Vect^+_n(X)$ of oriented rank $n$-vector bundles on $X$ with isomorphisms as morphisms.
Here, an oriented vector bundle of rank $n$ is a pair $(V,\omega)$ consisting of a vector bundle $V$ of rank $n$ and an isomorphism $\omega: \Lambda^nV\cong O_X$ of line bundles called orientation.
Morphisms of oriented vector bundles are isomorphisms of vector bundles preserving the orientation.
So, the set 
$$\Phi_n^+(X)=[X,BSL_n]_{\Zar}$$
is the set $\pi_0B\Vect^+_n(X)$ of isomorphism classes of rank $n$ oriented vector bundles on $X$.
\end{example}

\begin{example}
\label{ex:EMandCoh}
Let $n\geq 2$ be an integer, and let $F$ be a pointed simplicial presheaf such that $\pi_i(F_x)=0$ for $i\neq n$ and $x\in X$ where $F_x$ denotes the stalk of $F$ at $x\in X$.
Then there is a natural bijection of pointed sets
$$[X,F]_{\Zar}\cong H^n_{\Zar}(X,\tilde{\pi}_nF)$$
where the right hand side denotes Zariski cohomology of $X$ with coefficients in the sheaf of abelian groups $\tilde{\pi}_nF$ associated with the presheaf $U \mapsto \pi_n(F(U))$ \cite[Propositions 2 and 3]{BrownGersten}.
\end{example}

Denote by $\H_n(SL_n,SL_{n-1})$ the Zariski sheaf associated to the presheaf 
$$U\mapsto H_n(SL_n(\Gamma(U,O_X)),SL_{n-1}(\Gamma(U,O_X))).$$
Similarly, denote by $\K_n^{MW}$ the sheaf associated to the presheaf
$U\mapsto K_n^{MW}(\Gamma(U,O_X))$.

\begin{lemma}
\label{lem:SheafHnSLnIsKMW}
Let $X$ be a scheme with infinite residue fields.
Then for $n\geq 2$ there is an isomorphism of sheaves of abelian groups on $X$
$$\H_n(SL_n,SL_{n-1}) \cong \K_n^{MW}.$$
\end{lemma}

\begin{proof}
Let $A$ be a commutative ring.
Recall from \S \ref{subsec:SpSeqLocRings}
the graded $\Z[A^*]$-algebra $S(A)=\bigoplus_{n\geq 0}S_n(A)$.
It has $S_0(A)=\Z[A^*]$ and $S_1(A) =I[A^*]$.
Denote by $\SSS$ the sheaf of graded algebras associated with the presheaf
$U \mapsto S(\Gamma(U,O_X))$.
Let $2\leq n \leq n_0$ and $\sigma \in \Z[A^*]$ as in $(\ast)$.
By Corollary \ref{cor:Centrality}, the element $[-1,1] \in \SSS_2(X)$ is central in $\sigma^{-1}\SSS_{\leq n}(X)$, and we can define the quotient sheaf of algebras
$$\sigma^{-1}\bar{\SSS}_{\leq n} = \sigma^{-1}\SSS_{\leq n}/[-1,1]\SSS_{\leq n-2}.$$
The map of graded algebras
$\Tens_{\Z[A^*]}I[A^*] \to S(A)$ induced by the identity in degrees $0$ and $1$
induces the homomorphism of sheaves of algebras 
$\hat{\K}^{MW}_{\leq n} \to \sigma^{-1}\bar{\SSS}_{\leq n}$, 
by Proposition \ref{prop:SteinbergInSbar}.
Together with the map (\ref{eqn:CanMapToSn}), which is defined even when $A$ is not local, we obtain 
the diagram of sheaves
$$\H_n(SL_n,SL_{n-1}) \longrightarrow \sigma^{-1}\bar{\SSS}_n \longleftarrow \hat{\K}^{MW}_n$$
in which the maps are isomorphisms,
by Lemma \ref{lem:RelHisSn}, and Proposition \ref{prop:SnisKMWn}.
Since $n_0\geq 2$ can be any integer, we have 
$$\H_n(SL_n,SL_{n-1}) \cong \hat{\K}^{MW}_n$$
for all $n\geq 2$.
Finally, by Theorem \ref{thm:KhatIsK}, the natural surjection of sheaves of graded algebras
$\hat{\K}^{MW} \to {\K}^{MW}:[a]\mapsto [a]$ 
is an isomorphism in degrees $\geq 2$.
\end{proof}

For a commutative ring $R$, we denote by $\Delta R$ the standard simplicial ring
$$n \mapsto \Delta_nR = R[T_0,...,T_n]/(-1+T_0+\cdots + T_n).$$
For an integer $n\geq 0$, let $\tilde{E}_n(R)$ be the maximal perfect subgroup of the kernel of $GL_n(R) \to \pi_0GL_n(\Delta R)$.
Note that $\tilde{E}_n(R)$ is in fact a subgroup of $SL_n(R)$ since it maps to zero in the commutative group $GL_n(R)/SL_n(R)=R^*$.

\begin{lemma}
Let $X = \Spec R$, and $x\in X$.
Let $n\geq 1$ be an integer.
If $n\neq 2$ or $n=2$ and 
the residue field $k(x)$ of $x$ has more than $3$ elements, then
the inclusion $\tilde{E}_n \subset SL_n$ of presheaves induces an isomorphism on stalks at $x$
$$(\tilde{E}_n)_x \cong SL_n(O_{X,x}).$$
\end{lemma}

\begin{proof}
The statement is trivial for $n=1$.
So we assume $n\geq 2$.
Every elementary matrix $e_{i,j}(r)\in GL_n(R)$, $r\in R$, is the evaluation at $T=1$ of the elementary matrix $e_{i,j}(Tr)\in GL_n(R[T])$ whereas the evaluation at $T=0$ yields $1$.
Therefore, on fundamental groups the natural map $BGL_n(R) \to BGL_n(\Delta R)$ sends $E_n(R)$ to $1$.

If $n\geq 3$, the group $E_n(R)$ is perfect.
If $n=2$, our hypothesis implies that there is $f\in R$ such that $0\neq f \in k(x)$ and $R_f$ has a unit $u$ such that $u+1$ and $u-1$ are also units. 
Replacing $R$ with $R_f$, we can assume that $R$ and any localization $A$ of $R$ has this property.
For such rings $A$ the group $E_2(A)$ is perfect since for all $a\in A$ we have
$$
\left[ \left(\begin{smallmatrix} u & 0 \\ 0 & u^{-1} \end{smallmatrix}\right) , 
\left(\begin{smallmatrix} 1 & (1-u^2)^{-1}a \\ 0 & 1 \end{smallmatrix}\right) \right]
= 
\left(\begin{smallmatrix} 1 & a \\ 0 & 1 \end{smallmatrix}\right),
$$
and $\left(\begin{smallmatrix} u & 0 \\ 0 & u^{-1} \end{smallmatrix}\right)$ is a product of elementary matrices.

In any case, we can assume $E_n(A)$ perfect and contained in $\tilde{E}_n(A)$ for all localizations $A$ of $R$.
From the inclusions $E_n(A) \subset \tilde{E}_n(A) \subset SL_n(A)$
we obtain the inclusions of corresponding stalks $(E_n)_x \subset (\tilde{E}_n)_x \subset (SL_n)_x = SL_n(O_{X,x})$.
Since the composition is an isomorphism, the lemma follows.
\end{proof}

We will write $BGL_n^+(R)$ and $BSL_n^+(R)$
 for a functorial version of Quillen's plus-construction applied to $BGL_n(R)$ and $BSL_n(R)$ with respect to the perfect normal subgroup $\tilde{E}_n(R)$.
See \cite[VII \S 6]{BousfieldKan} for how to make the plus-construction functorial.
The canonical inclusion $GL_{n-1}(R) \subset GL_n(R)$ induces maps
$BGL_{n-1}^+ \to BGL_{n}^+$ and $BSL_{n-1}^+ \to BSL_{n}^+$.

Recall from \cite[\S 8]{may:book} that for every integer $n\geq 0$ there is an endofunctor $P_{\leq n}:\sSets \to \sSets$ of the category of simplicial sets together with natural transformations $S \to P_{\leq n}S$ such that for every choice of base point $x\in S_0$, the map $\pi_i(S,x) \to \pi_i(P_{\leq n}S,x)$ is an isomorphism for $i\leq n$ and $\pi_i(P_{\leq n}S,x)=0$ for $i>n$.
Moreover, the map $S \to P_{\leq n}S$ factors naturally as $S \to P_{\leq n+1}S \to P_{\leq n}S$.
If $F$ is a simplicial presheaf then $P_{\leq n}F$ denotes the presheaf $U\mapsto P_{\leq n}(F(U))$.
For a pointed simplicial presheaf $F$ we denote by 
$\tilde{\pi}_iF$ the Zariski sheaf associated with the presheaf
 $U\mapsto \pi_i(F(U),x_0)$ where $x_0$ is the base point of $F$.
We consider the quotient $B/A$ of an inclusion of simplicial sets $A \subset B$ and $P_{\leq n}(B/A)$ pointed at $\{A\}$.

\begin{lemma}
\label{lem:TruncatedQuotientIsEM}
Let $X$ be a scheme with infinite residue fields.
Then for $n\geq 2$ there are isomorphisms of Zariski sheaves on $X$
$$\tilde{\pi}_iP_{\leq n}(BSL_n^+/BSL_{n-1}^+) \cong \left\{
\begin{array}{ll} 
0 & i\neq n\\
\K^{MW}_n& i=n.
\end{array}\right.
$$
\end{lemma}

\begin{proof}
From the properties of $P_{\leq n}$, the statement is clear for $i>n$.
So, assume $i\leq n$.
We have the following string of sheaves
$$\tilde{\pi}_iP_{\leq n}(BSL_n^+/BSL_{n-1}^+) \stackrel{\cong}{\leftarrow}
\tilde{\pi}_i(BSL_n^+/BSL_{n-1}^+) \to  \H_i(SL_n,SL_{n-1})$$
where the right arrow is the Hurewicz homomorphism which is an isomorphism for $i\leq n$, by Theorem \ref{thm:HnIsKMW} and the fact that $BSL_n^+R/BSL_{n-1}^+R$ is simply connected for local rings $R$ (with infinite residue field).
Using Lemma \ref{lem:SheafHnSLnIsKMW}, the claim follows.
\end{proof}

\begin{corollary}
\label{cor:TrunsAnKMW}
Let $X$ be a noetherian separated scheme with infinite residue fields.
Then for $n\geq 2$, there is a natural bijection of pointed sets
$$[X,P_{\leq n}(BSL_n^+/BSL_{n-1}^+)]_{\Zar}\cong H^n_{\Zar}(X,\K^{MW}_n).$$
\end{corollary}

\begin{proof}
This is Example \ref{ex:EMandCoh} and Lemma \ref{lem:TruncatedQuotientIsEM}.
\end{proof}

\begin{definition}
\label{dfn:Eulermap}
The {\em Euler class map} (for rank $n$ oriented vector bundles) is the composition of maps of simplicial presheaves
$$e: BSL_n^+ \to BSL_n^+/BSL_{n-1}^+ \to P_{\leq n}(BSL_n^+/BSL_{n-1}^+).$$
By definition, it is trivial when restricted to $BSL_{n-1}^+$.
\end{definition}

Let $n\geq 2$. 
In view of Corollary \ref{cor:TrunsAnKMW}, applying the functor $F\mapsto [X,F]_{\Zar}$ to the sequence 
$$BSL_{n-1}^+ \to BSL_n^+  \stackrel{e}{\to} P_{\leq n}(BSL_n^+/BSL_{n-1}^+)$$
yields the sequence
\begin{equation}
\label{eqn:PlusExSeq}
[X,BSL_{n-1}^+]_{\Zar} \longrightarrow [X,BSL_n^+]_{\Zar} \stackrel{e}{\longrightarrow} H^n_{\Zar}(X,\K^{MW}_n).
\end{equation}
\vspace{1ex}

A sequence $U \to V \to W$ of sets with $W$ pointed is called {\em exact} if every element of $V$ which is sent to the base point in $W$ comes from $U$.

\begin{theorem}
\label{thm:ObstPlusExSeq}
Let $n\geq 2$ be an integer and
let $X$ be a noetherian separated scheme with infinite residue fields.
Assume that the dimension of $X$ is at most $n$.
Then the sequence of sets (\ref{eqn:PlusExSeq}) is exact.
\end{theorem}

\begin{proof}
This follows from obstruction theory \cite[Corollary B.10]{morel:book} in view of Lemma \ref{lem:TruncatedQuotientIsEM}.
\end{proof}

One would like to replace $BSL_r^+$ with $BSL_r$ for $r=n-1,n$ in 
Theorem \ref{thm:ObstPlusExSeq}.
This motivates the following.

\begin{question}
\label{quest1}
For which (affine) noetherian scheme $X$ is the canonical map 
$$[X,BGL_n]_{\Zar} \to [X,BGL_n^+]_{\Zar}$$
a bijection?
\end{question}

Unfortunately, we don't know the answer to Question \ref{quest1} (other than for $n=1$ or when $X$ is affine and $\dim X > n$).
Instead we will prove a weaker version in Corollary  \ref{cor:BGLnPlusFactor} below.
This will be sufficient for our application.

\begin{lemma}
\label{lem:BZarDescent}
Let $R$ be a commutative ring.
If $f,g\in R$ with $fR+gR=R$, then the following diagram is homotopy cartesian
$$\xymatrix{
B_{\Zar}GL_n(R) \ar[r] \ar[d] & B_{\Zar}GL_n(R_f) \ar[d]\\
B_{\Zar}GL_n(R_g) \ar[r] & B_{\Zar}GL_n(R_{fg}).
}$$
\end{lemma}

\begin{proof}
This follows from descent and can also be checked using Quillen's Theorem B \cite{quillen:higherI}.
\end{proof}

For $n\in \N$, we write $\Vect^R_n(R[T_1,...,T_n])$ for the full subcategory of  the category
$\Vect_n(R[T_1,...,T_n])$ of those projective modules $P$ which are extended from $R$, that is, which are isomorphic to $Q\otimes_RR[T_1,...,T_n]$ for some $Q\in \Vect_n(R)$.
Write $B^R_{\Zar}GL_n(R[T_1,...,T_n])$ for the classifying space (that is, nerve) of the category $\Vect^R_n(R[T_1,...,T_n])$.

\begin{lemma}
\label{lem:BZarPolynomialDescent}
Let $R$ be a commutative ring.
If $f,g\in R$ with $fR+gR=R$, then for all integers $q\geq 0$ the following diagram is homotopy cartesian
$$\xymatrix{
B_{\Zar}^R GL_n(R[T_1,...,T_q])
 \ar[r]
\ar[d]  
& B^{R_f}_{\Zar} GL_n(R_f[T_1,...,T_q]) 
\ar[d]
\\
B_{\Zar}^{R_g}GL_n(R_g[T_1,...,T_q]) \ar[r] & B_{\Zar}^{R_{fg}}GL_n(R_{fg}[T_1,...,T_q]).
}$$
\end{lemma}

\begin{proof}
By Quillen's Patching Theorem \cite[Theorem 1']{quillen:Serre}, a projective $R[T_1,...,T_q]$-module $P$ is extended from $R$ if and only if $P_f$ and $P_g$ are extended from $R_f$ and $R_g$.
Hence, the lemma follows from Lemma \ref{lem:BZarDescent} with $R[T_1,...,T_q]$ in place of $R$.
\end{proof}

Write $B_{\Zar}^RGL_n(\Delta R)$ for the diagonal of the simplicial space
$q \mapsto B_{\Zar}^RGL_n(\Delta_q R)$.

\begin{corollary}
\label{cor:BZarDeltaDescent}
Let $R$ be a commutative ring.
If $f,g\in R$ with $fR+gR=R$, then the following diagram of simplicial sets is homotopy cartesian
$$\xymatrix{
B_{\Zar}^RGL_n(\Delta R) \ar[r] \ar[d] & B^{R_f}_{\Zar}GL_n(\Delta  R_f) \ar[d]\\
B_{\Zar}^{R_g}GL_n(\Delta R_g) \ar[r] & B_{\Zar}^{R_{fg}}GL_n(\Delta R_{fg}).
}$$
\end{corollary}

\begin{proof}
For $q\in \N$, the diagram is homotopy cartesian for $\Delta_q$ in place of $\Delta$, in view of Lemma \ref{lem:BZarPolynomialDescent}.
The Corollary now follows from the Bousfield-Friedlander Theorem \cite[Theorem B.4]{BousfieldFriedlander} which we can apply since the simplicial set $q\mapsto \pi_0B_{\Zar}^RGL_n(\Delta_q R)$ of connected components is a constant simplicial set for any $R$.
\end{proof}

Write $B_{\Zar}^{\bullet}GL_n^{\Delta}$  the simplicial presheaf 
$$X \mapsto B_{\Zar}^RGL_n(\Delta R),\hspace{3ex}\text{where}\hspace{1ex}R=\Gamma(X,O_X).$$
Inclusion of its degree zero space into the simplicial space induces a map of simplicial presheaves
$B_{\Zar}GL_n \to B_{\Zar}^{\bullet}GL_n^{\Delta}$.

\begin{theorem}
\label{thm:SingBGLn}
Let $X=\spec R$ where $R$ is a noetherian ring.
Then the natural maps of simplicial presheaves $BGL_n \to B_{\Zar}GL_n \to  B_{\Zar}^{\bullet}GL_n^{\Delta}$ induce a bijection
$$[X,BGL_n]_{\Zar} \cong [X,B_{\Zar}^{\bullet}GL_n^{\Delta}]_{\Zar}.$$
\end{theorem}

\begin{proof}
This follows from Corollary \ref{cor:BZarDeltaDescent} in view of Theorem \ref{thm:AffineBGToBG}.
\end{proof}

We can reformulate the theorem as follows.
For a simplicial presheaf $F$ defined on the category of schemes, we write $\Sing F$ for the simplicial presheaf $X\mapsto (q\mapsto F(X\times \Spec \Delta_q\Z))$.
The map of simplicial rings $\Z\to \Delta$ induces a natural map $F \to \Sing F$ of simplicial presheaves.

\begin{theorem}
\label{thm:SingBGRepsVect}
Let $X=\spec R$ where $R$ is a noetherian ring.
Then the natural map of simplicial presheaves $BGL_n \to \Sing BGL_n$ induces a bijection
$$\Phi_n(X)=[X,BGL_n]_{\Zar} \cong [X,\Sing BGL_n]_{\Zar}.$$
\end{theorem}

\begin{proof}
This follows from Theorem \ref{thm:SingBGLn} since the natural map of simplicial presheaves
$\Sing BGL_n \to B_{\Zar}^{\bullet}GL_n^{\Delta}$ is a weak equivalence at the local rings of $X$.
\end{proof}

By definition of the presheaf of perfect groups $\tilde{E}_n$, the canonical map of simplicial presheaves $BGL_n \to  \Sing BGL_n$ factors through $BGL_n^+$.
From Theorem \ref{thm:SingBGRepsVect} we therefore obtain the following.

\begin{corollary}
\label{cor:BGLnPlusFactor}
Let $X=\Spec R$ be an affine noetherian scheme.
Then the string of maps of simplicial presheaves
$BGL_n \to BGL_n^+ \to \Sing BGL_n$ induces the sequence of maps
$$[X,BGL_n]_{\Zar} \to [X,BGL_n^+]_{\Zar} \to [X,\Sing BGL_n]_{\Zar}$$
whose composition is a bijection.
\end{corollary}

\begin{definition}
\label{dfn:EulerClass}
Let $X$ be a scheme with infinite residue fields and $V$ an oriented rank $n$ vector bundle on $X$.
The {\em Euler class} 
$$e(V)\in H^n_{\Zar}(X,\K^{MW}_n)$$ of $V$ is the image of $[V]\in [X,BSL_n]_{\Zar}$  under the canonical map
$$[X,BSL_n]_{\Zar} \to [X,BSL_n^+]_{\Zar} \stackrel{e}{\to}  H^n_{\Zar}(X,\K^{MW}_n).$$
By construction, we have
$e(W\oplus O_X) = 0$ for any rank $n-1$ oriented vector bundle $W$.
\end{definition}

\begin{theorem}
\label{thm:GeneralEulerclass}
Let $R$ be a commutative noetherian ring of dimension $n\geq 2$.
Assume that all its residue fields are infinite.
Let $P$ be an oriented rank $n$ projective $R$-module.
Then 
$$P\cong Q\oplus R \Leftrightarrow e(P)=0 \in H^n_{\Zar}(R,\K^{MW}_n).$$
\end{theorem}

\begin{proof}
We already know that $e(Q\oplus R)=0$.
So assume $e(P)=0$.
In view of Corollary \ref{cor:BGLnPlusFactor}, 
the maps of simplicial presheaves
$$BSL_r \to BSL_r^+ \to BGL_r^+ \to \Sing BGL_r$$
induce a commutative diagram 
$$\xymatrix{
[X,BSL_{n-1}]_{\Zar} \ar[r] \ar[d] & [X,BSL_{n-1}^+]_{\Zar} \ar[r] \ar[d] & [X,BGL_{n-1}]_{\Zar} \ar[d]\\
[X,BSL_n]_{\Zar} \ar[r] & [X,BSL_n^+]_{\Zar} \ar[r] & [X,BGL_n]_{\Zar}
}$$
where the horizontal composition is the map which forgets the orientation.
The commutativity of this diagram together with Theorem \ref{thm:ObstPlusExSeq} and the hypothesis $e(P)=0$ implies the result.
\end{proof}

\begin{remark}
Theorem \ref{thm:GeneralEulerclass} is a generalization of a theorem of Morel \cite[Theorem 8.14]{morel:book} who proved it for $X$ smooth affine over an infinite perfect field.
To compare the two versions, note that instead of our Milnor-Witt $K$-theory sheaf, Morel uses the unramified Milnor-Witt $K$-theory sheaf.
But for a smooth $X$ over an infinite field of characteristic not $2$, the canonical map from our Milnor-Witt $K$-sheaf to Morel's Milnor-Witt $K$-sheaf is an isomorphism which follows from the exactness of the Gersten complex for Milnor-Witt $K$-theory of regular local rings containing an infinite field of characteristic not $2$ \cite{GilleEtAl}.
Moreover, Morel uses Nisnevich cohomology instead of Zariski cohomology.
Again because of the exactness of the Gersten complex for $K^{MW}$, the change of topology map is an isomorphism for $X$ smooth over an infinite field of characteristic not $2$:
$$H_{\Zar}^*(X,\K^{MW}) \cong H_{\Nis}^*(X,\K^{MW}).$$
\end{remark}

\begin{remark}
\label{rmk:OrientInL}
Let $L$ be a line bundle on $X=\Spec R$.
Theorem \ref{thm:GeneralEulerclass} has an evident generalization to rank $n$ vector bundles $P$ with orientation $w:\Lambda^n_RP \cong L$ in $L$.
Equip $R^{n-1}\oplus L$ with the canonical orientation $\Lambda^n_R(R^{n-1}\oplus L) \cong \Lambda^{n-1}_RR^{n-1}\otimes \Lambda^1_RL = L$, and denote by 
$SL_n^L(R)$ the group of orientation perserving $R$-linear automorphisms of $R^{n-1}\oplus L$.
Then 
$$\Phi_n^L(X)=[X,BSL_n^L]_{\Zar}$$ 
is the set of isomorphism classes of rank $n$ vector bundles on $X$ with orientation in $L$.
Define the sheaf $\K^{MW}_n(L)$ on $X$ as 
$$\K^{MW}_n(L) = \H_n(SL_n^L,SL_{n-1}^L).$$
Its stalks are, of course, the usual Milnor-Witt $K$-groups of the local rings of $X$.
Replacing $SL_n$ with $SL_n^L$ everywhere, we obtain an Euler class map as in Definition \ref{dfn:Eulermap} and an Euler class 
$e(P,L)\in H^n(R,\K^{MW}(L))$
for projective modules $P$ with orientation in $L$ as in Definition \ref{dfn:EulerClass}.
\end{remark}

With the definitions in Remark \ref{rmk:OrientInL} we have the following theorem whose proof is {\it mutatis mutandis} the same as in the case $L=R$ in Theorem \ref{thm:GeneralEulerclass}.

\begin{theorem}
\label{thm:GeneralEulerclassOrientL}
Let $R$ be a commutative noetherian ring of dimension $n\geq 2$.
Assume that all its residue fields are infinite.
Let $L$ be a line bundle on $R$.
Let $P$ be a rank $n$ projective $R$-module with orientation in $L$.
Then 
$$P\cong Q\oplus R \Leftrightarrow e(P,L)=0 \in H^n_{\Zar}(R,\K^{MW}_n(L)).$$
\end{theorem}

For a field $k$, denote by ${\H(k)}$ the Morel-Voevodsky unstable $\aaa^1$-homotopy category of smooth schemes over $k$ \cite{morelvoev}.
Recall that $\Phi_n(X)$ denotes the set of isomorphism classes of rank $n$ vector bundles on the scheme $X$.
The arguments in the proof of Theorem \ref{thm:SingBGLn} 
can be used to give a simple proof of a theorem of Morel \cite[Theorem 8.1 (3)]{morel:book}.
Note that we do not need to exclude the case $n=2$.

\begin{theorem}[Morel]
\label{thm:MorelReps}
Let $k$ be an infinite perfect field.
Then for any smooth affine $k$-scheme $X$, there is a natural bijection
$$\Phi_n(X)\cong [X,BGL_n]_{\H(k)}.$$
\end{theorem}

\begin{proof}
Let $R$ be a smooth $k$-algebra.
For each $q\geq 0$, the simplicial presheaf 
$$B_{\Zar}GL_n\Delta_q = B\Vect_n\Delta_q$$
has the affine B.G.-property for the Zariski and the Nisnevich topology (see \cite{morel:book} for the definition), by descent, or an application of Quillen's theorem B.

By a result of Lindel \cite{Lindel},  
for any smooth $k$-algebra $R$, extension by scalars induces a bijection  $\Phi_n(R)\cong \Phi_n(R[T])$.
In other words, the simplicial set of connected components $q\mapsto \Phi_n(\Delta_qR) =\pi_0  B_{\Zar}GL_n(\Delta_qR)$ of the simplicial space 
$q \mapsto B_{\Zar}GL_n\Delta_qR$
is constant.
In view of the Bousfield-Friedlander Theorem \cite[Theorem B.4.]{BousfieldFriedlander}, it follows that
the diagonal $\Sing B_{\Zar}GL_n$ of the bisimplicial presheaf $q \mapsto B_{\Zar}GL_n\Delta_q$ has the affine B.G.-property for the Zariski and the Nisnevich topology. 
By construcion, the simplicial presheaf 
$\Sing B_{\Zar}GL_n$
is $\aaa^1$-invariant.
We will show that the map of simplicial presheaves
\begin{equation}
\label{eqn:SingToLASing}
\Sing B_{\Zar}GL_n \to L_{\aaa^1} \Sing B_{\Zar}GL_n
\end{equation}
is a weak equivalence on affine $k$-schemes, by an application of \cite[Theorem A.19]{morel:book}.
We already know that the source of (\ref{eqn:SingToLASing}) is $\aaa^1$-invariant, and satisfies the affine B.G.-property for the Zariski and the Nisnevich topology.
Furtheremore, 
$\Sing GL_n =  \Omega^1_s \Sing B_{Zar}GL_n$  has the affine B.G.-property for the Nisnevich topology because $\Sing B_{Zar}GL_n$ has.
The $\pi_0$ sheaf of $\Sing B_{Zar}GL_n$ is trivial in the Zariski topology because over a local ring every rank $n$ vector bundle is trivial.
Finally, the $\pi_1$ sheaf of $\Sing B_{Zar}GL_n$ in the Zariski topology is the $\pi_0$-sheaf of $\Sing GL_n$ which is the group of units (for integral schemes), hence strongly $\aaa^1$-invariant, since for a local ring $R$ and $n\geq 1$, we have $SL_nR = E_nR$.
In view of \cite[Theorem A.19]{morel:book},
the map (\ref{eqn:SingToLASing}) is a weak equivalence on affine $k$-schemes.

Now, the map $BGL_n \to \Sing B_{\Zar}GL_n$ is an $\aaa^1$-weak equivalence, and, by Lindel's theorem, we have $\Phi_n(R) = \pi_0 B_{\Zar}GL_n(\Delta R)$.
This finishes the proof.
\end{proof}

\begin{remark} 
Similar arguments apply to the symplectic groups $Sp_n$ in place of $GL_n$.
\end{remark}

\appendix

\section{The affine B.G.-property for the Zariski topology}

\begin{definition}
Let $X$ be a scheme and let $F:\Open_X^{op} \to \sSets$ be a simplicial presheaf on $X$.
We say that $F$ has the {\em affine B.G.-property for the Zariski topology}
if $F(\emptyset)$ is contractible and for any affine $U=\Spec R \in \Open_X$ and $f,g\in R$ with $(f,g)=R$, the following square of simplicial sets is homotopy cartesian
\begin{equation}
\label{eqn:dfn:AffineBGZariski}
\xymatrix{
F(R) \ar[r] \ar[d] & F(R_f) \ar[d] \\
F(R_g) \ar[r] & F(R_{fg}).}
\end{equation}
\end{definition}

The aim of this appendix is to give a proof of the  following result due to Marc Hoyois \cite{Hoyois}.
Whereas Hoyois' proof uses $\infty$-categories, we give a proof in the framework of model categories based on standard manipulations of homotopy limits.

\begin{theorem}[Hoyois]
\label{thm:AffineBGToBG}
Let $X$ be a noetherian scheme and let
$F:\Open_X^{op} \to \sSets$ be a simplicial presheaf on $X$ which has the affine B.G.-property.
Then for all affine $Y\in \Open_X$, the following canonical map is a weak equivalence
$$F(Y) \stackrel{\sim}{\longrightarrow}  F_{\Zar}(Y),$$
where $F\to F_{\Zar}$ is a fibrant replacement of $F$ for the Zariski topology on $X$.
\end{theorem}

The proof will occupy the rest of this appendix.
We start by reviewing basic properties of homotopy limits \cite{BousfieldKan}, \cite{SchererEtAl}.

Let $f:\C \to \D$ be a functor between small categories.
For an object $D$ of $\D$, the category $(f\downarrow D)$ has objects pairs $(C,a)$ where $C$ is an object of $\C$ and $a:f(C)\to D$ is a map in $\D$.
A map $(C,a)\to (C',a')$ in $(f\downarrow D)$ is a map $C\to C'$ in $\C$ which makes the induced triangle in $\D$ commute. 
Composition is composition of maps in $\C$.
There is a similar category $(D\downarrow f)$ whose objects are pairs $(C,a:D\to fC)$.
When $f=id:\C \to \C$ is the identity functor, one writes $(\C\downarrow C)$ and $(C\downarrow \C)$ for $(id\downarrow C)$ and $(C\downarrow id)$.
For a small category $\C$, we denote by $B\C$ the classifying space of $\C$, that is, the nerve simplicial set of $\C$.

Let $F:\C \to \sSets$ be a functor from a small category $\C$ to simplicial sets.
Assume that $F$ is object-wise fibrant, that is, $FC$ is a fibrant simplicial set for all objects $C$ of $\C$.
Then the {\em homotopy limit} of $F$ over $\C$ is the simplicial set defined by the equalizer diagram
$$\holim_{\C}F \to  \prod_{C \in \C} \Hom(B(\C\downarrow C),F(C)) 
\xymatrix{ \ar@<1ex>[r]^a \ar@<-1ex>[r]_b & }
\prod_{\gamma:C \to C' \in \C} \Hom(B(\C\downarrow C),F(C'))
$$ 
where $a$ and $b$ are induced by 
$$\xymatrix{
\Hom(B(\C\downarrow C),F(C)) \ar[r]^{F\gamma} & \Hom(B(\C\downarrow C),F(C'))\\
\Hom(B(\C\downarrow C'),F(C')) \ar[r]^{(\C\downarrow \gamma)} & \Hom(B(\C\downarrow C),F(C')).
}$$
If $F$ is not object-wise fibrant, we define the homotopy limit of $F$ over $\C$ as the homotopy limit of $\Ex^{\infty}F$ over $\C$ as above
where $F \to \Ex^{\infty}F$ is Kan's fibrant replacement functor in the category of simplicial sets.
So, $F \to \Ex^{\infty}F$ is an object-wise weak equivalence, that is, $FC \to \Ex^{\infty}FC$ is a weak equivalence for all $C\in \C$, and $\Ex^{\infty}F$ is object-wise fibrant. 

The homotopy limit has the following useful properties.
\vspace{1ex}

\noindent
{\bf Functoriality}.
The homotopy limit $\holim_{\C}F$ is covariantly functorial in $F$ and contravariantly functorial in $\C$.
More precisely, 
define a category $[\Cat,\sSets]$ whose objects are pairs $(\C,F)$ where $\C$ is a small category and $F:\C \to \sSets$ is a functor.
Given two objects $(\C,F)$ and $(\D,G)$ of $[\Cat,\sSets]$, a morphism 
$(\C,F)\to (\D,G)$ in $[\Cat,\sSets]$ is a
pair $(f,\ffi)$  where $f:\C \to \D$ is a functor and $\ffi:f^*G \to F$ a natural transformation.
Composition is defined as $(g,\gamma)  \circ (f,\ffi) = (gf,\ffi\circ f^*(\gamma))$.
The homotopy limit defines a functor
\begin{equation}
\label{eqn:Functoriality1}
\holim: [\Cat,\sSets]^{op}\to \sSets
\end{equation}
which sends the map $(f,\ffi) : (\C,F)\to (\D,G)$ in $[\Cat,\sSets]$ to the map of simplicial sets
$$(f,\ffi)^*:\holim_{\D}G \to \holim_{\C}F$$
which is the composition of the two maps \cite[XI \S 3.2]{BousfieldKan}
$$\holim_{\D}G \stackrel{\holim(f)}{\longrightarrow} \holim_{\C}f^*G \stackrel{\holim(\ffi)}{\longrightarrow} \holim_{\C}F.$$
\vspace{1ex}

\noindent
{\bf Homotopy Lemma}.
Let $F \to F'$ be a natural transformation of functors $F,F':\C \to \sSets$ such that for all $C\in \C$ the map $F(C) \to F'(C)$ is a weak equivalence of simplicial sets, then the induced map on homotopy limits is a weak equivalence \cite[XI \S 5.6]{BousfieldKan}:
$$\holim_{\C}F \stackrel{\sim}{\longrightarrow} \holim_{\C}F'.$$
\vspace{1ex}

\noindent
{\bf Cofinality}.
A functor $f:\C \to \D$ between small categories is called {\em left cofinal}  if for every $D \in \D$, the classifyting space of the category $(f\downarrow D)$ is contractible.

Let $f:\C \to \D$ be a left cofinal functor.
Then for every functor $F:\D \to \sSets$, the induced map on homotopy limits is a weak equivalence \cite[XI \S 9.2]{BousfieldKan}:
$$(f,1)^*: \holim_{\D}F \stackrel{\sim}{\longrightarrow} \holim_{\C}f^*F.$$
\vspace{1ex}

\noindent
{\bf Fubini's theorem}.
A functor 
\begin{equation}
\label{eqn:Fubini1}
\C \to [\Cat,\sSets]^{op}:C\mapsto (\D_C,F_C)
\end{equation}
is given by the following data:
\begin{itemize}
\item
a functor $\D:\C^{op} \to \Cat:C\mapsto \D_C$,
\item
for every object $C\in \C$ a functor $F_C: \D_C\to \sSets$, and
\item
for every map $\gamma:C_0 \to C_1$ a natural transformation $\delta_{\gamma}: \D_{\gamma}^*F_{C_0} \to F_{C_1}$ such that $\delta_1=id$ and  $\delta_{\gamma_1\gamma_0} = 
\delta_{\gamma_1}\D_{\gamma_1}(\delta_{\gamma_0})$ for any two composable arrows $\gamma_0$, $\gamma_1$ in $\C$.
\end{itemize}
To give such data is equivalent to giving a functor
\begin{equation}
\label{eqn:Fubini2}
\xymatrix{F: \C \oint \D \to \sSets}
\end{equation}
where $\C\oint \D = (\C^{op}\int \D)^{op}$ is the opposite of the Grothendieck construction on the functor $\D:\C^{op}\to \Cat$.
In detail, $\C \oint \D$ is the category whose objects are pairs $(C,x)$ where $C$ is an object of $\C$ and $x$ is an object of $\D_C$.
A map $(C_0,x_0)\to (C_1,x_1)$ in $\C\oint \D$ is given by a pair $(\gamma,a)$ where $\gamma:C_0 \to C_1$ is a map in $\C$ and $a:x_0\to \D_{\gamma}x_1$ is a map in $\D_{C_0}$.
Composition is defined by $(\gamma_1,a_1)\circ (\gamma_0,a_0) = (\gamma_1\gamma_0,\D_{\gamma_0}(a_1)\circ a_0)$.
The functor (\ref{eqn:Fubini2}) induced by the collection of data above sends $(C,x)$ to $F_C(x)$ and a map $(\gamma,a)$ to $\delta_{\gamma}(x_1)\circ F_{C_0}(a)$.

The composition of the functors (\ref{eqn:Functoriality1}) and (\ref{eqn:Fubini1}) determine a functor
$$\C \to \sSets: C\mapsto \holim_{x\in \D_C}F_C(x)$$
which in turn defines a simplicial set 
$$\holim_{C\in \C}\holim_{x\in \D_C}F_C(x).$$
On the other hand, the functor (\ref{eqn:Fubini2}) also determines a simplicial set
$$\holim_{(C,x)\in \C\oint \D}F_C(x).$$
The Fubini Theorem for homotopy limits asserts that these two simplicial sets are naturally weakly equivalent \cite[III Theorem 26.8 and III 31.5]{SchererEtAl}:
\begin{equation}
\label{FubiniThm1}
\holim_{C\in \C}\holim_{x\in \D_C}F_C(x) \simeq \holim_{(C,x)\in \C\oint \D}F_C(x).
\end{equation}
If $\D:\C^{op}\to \Cat$ is a constant functor, that is, $\D_{\gamma}=id$ for all maps $\gamma$ in $\C$, then $\C\oint \D = \C\times \D$ and Fubini's Theorem reduces to a weak equivalence \cite[XI Example 4.3]{BousfieldKan}
\begin{equation}
\label{FubiniThm2}
\holim_{C\in \C}\holim_{x\in \D}F_C(x)\simeq \holim_{(C,x)\in \C\times \D}F_C(x).
\end{equation}
\vspace{1ex}

\noindent
{\bf Homotopy pull-backs}.
\vspace{1ex}
A commutative square of simplicial sets
$$\xymatrix{X \ar[r] \ar[d] & Y \ar[d] \\
Z \ar[r] & W}$$
is homotopy cartesian if and only if the natural map,
induced by the unique map from the index category to the final object in $\Cat$,
$$X \longrightarrow \holim(Y \to W \leftarrow  Z),$$
is a weak equivalence of simplicial sets \cite[XI Example 4.1 (iv)]{BousfieldKan}.
\vspace{1ex}

\noindent
{\bf Extended Functoriality}.
Let $(f_0,\ffi_0), (f_1,\ffi_1):(\C,F) \to (\D,G)$ be morphisms in $[\Cat,\sSets]$.
A natural transformation $\delta:(f_0,\ffi_0) \to (f_1,\ffi_1)$ in $[\Cat,\sSets]$ is 
a natural transformation of functors $\delta:f_0 \to f_1$ such that 
$\ffi_0 = \ffi_1 \circ G(\delta) : f_0^*G \to F$.

If there is a natural transformation 
$\delta:(f_0,\ffi_0) \to (f_1,\ffi_1)$ of maps in $[\Cat,\sSets]$ as above
then the induced maps on homotopy limits  
$$(f_0,\ffi_0)^*, (f_1,\ffi_1)^*:\holim_{\D}G \to \holim_{\C}F$$
are homotopic.

\begin{proof}
The Extended Functoriality is a consequence of Cofinality as follows.
Denote by $p:\C\times [1] \to \C$ the projection.
The two maps $(f_i,\ffi_i): (\C,F) \to (\D,G)$ are the two compositions in a diagram in $[\Cat,\sSets]$
$$\xymatrix{ 
(\C,F) \ar@<1ex>[r]^{\hspace{-4ex}(s_0,1)} \ar@<-1ex>[r]_{\hspace{-4ex}(s_1,1)} & 
(\C\times [1], p^*F) \ar[r]^{\hspace{4ex}(f,\ffi)} & (\D,G)
}$$
where $[1]$ is the poset $0<1$ and $s_i:\C \to \C\times [1]:C\mapsto (C,i)$ is the obvious inclusion, $i=0,1$.
The functor $p:\C \times [1] \to \C$ is left cofinal since for every $C \in \C$
the composition 
$$(\C\downarrow C) \stackrel{s_0}{\to} (p\downarrow C) \stackrel{p}{\to} (\C\downarrow C)$$
is the identity whearas the the composition 
$$(p\downarrow C) \stackrel{p}{\to} (\C\downarrow C) \stackrel{s_0}{\to} (p\downarrow C)$$
admits a natural transformation to the identity.
Since $(\C\downarrow C)$ has a final object, this category and hence $(p\downarrow C)$ are contractible.
By Cofinality, the map
$$(p,1)^*:\holim_{\C}F \to \holim_{\C\times [1]}p^*F $$
is a weak equivalence.
Since $(s_i,1)^*(p,1)^*=1$, the two maps $(s_i,1)^*$ are homotopic, $i=0,1$.
In particular, $(f_0,\ffi_0)^* = (s_0,1)^*(f,\ffi)^*$ is homotopic to 
$(f_1,\ffi_1)^* =  (s_1,1)^*(f,\ffi)^*$.
\end{proof}

Most functors we want to take a homotopy limit of factor through the category $\Open_X^{op}$ of open subsets of a space $X$.
Since the category $\Open_X^{op}$ is a poset, this simplifies the treatment, and we introduce the following category $\Cat_X$.
Its objects are pairs $(\C,U)$ where $\C$ is a small category and
 $U:\C \to \Open_X^{op}$ is a functor.
A map $f:(\C,U) \to (\D,V)$ in $\Cat_X$ can be thought of as a ``refinement''.
It is a functor $f:\C \to \D$ such that
$U(C) \subset V(f(C))$ for all $C \in \C$.
Composition in $\Cat_X$ is composition of functors.
A natural transformation $\delta:f\to g$ of maps $f,g:(\C,U) \to (\D,V)$ in $\Cat_X$ is by definition a natural transformation $\delta:f\to g$ of functors $f,g:\C \to \D$.

If $F:\Open^{op}_X \to \sSets$ is a simplicial presheaf on $X$, 
then $\holim \circ F$ defines a functor $\Cat_X^{op} \to \sSets$ 
by $(\holim\circ F) (\C,U)= \holim_{\C}FU$.
A map $f:(\C,U) \to (\D,V)$ in $\Cat_X$ induces a map $(f,\can):(\C,FU) \to (\D,FV)$ in $[\Cat,\sSets]$ where $\can: f^*FV \to FU$ is the restriction map.
Thus, $F$ defines a functor
$$F: \Cat_X \to [\Cat,\sSets]: (\C,U)\mapsto (\C,FU)
$$
which sends natural transformations to natural transformations,
and $\holim \circ F$ is the composition 
\begin{equation}
\label{eqn:CatXtoCatsSets}
\holim\circ F: \Cat_X^{op}\stackrel{F}{\longrightarrow} [\Cat,\sSets]^{op}\stackrel{\holim}{\longrightarrow} \sSets.
\end{equation}
Call two maps $f,g:(\C,U) \to (\D,V)$ in $\Cat_X$ homotopic if there is a zigzag
$f=f_0 \to f_1 \leftarrow f_2 \to \cdots \leftarrow f_n=g$ of natural transformations of maps $(\C,U) \to (\C,V)$ in $\Cat_X$.
By the Extended functoriality for homotopy limits, the functor 
(\ref{eqn:CatXtoCatsSets}) sends homotopic maps to homotopic maps.

The category $\Cat_X$ has a final object, namely $(*,X)$ where $*$ denotes the one-object-one-morphism category, that is, the final object in $\Cat$.
In particular, for any $(\C,U)$ in $\Cat_X$, there is a natural map of simplicial sets
$$F(X) \to \holim_{\C}F(U).$$

If $I$ is a set, we denote by $\P_0(I)$ the category of non-empty subsets $S \subset I$  where we have a unique arrow $S \to S'$ if $S \subset S'$, otherwise there is no arrow.
An open cover $\U=\{U_i \to U\}_{i\in I}$ of some open $U\subset X$ defines a functor $U:\P_0(I) \to \Open_X^{op}: S \mapsto U_S$ where $U_S=\bigcap_{s\in S}U_s$.

\begin{definition}
Let $X$ be a noetherian scheme, and $\U = \{U_i \to U|i\in I\}$ an open cover of some open subset $U\subset X$.
We say that a simplicial presheaf $F:\Open_X^{op}\to \sSets$ has {\em descent for $\U$} if the following canonical map is a weak equivalence
$$F(U) \stackrel{\sim}{\to} \holim_{\emptyset \neq S \subset I}F(U_S).$$
\end{definition}
The following is a version of \cite[Lemma 5.6]{VoevodskyCD}.

\begin{lemma}[Refinement Lemma]
\label{lem:Refinement}
Let $\U=\{U_i \to X\}_{i\in I}$ and $\V = \{V_j \to X \}_{j\in J}$ be open covers of $X$, and assume that $\V$ is a refinement of $\U$, that is, there is a map $f:J \to I$ such that $V_j \subset U_{f(j)}$ for all $j\in J$.
Let $F$ be a simplicial presheaf on $X$.
If $F$ has descent for $\V$ and for $\V \cap U_S = \{V_j\cap U_S \to U_S\}_{j\in J}$ for all $S \subset I$, then $F$ has descent for $\U$.
\end{lemma}

\begin{proof}
Consider the diagram in $\Cat_X$
$$\xymatrix{
(*,X) & (\P_0(I),U) \ar[l]\\
(\P_0(J),V) \ar[u] \ar[ur]^f & (\P_0(I)\times P_0(J),U\cap V) 
\ar[l]^{\hspace{-4ex}p_J} \ar[u]_{p_I}.
}$$
The square and the upper left triangle commute in $\Cat_X$.
We check that the lower right triangle commutes up to homotopy.
To this end, consider the map in $\Cat_X$
$$g:(\P_0(I)\times P_0(J),U\cap V) \to (\P_0(I),U): (S,T)\mapsto S\cup f(T)$$
which is well-defined as $U_S\cap V_T \subset U_{S\cup f(T)}$ in view of the equality  $U_{S\cup f(T)} = U_S\cap U_{f(T)}$ and the inclusion $V_T \subset U_{f(T)}$.
Now, the (unique) natural transformations $p_I \to g$ and $f\circ p_J \to g$ show that the lower right triangle commutes up to homotopy in $\Cat_X$.
Applying the functor $\holim\circ F$ yields a diagram of simplicial sets
\begin{equation}
\label{eqn:lem:refine}
\xymatrix{
F(X) \ar[r] \ar[d] & \holim_{\emptyset \neq S \subset I} F(U_S) \ar[dl] \ar[d]\\
\holim_{\emptyset \neq T \subset J} F(V_T) \ar[r] & \holim_{\underset{\emptyset \neq S \subset I}{\emptyset \neq T \subset J}}F(U_S\cap V_T)}
\end{equation}
in which the outer square and the upper triangle commute and the lower triangle commutes up to homotopy.
The left vertical map is a weak equivalence, by assumption.
By the Fubini Theorem for homotopy limits (\ref{FubiniThm2}), the right vertical map can be identified with the map 
$$\holim_{\emptyset \neq S \subset I} F(U_S) \to \holim_{\emptyset \neq S \subset I} \holim_{\emptyset \neq T \subset J} F(U_S\cap V_T)$$
induced by the maps
$$F(U_S) \to \holim_{\emptyset \neq T \subset J} F(U_S\cap V_T)$$
which are weak equivalence, by assumption.
Hence both vertical maps in diagram (\ref{eqn:lem:refine}) are weak equivalences.
It follows that the diagonal map is a weak equivalence, and hence, so are the horizontal maps.
\end{proof}

\begin{corollary}
\label{cor:trivCover}
Let $F$ be a simplicial presheaf on $X$ and $\{U_i \to U\}_{i\in I}$ an open cover of some open $U\subset X$.
If for some $i\in I$, the map $U_i \to U$ is the identity, then $F$ has descent for the cover $\{U_i \to U\}_{i\in I}$.
\end{corollary}

\begin{proof}
By hypothesis, the cover $\{1:U \to U\}$ refines $\{U_i \to U\}$.
Since $F$ has descent for any cover of the form $\{1:V \to V\}$, the Refinement Lemma \ref{lem:Refinement} implies the result.
\end{proof}

\begin{corollary}
\label{cor:repetition}
Let $F$ be a simplicial presheaf on $X$.
Let $\U$ and $\V$ be open covers of some open $U \subset X$.
If $\V$ is obtained from $\U$ by repeating some open sets, then $F$ has descent for $\U$ if and only if it has descent for $\V$.
\end{corollary}

\begin{proof}
By assumption, $\U$ refines $\V$ and $\V$ refines $\U$.
The result follows from the Refinement Lemma \ref{lem:Refinement} whose hypothesis we check using Corollary \ref{cor:trivCover}.
\end{proof}

\begin{lemma}[Covering Lemma]
\label{lem:Covering}
Lef $F$ be a simplicial presheaf on $X$ and $V\subset X$ some open subset.
Let $\{V_i \to V\}_{i\in I}$ and $\{U_{i,j}\to V_i\}_{j\in J}$ be open covers for $i\in I$.
For a non-empty $S \subset I$ and function $\sigma:S \to J$ write $U_{\sigma}=\bigcap_{i\in S}U_{i,\sigma(i)}$.
Assume that $F$ has descent for 
$\{V_i \to V\}_{i\in I}$ and for $\{U_{\sigma}\to V_S\}_{\sigma:S \to J}$, $\emptyset \neq S \subset I$.
Then $F$ has descent for $\{U_{i,j} \to V\}_{(i,j)\in I\times J}$.
\end{lemma}

\begin{proof}
For two sets $S,J$ write $J^S$ for the set of functions $S\to J$.
As before, for a non-empty $T\subset J^S$, write $U_T$ for $\bigcap_{\sigma\in T}U_{\sigma}$.
By assumption, we have weak equivalences of simplicial sets
$$F(V) \stackrel{\sim}{\to} \holim_{S \in \P_0(I)}F(V_S) \stackrel{\sim}{\to}
\holim_{S \in \P_0(I)}\holim_{T \in \P_0(J^S)}F(U_T).$$
By the Fubini Theorem  for homotopy limits (\ref{FubiniThm1}), the right hand term is $\holim_{\C}f^*FU$ for the functor (map of posets)
$$f:\C = \xymatrix{\P_0(I)\oint \P_0(J^?)} = \{(S,T)|\ S \in \P_0(I),\ T\in \P_0(J^S)\}\longrightarrow \P_0(I\times J)$$
defined by 
$$f(S,T)= \{(i,j)\in I\times J|\ i\in S,\ j\in \{\sigma(i)| \sigma\in T\}\}$$
where $FU=F\circ U$ is the usual functor with
$$U:\P_0(I\times J) \to \Open^{op}_X:R \mapsto U_R=\bigcap_{(i,j)\in R}U_{i,j}.$$

By Cofinality, we are done once we show that the functor $f$ is left cofinal.
Thus, for $R\in \P_0(J^S)$, we have to check that $(f\downarrow R)$ is contractible.
But the category $(f\downarrow R)$, considered as a full subcategory of $\C$, has a final object, namely $(S_R,T_R)$ where 
$$
\renewcommand\arraystretch{1.5}
\begin{array}{rcl}
S_R&=&\{i\in I| \exists j\in J|\ (i,j)\in R\},\\
R_i& =& \{ j\in J|\ (i,j)\in R\},\\
T_R&=&\{\sigma:S_R\to J|\ \sigma(i)\in R_i\}.
\end{array}$$
Therefore, $(f\downarrow R)$ is contractible, and we are done.
\end{proof}

\begin{corollary}
\label{cor:CoverLem}
Let $Y\subset X$ be an open subset of a space $X$.
If a simplicial presheaf $F$ on $X$ has descent for the open covers 
$\{V_j \to V\}_{j\in J}$, $\{V\to Y,\ W \to Y\}$ and $\{V_j\cap W \to V\cap W\}_{j\in J}$, then $F$ has descent for $\{V_j \to Y, W \to Y\}_{j\in J}$.
\end{corollary}

\begin{proof}
In view of the hypothesis and Corollary \ref{cor:repetition} we can apply the Covering Lemma \ref{lem:Covering} to
$I=\{0,1\}$, $V_0=V$, $V_1=W$, $U_{0,j}=V_j$, $U_{1,j}=W$.
Therefore, $F$ has descent for $\{U_{i,j}\to Y\}_{i\in I,\ j\in J}$ which, after omitting repetitions, is $\{V_j \to Y,W\to Y\}_{j\in J}$.
By Corollary \ref{cor:repetition}, we are done.
\end{proof}

Let $X$ be a noetherian scheme and $U\subset X$ an open subscheme.
We will call a finite cover $\{U_i \to U\}_{i\in I}$ of $U$ {\em elementary} if there exists a total order on $I$ such that for all $i\in I$, there are $f,g\in \Gamma(U_{\leq i},O_X)$ such that $(f,g)$ generates the unit ideal in $\Gamma(U_{\leq i},O_X)$ and such that $U_i=(U_{\leq i})_f$ and $U_{<i}=(U_{\leq i})_g$ where
$U_{\leq i}=\bigcup_{j\leq i}U_j$ and $U_{<i}=\bigcup_{j<i}U_j$.
Note that elementary covers are closed under taking base change.
Note also that if $U$ is affine then $U_i$, $U_{<i}$ and $U_{\leq i}$ are also affine.

\begin{lemma}
Let $R$ be a noetherian ring.
Then any open cover of $\Spec R$ can be refined by a finite elementary open cover.
\end{lemma}

\begin{proof}
Since $R$ is noetherian, any cover of $X=\Spec R$ can be refined by a finite cover.
So, it suffices to prove the claim for finite covers $\{U_i\to X\}_{i=1,...,n}$.
We will prove by induction on $n \in \N_{\geq 1}$ that any cover consisting of $n$ open subsets can be refined by an elementary open cover.
If $n=1$ then $U_1=X$ and the cover is already elementary as we can choose $f=1$ and $g=0$.
Assume now that $n\geq 2$.
Let $I$ and $J$ be the vanishing ideals of $X-U_n$ and $X-U_{<n}$.
Since $U_n$ and $U_{<n}$ cover $X$, we have $I+J=R$, and we can choose $f\in I$, $g\in J$ with $f+g=1$.
Then $X_f\subset U_n$ and $X_g\subset U_{<n}$, and $X_f$ and $X_g$ cover $X$.
By induction hypothesis, the cover $\{(U_i)_g \to (U_{<n})_g=X_g\}_{i=1,...,n-1}$ can be refined by an elementary cover $\{V_i \to X_g\}_{i\in I}$.
Then $\{V_i \to X,\ X_f \to X\}_{i\in I}$ is an elementary cover of $X$ which refines $\{U_i \to X\}_{i=1,...,n}$.
\end{proof}

\begin{lemma}
\label{lem:ElementaryCoverDescent}
Let $X$ be a noetherian scheme and $F$ a simplicial presheaf on $X$ which has the affine B.G.-property.
Then for every open affine $U\subset X$, the simplicial presheaf $F$ has descent for all elementary open covers of $U$.
\end{lemma}

\begin{proof}
We will prove the claim by induction on the cardinality of an elementary open cover of an affine open subset of $X$.
If $n\leq 2$, the claim follows from the definition of the affine B.G.-property.
Now assume $n\geq 3$.
Let $\{U_i \to U\}_{i=1,...,n}$ be an elementary cover of an open affine $U\subset X$.
By induction hypothesis and the definition of elementary cover, the simplicial presheaf $F$ has descent for the covers
$\{U_i \to U_{<n}\}_{i=1,...,n-1}$, $\{U_i\cap U_n \to U_{<n}\cap U_n\}_{i=1,...,n-1}$ and $\{U_n \to U,\ U_{<n}\to U\}$.
By Corollary \ref{cor:CoverLem}, the simplicial presheaf $F$ has descent for $\{U_i \to U\}_{i=1,...,n}$.
\end{proof}

\begin{lemma}
\label{lem:AffineCoverDescent}
Let $X$ be a noetherian scheme and $F$ a simplicial presheaf on $X$ which has the affine B.G.-property.
Then for every open affine $U\subset X$, the simplicial presheaf $F$ has descent for all open affine covers of $U$.
\end{lemma}

\begin{proof}
Note that an elementary open cover of an affine scheme is an affine cover.
Now the claim follows from Lemma \ref{lem:ElementaryCoverDescent} and the Refinement Lemma \ref{lem:Refinement} which we can apply since elementary covers are closed under base change, and every intersection of affine open subsets in $U$ is affine.
\end{proof}

Denote by $\Sch$  a small full subcategory of the category of schemes closed under taking open subschemes and fibre products.
For instance, $\Sch$ could be the category of open subsets of a given scheme, the category of finite type $S$-scheme, or smooth $S$-scheme for a noetherian scheme $S$.
Let $\Aff\subset \Sch$ be the full subcategory of affine schemes.
For a simplicial presheaf $F$ on $\Sch$,
its homotopy right Kan extension 
from $\Aff$ to $\Sch$ is the simplicial presheaf $\hat{F}$ on $\Sch$ defined by
\begin{equation}
\label{eqn:htpyKanExtn}
\hat{F}(X)=\holim_{U\in (\Aff\downarrow X)}F(U).
\end{equation}
The canonical map $((\Aff\downarrow X), F) \to (*,F(X))$ in $[\Cat,\sSets]$ induces a map of simplicial presheaves
$$F \to \hat{F}.$$
When $U\in \Sch$ is affine then this map induces a weak equivalence of simplicial sets $F(U) \stackrel{\sim}{\to}\hat{F}(U)$ since then $(\Aff\downarrow U)$ has a final object.

\begin{lemma}
\label{lem:AmpleFamCover}
Let $X\in \Sch$ be a scheme, let $L_i$ be line bundles on $X$, and let $f_i\in \Gamma(X,L_i)$ be global sections of $L_i$, $i=1,...,n$.
Assume that $X = \bigcup_{i\in I}X_{f_i}$.
If $F$ is a simplicial presheaf on $\Sch$ which has the affine B.G.-property,
then its homotopy right Kan extension $\hat{F}$ has descent for the open cover $\{X_{f_i}\to X\}_{i\in I}$.
\end{lemma}

\begin{proof}
Let $y:Y\to X$ be an affine map of schemes.
Consider the functor 
$$f: (\Aff\downarrow X) \to (\Aff\downarrow Y):(V\to X) \mapsto y^*V=(V\times_XY \to Y).$$
The induced functor on opposite categories $f^{op}$ is left cofinal because for every $w:W\to Y$ in $(\Aff\downarrow Y)$, the category $(w\downarrow f^{op})^{op}=(w\downarrow f)$ has an initial object given by $yw:W \to X$ and $(1,yw):W \to W\times_XY$.
For $\emptyset \neq S\subset I$ and $Y=U_S = \bigcap_{i\in S}X_{f_i} \to X$ the open inclusion, Cofinality for homotopy limits then yields a weak equivalence of simplicial sets
$$\holim_{W\in (\Aff\downarrow U_S)}F(W) \stackrel{\sim}{\longrightarrow} 
\holim_{V\in (\Aff\downarrow X)}F(V\times_XU_S).$$
Taking homotopy limit over $\P_0(I)$, we obtain from the Homotopy Lemma the 
weak equivalence of simplicial sets
$$\holim_{S\in \P_0(I)}\holim_{W\in (\Aff\downarrow U_S)}F(W) \stackrel{\sim}{\longrightarrow} 
\holim_{S\in \P_0(I)}\holim_{V\in (\Aff\downarrow X)}F(V\times_XU_S).$$
The left hand side is 
$$\holim_{S\in \P_0(I)}\hat{F}(U_S)$$
and the right hand side is
$$\holim_{V\in (\Aff\downarrow X)}\holim_{S\in \P_0(I)}F(V\times_XU_S) = 
\holim_{V\in (\Aff\downarrow X)}F(V) = \hat{F}(X)$$
since $F$ has descent for open covers of $V$, by Lemma \ref{lem:AffineCoverDescent}.
Thus, we have a sequence of maps in which the second map is a weak equivalence of simplicial sets
$$\hat{F}(X) \longrightarrow \holim_{S\in \P_0(I)}\hat{F}(U_S) \stackrel{\sim}{\longrightarrow} \hat{F}(X).$$
We are done once we show that the composition is homotopic to the identity.
The existence of the homotopy follows from the Extended Functoriality for homotopy limits 
since the following diagram in $[\Cat,\sSets]$ commutes up to natural transformation
$$\xymatrix{
(\P_0(I)\times (\Aff\downarrow X), f^*F ) \ar[r]^f \ar[dr]
& (\P_0(I)\oint (\Aff\downarrow U_?), F ) \ar[d] \\
& ((\Aff\downarrow X),F)
}
$$
where the horizontal functor is $(S,V\to X)\mapsto (S,V\times_XU_S\to U_S)$, the diagonal functor is $(S,V\to X)\mapsto (V\to X)$ and the vertical functor is
$(S,W \to U_S)\mapsto (W \to X)$ using the inclusion $U_S\subset X$.
The functor $F$ sends $(S,W\to U_S)$ and $(W\to X)$ to $F(W)$.
The natural transformation at $(S,V\to X)$ is the projection map $V\times_XU_S \to V$.
\end{proof}

Recall that a quasi-compact scheme $X$ admits an ample family of line bundles if the open subsets $X_{f}$ form a basis for the Zariski topology on $X$ where $f\in \Gamma(X,L)$ and $L$ runs through all line bundles on $X$.
For instance, any quasi-affine scheme has an ample family of line bundles.

\begin{theorem}
\label{thm:KanExtDescent}
Let $\Sch$ be a small category of noetherian schemes closed under open immersions and fibre products.
Let $F$ be a simplicial presheaf on $\Sch$ which has the affine B.G.-property.
Let $\hat{F}$ be the homotopy right Kan extension of $F$ from the full subcategory $\Aff$ of affine schemes to $\Sch$; see (\ref{eqn:htpyKanExtn}).
Let $X\in \Sch$ be a noetherian scheme with an ample family of line bundles.
Then $\hat{F}$ has descent for all open covers of $X$.
\end{theorem}

\begin{proof}
Since $X$ has an ample family of line bundles, every open cover $\U$ of $X$ can be refined by a cover as in \ref{lem:AmpleFamCover}.
Since those covers are closed under base change, Lemma \ref{lem:AmpleFamCover} together with the Refinement Lemma \ref{lem:Refinement} implies that $\hat{F}$ has descent for $\U$.
\end{proof}

\begin{proof}[Proof of Theorem \ref{thm:AffineBGToBG}]
For an open inclusion $j:Y\subset X$, the natural map of simplicial presheaves
$(j^*F)_{\Zar} \to j^*(F_{\Zar})$ is a map of fibrant objects and a weak equivalence for the Zariski topology, hence it is an object-wise weak equivalence.
Therefore, we can replace $F$ with $j^*F$ and
assume that $X=Y$ is affine.
Then $X$ and all its open subsets have an ample family of line bundles.
Let $U,V\subset X$ be open subsets.
By Theorem \ref{thm:KanExtDescent} with $\Sch = \Open_X$ and $U\cup V$ in place of $X$, the simplicial presheaf $\hat{F}$ has descent for the cover $\{U\to U\cup V, V\to U\cup V\}$ of $U\cup V$.
That is, $\hat{F}$ sends the square (\ref{eqn:MV}) to a homotopy cartesian square of simplicial sets.
By \cite[Theorem 4]{BrownGersten}, the map
$\hat{F} \to \hat{F}_{\Zar}$ from $\hat{F}$ to its Zariski fibrant replacement is an object-wise weak equivalence.
Since $F\to \hat{F}$ is an equivalence on affine schemes, the composition
$F \to \hat{F}_{\Zar}$ is an equivalence on affine schemes and a Zariski weak equivalences to a fibrant simplicial presheaf.
\end{proof}


\begin{thebibliography}{OVV07}

\bibitem[Bas64]{Bass:StableK}
H.~Bass.
\newblock {$K$}-theory and stable algebra.
\newblock {\em Inst. Hautes \'Etudes Sci. Publ. Math.}, (22):5--60, 1964.

\bibitem[Bas73]{BassConj}
Hyman Bass.
\newblock Some problems in ``classical'' algebraic {$K$}-theory.
\newblock In {\em Algebraic {$K$}-theory, {II}: ``{C}lassical'' algebraic
  {$K$}-theory and connections with arithmetic ({P}roc. {C}onf., {B}attelle
  {M}emorial {I}nst., {S}eattle, {W}ash., 1972)}, pages 3--73. Lecture Notes in
  Math., Vol. 342. Springer, Berlin, 1973.

\bibitem[BF78]{BousfieldFriedlander}
A.~K. Bousfield and E.~M. Friedlander.
\newblock Homotopy theory of {$\Gamma $}-spaces, spectra, and bisimplicial
  sets.
\newblock In {\em Geometric applications of homotopy theory ({P}roc. {C}onf.,
  {E}vanston, {I}ll., 1977), {II}}, volume 658 of {\em Lecture Notes in Math.},
  pages 80--130. Springer, Berlin, 1978.

\bibitem[BG73]{BrownGersten}
Kenneth~S. Brown and Stephen~M. Gersten.
\newblock Algebraic {$K$}-theory as generalized sheaf cohomology.
\newblock In {\em Algebraic {K}-theory, {I}: {H}igher {K}-theories ({P}roc.
  {C}onf., {B}attelle {M}emorial {I}nst., {S}eattle, {W}ash., 1972)}, pages
  266--292. Lecture Notes in Math., Vol. 341. Springer, Berlin, 1973.

\bibitem[BK72]{BousfieldKan}
A.~K. Bousfield and D.~M. Kan.
\newblock {\em Homotopy limits, completions and localizations}.
\newblock Lecture Notes in Mathematics, Vol. 304. Springer-Verlag, Berlin-New
  York, 1972.

\bibitem[BM00]{BargeMorel}
Jean Barge and Fabien Morel.
\newblock Groupe de {C}how des cycles orient\'es et classe d'{E}uler des
  fibr\'es vectoriels.
\newblock {\em C. R. Acad. Sci. Paris S\'er. I Math.}, 330(4):287--290, 2000.

\bibitem[Bro82]{brown:book}
Kenneth~S. Brown.
\newblock {\em Cohomology of groups}, volume~87 of {\em Graduate Texts in
  Mathematics}.
\newblock Springer-Verlag, New York-Berlin, 1982.

\bibitem[BS98]{Bhatwadekar:Smooth}
S.~M. Bhatwadekar and Raja Sridharan.
\newblock Projective generation of curves in polynomial extensions of an affine
  domain and a question of {N}ori.
\newblock {\em Invent. Math.}, 133(1):161--192, 1998.

\bibitem[BS00]{Bhatwadekar:ContainsQ}
S.~M. Bhatwadekar and Raja Sridharan.
\newblock The {E}uler class group of a {N}oetherian ring.
\newblock {\em Compositio Math.}, 122(2):183--222, 2000.

\bibitem[CS02]{SchererEtAl}
Wojciech Chach{\'o}lski and J{\'e}r{\^o}me Scherer.
\newblock Homotopy theory of diagrams.
\newblock {\em Mem. Amer. Math. Soc.}, 155(736):x+90, 2002.

\bibitem[Fas08]{Fasel:ChowWitt}
Jean Fasel.
\newblock Groupes de {C}how-{W}itt.
\newblock {\em M\'em. Soc. Math. Fr. (N.S.)}, (113):viii+197, 2008.

\bibitem[GSZ15]{GilleEtAl}
S.~Gille, S.~Scully, and C.~Zhong.
\newblock Milnor-{W}itt {$K$}-groups of local rings.
\newblock arXiv:1501.07631, 2015.

\bibitem[Hoy15]{Hoyois}
Marc Hoyois.
\newblock {\it Personal communication}, 2015.

\bibitem[HT10]{HutchinsonTao:Stability}
Kevin Hutchinson and Liqun Tao.
\newblock Homology stability for the special linear group of a field and
  {M}ilnor-{W}itt {$K$}-theory.
\newblock {\em Doc. Math.}, (Extra volume: Andrei A. Suslin sixtieth
  birthday):267--315, 2010.

\bibitem[HT13]{HutchinsonTao:AugmIdeal}
Kevin Hutchinson and Liqun Tao.
\newblock Milnor-{W}itt {$K$}-theory and tensor powers of the augmentation
  ideal of {$\Bbb Z[F^\times]$}.
\newblock {\em Algebra Colloq.}, 20(3):515--522, 2013.

\bibitem[Hut90]{Hutchinson:Matsumoto}
Kevin Hutchinson.
\newblock A new approach to {M}atsumoto's theorem.
\newblock {\em $K$-Theory}, 4(2):181--200, 1990.

\bibitem[Ker09]{Kerz}
Moritz Kerz.
\newblock The {G}ersten conjecture for {M}ilnor {$K$}-theory.
\newblock {\em Invent. Math.}, 175(1):1--33, 2009.

\bibitem[Lin82]{Lindel}
Hartmut Lindel.
\newblock On the {B}ass-{Q}uillen conjecture concerning projective modules over
  polynomial rings.
\newblock {\em Invent. Math.}, 65(2):319--323, 1981/82.

\bibitem[Mat69]{matsumoto}
Hideya Matsumoto.
\newblock Sur les sous-groupes arithm\'etiques des groupes semi-simples
  d\'eploy\'es.
\newblock {\em Ann. Sci. \'Ecole Norm. Sup. (4)}, 2:1--62, 1969.

\bibitem[May67]{may:book}
J.~Peter May.
\newblock {\em Simplicial objects in algebraic topology}.
\newblock Van Nostrand Mathematical Studies, No. 11. D. Van Nostrand Co., Inc.,
  Princeton, N.J.-Toronto, Ont.-London, 1967.

\bibitem[Maz05]{Mazzoleni}
A.~Mazzoleni.
\newblock A new proof of a theorem of {S}uslin.
\newblock {\em $K$-Theory}, 35(3-4):199--211 (2006), 2005.

\bibitem[Mil70]{milnor:KMpaper}
John Milnor.
\newblock Algebraic {$K$}-theory and quadratic forms.
\newblock {\em Invent. Math.}, 9:318--344, 1969/1970.

\bibitem[Moo68]{moore}
Calvin~C. Moore.
\newblock Group extensions of {$p$}-adic and adelic linear groups.
\newblock {\em Inst. Hautes \'Etudes Sci. Publ. Math.}, (35):157--222, 1968.

\bibitem[Mor04]{morel:puissances}
Fabien Morel.
\newblock Sur les puissances de l'id\'eal fondamental de l'anneau de {W}itt.
\newblock {\em Comment. Math. Helv.}, 79(4):689--703, 2004.

\bibitem[Mor12]{morel:book}
Fabien Morel.
\newblock {\em {$\Bbb A^1$}-algebraic topology over a field}, volume 2052 of
  {\em Lecture Notes in Mathematics}.
\newblock Springer, Heidelberg, 2012.

\bibitem[Mur94]{MurthyChernClass}
M.~Pavaman Murthy.
\newblock Zero cycles and projective modules.
\newblock {\em Ann. of Math. (2)}, 140(2):405--434, 1994.

\bibitem[MV99]{morelvoev}
Fabien Morel and Vladimir Voevodsky.
\newblock {${\bf A}^1$}-homotopy theory of schemes.
\newblock {\em Inst. Hautes \'Etudes Sci. Publ. Math.}, (90):45--143 (2001),
  1999.

\bibitem[NS89]{SuslinNesterenko}
Yu.~P. Nesterenko and A.~A. Suslin.
\newblock Homology of the general linear group over a local ring, and
  {M}ilnor's {$K$}-theory.
\newblock {\em Izv. Akad. Nauk SSSR Ser. Mat.}, 53(1):121--146, 1989.

\bibitem[OVV07]{voevodskyCollab}
D.~Orlov, A.~Vishik, and V.~Voevodsky.
\newblock An exact sequence for {$K^M_\ast/2$} with applications to quadratic
  forms.
\newblock {\em Ann. of Math. (2)}, 165(1):1--13, 2007.

\bibitem[Qui73]{quillen:higherI}
Daniel Quillen.
\newblock Higher algebraic {$K$}-theory. {I}.
\newblock In {\em Algebraic {$K$}-theory, {I}: {H}igher {$K$}-theories ({P}roc.
  {C}onf., {B}attelle {M}emorial {I}nst., {S}eattle, {W}ash., 1972)}, pages
  85--147. Lecture Notes in Math., Vol. 341. Springer, Berlin, 1973.

\bibitem[Qui76]{quillen:Serre}
Daniel Quillen.
\newblock Projective modules over polynomial rings.
\newblock {\em Invent. Math.}, 36:167--171, 1976.

\bibitem[Sah89]{Sah}
Chih-Han Sah.
\newblock Homology of classical {L}ie groups made discrete. {III}.
\newblock {\em J. Pure Appl. Algebra}, 56(3):269--312, 1989.

\bibitem[SS03]{ShipSchwede:Equivalences}
Stefan Schwede and Brooke Shipley.
\newblock Equivalences of monoidal model categories.
\newblock {\em Algebr. Geom. Topol.}, 3:287--334 (electronic), 2003.

\bibitem[Sus82]{Suslin:KStability}
A.~A. Suslin.
\newblock Stability in algebraic {$K$}-theory.
\newblock In {\em Algebraic {$K$}-theory, {P}art {I} ({O}berwolfach, 1980)},
  volume 966 of {\em Lecture Notes in Math.}, pages 304--333. Springer, Berlin,
  1982.

\bibitem[Vas69]{Vaserstein:Stabilization}
L.~N. Vaser{\v{s}}te{\u\i}n.
\newblock On the stabilization of the general linear group over a ring.
\newblock {\em Math. USSR-Sb.}, 8:383--400, 1969.

\bibitem[Vas71]{dfn:stableRank}
L.~N. Vaser{\v{s}}te{\u\i}n.
\newblock The stable range of rings and the dimension of topological spaces.
\newblock {\em Funkcional. Anal. i Prilo\v zen.}, 5(2):17--27, 1971.

\bibitem[vdK76]{VanDerKallenCounterEx}
Wilberd van~der Kallen.
\newblock Injective stability for {$K_{2}$}.
\newblock In {\em Algebraic {$K$}-theory ({P}roc. {C}onf., {N}orthwestern
  {U}niv., {E}vanston, {I}ll., 1976)}, pages 77--154. Lecture Notes in Math.,
  Vol. 551. Springer, Berlin, 1976.

\bibitem[vdK77]{vdKH2SL2}
Wilberd van~der Kallen.
\newblock The {$K_{2}$} of rings with many units.
\newblock {\em Ann. Sci. \'Ecole Norm. Sup. (4)}, 10(4):473--515, 1977.

\bibitem[vdK80]{VanDerKallen:Invent}
Wilberd van~der Kallen.
\newblock Homology stability for linear groups.
\newblock {\em Invent. Math.}, 60(3):269--295, 1980.

\bibitem[Voe10]{VoevodskyCD}
Vladimir Voevodsky.
\newblock Homotopy theory of simplicial sheaves in completely decomposable
  topologies.
\newblock {\em J. Pure Appl. Algebra}, 214(8):1384--1398, 2010.

\end{thebibliography}

\end{document}